\documentclass[10pt]{amsart}
\usepackage[foot]{amsaddr}
\usepackage{hyperref}
\usepackage{fixme}
\usepackage{mathrsfs}

\usepackage[a4paper]{geometry}
\usepackage{amssymb, amsthm, amsfonts, amsxtra, amsmath}
\usepackage{latexsym}
\usepackage{mathabx}
\usepackage{enumitem}
\usepackage[all]{xy}
\usepackage{graphics}
\usepackage{pdfpages}
\usepackage{epic}
\usepackage{fouridx}
\usepackage{parskip} 
\usepackage{bbm}
\makeatletter 
\def\thm@space@setup{%
  \thm@preskip=\parskip \thm@postskip=0pt
}
\makeatother

\DeclareMathOperator{\fin}{\mathrm{f}}

\DeclareMathOperator{\id}{id}
\DeclareMathOperator{\img}{img}

\DeclareMathOperator{\Hom}{Hom}

\DeclareMathOperator{\Nat}{\mathrm{Nat}}
\DeclareMathOperator{\op}{\mathrm{op}}

\DeclareMathOperator{\rcf}{\mathrm{rcfd}}

\DeclareMathOperator{\sgn}{\mathrm{sgn}}
\DeclareMathOperator{\Span}{\mathrm{span}}

\DeclareMathOperator{\tr}{\mathrm{tr}}

\newcommand{\dual}[1]{#1^{*}}
\newcommand{\dualop}[1]{#1^{\tr}}
\newcommand{\dualco}[1]{\hat{#1}}
\newcommand{\dualcor}[1]{\check{#1}}

\newcommand{\Corep}{\mathrm{Corep}}

\newcommand{\Circt}{{\mathop{\ooalign{$\ovoid$\cr\hidewidth\raise-.05ex\hbox{$\scriptstyle\mathsf T\mkern3.5mu$}\cr}}}} 
\newcommand{\Circtv}[1]{\underset{#1}{\mathop{\ooalign{$\ovoid$\cr\hidewidth\raise-.05ex\hbox{$\scriptstyle\mathsf T\mkern3.5mu$}\cr}}}} 
\newcommand{\smCirct}{\mathop{\ooalign{$\scriptstyle\ovoid$\cr\hidewidth\raise-.05ex\hbox{$\scriptscriptstyle\mathsf T\mkern2.8mu$}\cr}}}  

\newcommand{\n}{\mathfrak{n}}

\newcommand{\G}{\mathbb{G}}

\newcommand{\C}{\mathbb{C}}
\newcommand{\R}{\mathbb{R}}
\newcommand{\Z}{\mathbb{Z}}
\newcommand{\N}{\mathbb{N}}
\newcommand{\X}{\mathbb{X}}

\newcommand{\Hsp}{\mathcal{H}}

\newcommand{\Tr}{\mathrm{Tr}}

\newcommand{\CatC}{\mathcal{C}}
\newcommand{\CatD}{\mathcal{D}}
\newcommand{\CatCC}{\mathscr{C}}
\newcommand{\CatDD}{\mathscr{D}}

\newcommand{\Vectif}{\Gr{\mathrm{Vect}}{I}{I}{}{\fin}}
\newcommand{\Vectrcf}{\Gr{\mathrm{Vect}}{I}{I}{}{\rcf}}
\newcommand{\Hilb}{\mathrm{Hilb}}

\newcommand{\Hilbif}{\Gr{\mathrm{Hilb}}{I}{I}{}{\fin}}

\newcommand{\Mor}{\mathrm{Mor}}

\newcommand{\ev}{\mathrm{ev}}
\newcommand{\coev}{\mathrm{coev}}

\newcommand{\Fun}{\mathrm{Fun}}

\newcommand{\itimes}{\underset{I}{\otimes}}
\newcommand{\osum}[1]{\underset{#1}{\sum}^{\oplus}}

\newcommand{\Rep}{\mathrm{Rep}}

\newcommand{\fd}{\mathrm{fd}}

\newcommand{\Vect}{\mathrm{Vect}}

\newcommand{\Grs}[3]{#1{\begin{pmatrix} #2,  #3\end{pmatrix}}}

\newcommand{\GrDA}[3]{{}_{\;#2}#1_{#3}} 

\newcommand{\Grt}[3]{#1{\tiny{\begin{pmatrix} #2\\#3\end{pmatrix}}}} 

\newcommand{\GrRA}[3]{#1^{#2}_{#3}} 

\newcommand{\Unit}{\mathbf{1}}
\newcommand{\Unitb}{\mathbbm{1}}
\newcommand{\UnitC}[2]{\Grt{\mathbf{1}}{#1}{#2}} 
\newcommand{\Grru}[2]{{\tiny \begin{pmatrix} #1 \\ #2\end{pmatrix}}}

\newcommand{\eGr}[5]{#1{{\tiny \begin{pmatrix} #2 \quad #3 \\ #4 \quad #5\end{pmatrix}}}}

\newcommand{\pmat}[4]{{\tiny \begin{pmatrix} #1 \quad #2 \\ #3 \quad #4\end{pmatrix}}}

\newcommand{\Gr}[5]{\fourIdx{#2}{#4}{#3}{#5}{#1}}
\newcommand{\Gru}[3]{\Gr{#1}{}{}{#2}{#3}}
\newcommand{\Grd}[3]{\Gr{#1}{#2}{#3}{}{}}

\newcommand{\wmult}{\cdot}
\newcommand{\bmult}{*}

\newcommand{\aste}[1]{\underset{#1}{\ast}}

\newtheorem{Theorem}{Theorem}[section]
\newtheorem{Lem}[Theorem]{Lemma}
\newtheorem{Prop}[Theorem]{Proposition}
\newtheorem{Cor}[Theorem]{Corollary}

\theoremstyle{definition}
\newtheorem{Def}[Theorem]{Definition}
\newtheorem{Rem}[Theorem]{Remark}
\newtheorem{Exa}[Theorem]{Example}
\newtheorem{Not}[Theorem]{Notation}

\newcommand{\phic}[2]{\Grt{\phi}{#1}{#2}}

\date{}

\numberwithin{equation}{section}

\begin{document}
\title{Partial compact quantum groups}

\author{Kenny De Commer$^{1}$}\address{$^{1}$Department of Mathematics, Vrije
    Universiteit Brussel, VUB, B-1050 Brussels, Belgium}\email{kenny.de.commer@vub.ac.be}
\author{Thomas Timmermann$^{2}$}\address{$^{2}$University of M\"{u}nster,
  Einsteinstrasse 62, 48149 Muenster}\email{timmermt@uni-muenster.de}
\thanks{second author supported by the SFB 878 of the DFG}

\begin{abstract}
\noindent Compact quantum groups of face type, as introduced by
Hayashi, form a class of compact quantum groupoids with a classical,
finite set of objects. Using the notions of a weak multiplier bialgebra
and weak multiplier Hopf algebra (resp.~ due to
B{\"o}hm--G\'{o}mez-Torrecillas--L\'{o}pez-Centella and Van
Daele--Wang), we generalize Hayashi's definition to allow for an
infinite set of objects, and call the resulting objects partial
compact quantum groups. We prove a Tannaka-Kre$\breve{\textrm{\i}}$n-Woronowicz
reconstruction result for such partial compact quantum groups using
the notion of a partial fusion C$^*$-category. As examples, we consider the dynamical quantum $SU(2)$-groups from the point of view of partial compact quantum groups.
\end{abstract}



\maketitle


\tableofcontents

\section*{Introduction}

The concept of  a \emph{face algebra} was introduced by T. Hayashi in \cite{Hay2}, motivated by the theory of solvable lattice models in statistical mechanics. It was further studied in \cite{Hay1,Hay3,Hay4,Hay5,Hay6,Hay7,Hay8}, where for example associated $^*$-structures and a canonical Tannaka duality were developed. This canonical Tannaka duality allows one to construct a canonical face algebra from any (finite) fusion category. For example, a face algebra can be associated to the fusion category of a quantum group at root unity, for which no genuine quantum group implementation can be found. 

In \cite{Nil1,Sch1,Sch2}, it was shown that face algebras are particular kinds of $\times_R$-algebras \cite{Tak2} and of weak bialgebras \cite{Boh3,BCJ,Nik1}. More intuitively, they can be considered as quantum groupoids with a classical, finite object set. In this article, we want to extend Hayashi's theory by allowing an \emph{infinite} (but still discrete) object set. This requires passing from weak bialgebras to weak \emph{multiplier} bialgebras \cite{Boh1}. At the same time, our structures admit a piecewise description by what we call a \emph{partial bialgebra}, which is more in the spirit of Hayashi's original definition. In the presence of an antipode, an invariant integral and a compatible $^*$-structure, we call our structures \emph{partial compact quantum groups}. 

The passage to the infinite object case is delicate at points, and requires imposing the proper finiteness conditions on associated structures. However, once all conditions are in place, many of the proofs are similar in spirit to the finite object case. 

Our main result is a Tannaka-Kre$\breve{\textrm{\i}}$n-Woronowicz duality result which states that partial compact quantum groups are in one-to-one correspondence with \emph{concrete partial fusion C$^*$-categories}. In essence, a partial fusion C$^*$-category is a multifusion C$^*$-category \cite{ENO1}, except that (in a slight abuse of terminology) we allow an infinite number of irreducible objects as well as an infinite number of summands inside the unit object. By a \emph{concrete} multifusion C$^*$-category, we mean a multifusion C$^*$-category realized inside a category of (locally finite-dimensional) bigraded Hilbert spaces. Of course, Tannaka reconstruction is by now a standard procedure. For closely related results most relevant to our work, we mention \cite{Wor2,Sch3,Hay8,Ost1,Hai1,Szl1,Pfe1,DCY1,Nes1} as well as the surveys \cite{JoS1} and \cite[Section 2.3]{NeT1}.

As an application, we generalize Hayashi's Tannaka duality \cite{Hay8} (see also \cite{Ost1}) by showing that any module C$^*$-category over a multifusion C$^*$-category has an associated canonical partial compact quantum group. By the results of \cite{DCY1}, such data can be produced from ergodic actions of compact quantum groups. In particular,  we consider the case of ergodic actions of $SU_q(2)$ for $q$ a non-zero real. This will allow us to show that the construction of \cite{Hay4} generalizes to produce partial compact quantum group versions of the dynamical quantum $SU(2)$-groups of \cite{EtV1,KoR1}, see also \cite{Sto1} and references therein. This construction will immediately provide the right setting for the operator algebraic versions of these dynamical quantum $SU(2)$-groups, which was the main motivation for writing this paper. These operator algebraic details will be studied elsewhere \cite{DCT2}.

The precise layout of the paper is as follows.

The first two sections introduce the basic theory of the structures
which we will be concerned with in this paper. In the \emph{first
  section}, we introduce the notions of a \emph{partial bialgebra},
\emph{partial Hopf algebra} and \emph{partial compact quantum group},
and show how they are related to the notion of a weak multiplier
bialgebra \cite{Boh1}, weak multiplier Hopf algebra \cite{VDW1,VDW2}
and compact quantum group of face type \cite{Hay1}. In the
\emph{second section}, we introduce the corresponding notions of a \emph{partial tensor category} and \emph{partial fusion C$^*$-category}. 

In the next two sections, our main result is proven, namely the Tannaka-Kre$\breve{\textrm{\i}}$n-Woronowicz duality. In the \emph{third section} we develop the corepresentation theory of partial Hopf algebras and the representation theory of partial compact quantum groups, and we show that the latter allows one to construct a concrete partial fusion C$^*$-category. In the \emph{fourth} section, we show conversely how any concrete partial fusion C$^*$-category allows one to construct a partial compact quantum group, and we briefly show how the two constructions are inverses of each other.

In the final two sections, we provide some examples of our structures
and applications of our main result. In the \emph{fifth section}, we
first consider the construction of a canonical partial compact quantum
group from any partial module C$^*$-category for a partial fusion
C$^*$-category. We then introduce the notions of  \emph{Morita}, \emph{co-Morita} and \emph{weak Morita equivalence} \cite{Mug1} of partial compact quantum groups, and show that two partial compact quantum groups are weakly Morita equivalent if and only if they can be connected by a string of Morita and co-Morita equivalences. In the \emph{sixth section}, we study in more detail a concrete example of a canonical partial compact quantum group, constructed from an ergodic action of quantum $SU(2)$. In particular, we obtain a partial compact quantum group version of the dynamical quantum $SU(2)$-group. 

\emph{Note}: we follow the physicist's convention that inner products on Hilbert spaces are anti-linear in their \emph{first} argument.

\section{Partial compact quantum groups}

We generalize Hayashi's definition of a compact quantum group of face type \cite{Hay1} to the case where the commutative base algebra is no longer finite-dimensional. We will present two approaches, based on \emph{partial bialgebras} and \emph{weak multiplier bialgebras} \cite{Boh1,VDW1}. The first approach is piecewise and concrete, but requires some bookkeeping. The second approach is global but more abstract. As we will see from the general theory and the concrete examples, both approaches have their intrinsic value.

\subsection{Partial algebras}

Let $I$ be a set. We consider $I^2=I\times I$ as the pair groupoid with $\wmult$ denoting composition. That is, an element $K=(k,l)\in I^2$ has source $K_l = k$ and target $K_r=l$, and if $K=(k,l)$ and $L=(l,m)$ we write $K\wmult L = (k,m)$. 

\begin{Def} A \emph{partial algebra} $\mathscr{A}=(\mathscr{A},M)$ (over $\C$) is a set $I$ (the \emph{object} set) together with 
\begin{itemize}
\item[$\bullet$] for each $K=(k,l)\in I^2$ a vector space $A(K) = \Grs{A}{k}{l}=\!\!\GrDA{A}{k}{l}$ (possibly the zero vector space),
\item[$\bullet$] for each $K,L$ with $K_r = L_l$ a multiplication map \[M(K,L):A(K) \otimes A(L)\rightarrow A(K\cdot L),\qquad a\otimes b \mapsto ab\]  and 
\item[$\bullet$] elements $\Unit(k) = \Unit_k \in \Grs{A}{k}{k}$ (the units), 
\end{itemize}
such that the obvious associativity and unit conditions are satisfied. 

By an \emph{$I$-partial algebra} we will mean a partial algebra with object set $I$.
\end{Def}

\begin{Rem}
\begin{enumerate}[label=(\arabic*)]\item It will be important to allow the local units $\Unit_k$ to be zero.
\item A partial algebra is by definition the same as a small
  $\C$-linear category. However, we do not emphasize this viewpoint,
  as the natural notion of a morphism for partial algebras will be \emph{contravariant} on objects, see Definition \ref{DefMor}.
\end{enumerate}
\end{Rem}

Let $\mathscr{A}$ be an $I$-partial algebra. We define $A(K\wmult L)$ to be $\{0\}$ when $K\wmult L$ is ill-defined, i.e. $K_r\neq L_l$. We then let $\Grs{M}{K}{L}$ be the zero map.

\begin{Def} The \emph{total algebra} $A$ of an $I$-partial algebra $\mathscr{A}$ is the vector space \[A = \bigoplus_{K\in I^2} A(K)\] endowed with the unique multiplication whose restriction to $A(K)\otimes A(L)$ concides with $M(K,L)$. 
\end{Def} 

Clearly $A$ is an associative algebra. It will in general not possess a unit, but it is a \emph{locally unital algebra} as there exist mutually orthogonal idempotents $\mathbf{1}_k$ with $A = \osum{k,l} \mathbf{1}_kA\mathbf{1}_l$. An element $a\in A$ can be interpreted as a function assigning to each element $(k,l)\in I^2$ an element $a_{kl}\in A(k,l)$, namely the $(k,l)$-th component of $a$. This identifies $A$ with finitely supported $I$-indexed matrices whose $(k,l)$-th entry lies in $A(k,l)$, equipped with the natural matrix multiplication. 

\begin{Rem}\label{RemGrad} When $\mathscr{A}$ is an $I$-partial algebra with total algebra $A$, then $A\otimes A$ can be naturally identified with the total algebra of an $I\times I$-partial algebra $\mathscr{A}\otimes \mathscr{A}$, where \[(A\otimes A)((k,k'),(l,l')) = A(k,l)\otimes A(k',l')\] with the obvious tensor product multiplications and the $\Unit_{k,k'} = \Unit_k\otimes \Unit_{k'}$ as units. 
\end{Rem}

Working with non-unital algebras necessitates the use of their \emph{multiplier algebra}. Let us first recall some general notions concerning non-unital algebras from \cite{Dau1,VDae1}.

\begin{Def} Let $A$ be an algebra over $\C$, not necessarily with unit. We call $A$ \emph{non-degenerate} if $A$ is faithfully represented on itself by left and right multiplication. It is called \emph{idempotent} if $A^2 = A$. 
\end{Def}

\begin{Def} Let $A$ be an algebra. A \emph{multiplier} $m$ for $A$ consists of a couple of maps \begin{eqnarray*} L_m:A\rightarrow A,\quad a\mapsto ma\\ R_m:A\rightarrow A,\quad a\mapsto am\end{eqnarray*} such that $(am)b = a(mb)$ for all $a,b\in A$. 

The set of all multipliers forms an algebra under composition for the
$L$-maps and anti-composition for the $R$-maps. It is called the
\emph{multiplier algebra} of $A$, and is denoted by $M(A)$.
\end{Def}

One has a natural homomorphism $A\rightarrow M(A)$. When $A$ is non-degenerate,  this homomorphism is injective, and we can then identify $A$ as a subalgebra of the (unital) algebra $M(A)$. We then also have inclusions \[A\otimes A\subseteq M(A)\otimes M(A)\subseteq M(A\otimes A).\]

\begin{Exa}\label{ExaMult} 
\begin{enumerate}[label=(\arabic*)]
\item Let $I$ be a set, and $\Fun_{\fin}(I)$ the algebra of all finitely supported functions on $I$. Then $M(\Fun_{\fin}(I)) = \Fun(I)$, the algebra of all functions on $I$. 
\item Let $A$ be the total algebra of an $I$-partial algebra $\mathscr{A}$. As $A$ has local units, it is non-degenerate and idempotent. Then one can identify $M(A)$ with \[M(A) = \left(\prod_l \bigoplus_k A(k,l)\right) \bigcap \left(\prod_k\bigoplus_l A(k,l)\right) \subseteq \prod_{k,l} A(k,l),\] i.e.~ with the space of functions \[m:I^2\rightarrow A,\quad m_{kl}\in A(k,l)\] which have finite support in either one of the variables when the other variable has been fixed. The multiplication is given by the formula \[(mn)_{kl} = \sum_p m_{kp}n_{pl}.\]
\item Let $m_i$ be any collection of multipliers of $A$, and assume that for each $a\in A$, $m_ia =0$ for almost all $i$, and similarly $am_i=0$ for almost all $i$. Then one can define a multiplier $\sum_i m_i$ in the obvious way by termwise multiplication. One says that the sum $\sum_i m_i$ converges in the \emph{strict} topology. 
\end{enumerate}
\end{Exa}

The condition appearing in the second example above will appear time and again, so we introduce it formally in the next definition.

\begin{Def} We will call any general assignment $(k,l)\rightarrow m_{kl}$ into a set with a distinguished zero element \emph{row-and column-finite} (rcf) if the assignment has finite support in either one of the variables when the other variable has been fixed. 
\end{Def} 

Let us comment on the notion of a morphism for partial algebras. We first introduce the piecewise definition.

\begin{Def}\label{DefMor} Let $\mathscr{A}$ and $\mathscr{B}$ be respectively $I$ and $J$-partial algebras. Let \[\phi: I \ni k \mapsto J_k \subseteq J\] with the $J_k$ disjoint. A \emph{homomorphism} (based on $\phi$) from $\mathscr{A}$ to $\mathscr{B}$ consists of linear maps \[\GrDA{f}{r}{s}: A(k\;l)\rightarrow B(r\;s),\quad a\mapsto \GrDA{f(a)}{r}{s}\] for all $r\in J_k, s\in J_l$, satisfying 
\begin{enumerate}[label = (\arabic*)]
\item (Unitality) $\GrDA{f(\Unit_{k})}{r}{s} = \delta_{rs}\Unit_r$ for all $r,s\in J_k$.
\item (Local finiteness) For each $k,l\in I$ and $a\in A(k\;l)$, the assigment $(r,s)\rightarrow \GrDA{f(a)}{r}{s}$ on $J_k\times J_l$ is rcf. 
\item (Multiplicativity) For all $k,l,m\in I$, all $r\in J_k$ and all $t\in J_m$, and all $a\in A(k\;l)$ and $b\in A(l\;m)$, one has \[\GrDA{f(ab)}{r}{t} = \sum_{s\in J_l} \GrDA{f(a)}{r}{s}\GrDA{f(b)}{s}{t}.\]
\end{enumerate} 
The homomorphism is called \emph{unital} if $J=\bigcup \{J_k\mid k\in I\}$. 
\end{Def}
\begin{Rem}
\begin{enumerate}[label=(\arabic*)]
\item
Note that the multiplicativity condition makes sense because of the local finiteness condition.
\item
If $J = \bigcup_{k} J_k$, we can interpret $\phi$ as a map \[J\rightarrow I,\quad r\mapsto k \iff r\in J_k.\] In the more general case, we obtain a function $J\rightarrow I^*$, where $I^*$ is $I$ with an extra point `at infinity' added.
\end{enumerate}
\end{Rem}

The following lemma provides the global viewpoint concerning homomorphisms. 

\begin{Lem}\label{LemPAMor} Let $\mathscr{A}$ and $\mathscr{B}$ be respectively $I$- and $J$-partial algebras, and fix an assignment $\phi: k\mapsto J_k$. Then there is a one-to-one correspondence between homomorphisms $\mathscr{A}\rightarrow \mathscr{B}$ based on $\phi$ and homomorphisms $f:A\rightarrow M(B)$ with $f(\Unit_k) = \sum_{r\in J_k} \Unit_r$. 
\end{Lem} 
\begin{proof}
Straightforward, using the characterisation of the multiplier algebra
provided in Remark \ref{ExaMult} (2).
\end{proof}

\subsection{Partial coalgebras}

The notion of  a partial algebra nicely dualizes, one of the main benefits of the local approach. For this we consider again $I^2$ as the pair groupoid, but now with elements considered as column vectors, and with $\bmult$ denoting the (vertical) composition. So $K=\Grt{}{k}{l}$ has source $K_u = k$ and target $K_d = l$, and if $K=\Grt{}{k}{l}$ and $L=\Grt{}{l}{m}$ then $K\bmult L = \Grt{}{k}{m}$.

\begin{Def} A \emph{partial coalgebra} $\mathscr{A}=(\mathscr{A},\Delta)$ (over $\C$) consists of a set $I$ (the object set) together with 
\begin{itemize}
\item[$\bullet$] for each $K=\Grru{k}{l}\in I^2$ a vector space $A(K) = \Grt{A}{k}{l}=\!\!\GrRA{A}{k}{l}$,
\item[$\bullet$] for each $K,L$ with $K_d = L_u$ a comultiplication map \[\Grt{\Delta}{K}{L}:A(K*L)\rightarrow A(K)\otimes A(L),\qquad a \mapsto a_{(1)K}\otimes a_{(2)L},\] and 
\item[$\bullet$] counit maps $\epsilon_k:\Grt{A}{k}{k}\rightarrow \C$,
\end{itemize} 
satisfying the obvious coassociativity and counitality conditions.

By \emph{$I$-partial coalgebra} we will mean a partial coalgebra with object set $I$.
\end{Def}

\begin{Not}\label{NotCom} As the index of $\epsilon_k$ is determined by the element to which it is applied, there is no harm in dropping the index $k$ and simply writing $\epsilon$.

Similarly, if $K = \Grt{}{k}{l}$ and $L = \Grt{}{l}{m}$, we abbreviate $\Delta_l = \Grt{\Delta}{K}{L}$, as the other indices are determined by the element to which $\Delta_l$ is applied.
\end{Not}

We also make again the convention that $A(K*L)=\{0\}$ and
$\Grt{\Delta}{K}{L}$ is the zero map when $K_d \neq L_u$. Similarly $\epsilon$ is seen as the zero functional on $A(K)$ when $K=\Grt{}{k}{l}$ with $k\neq l$. 

\subsection{Partial bialgebras}

We can now superpose the notions of a partial algebra and a partial coalgebra. Let $I$ be a set, and let $M_2(I)$ be the set of 4-tuples of elements of $I$ arranged as 2$\times$2-matrices. We can endow $M_2(I)$ with two compositions, namely $\cdot$ (viewing $M_2(I)$ as a row vector of column vectors) and $*$ (viewing $M_2(I)$ as a column vector of row vectors). When $K\in M_2(I)$, we will write $K = \Grs{}{K_l}{K_r} = \Grt{}{K_u}{K_d} = \eGr{}{K_{lu}}{K_{ru}}{K_{ld}}{K_{rd}}$. One can view $M_2(I)$ as a double groupoid, and in fact as a \emph{vacant} double groupoid in the sense of \cite{AN1}. 

In the following, a vector $(r,s)$ will sometimes be written simply as $r,s$ (without parentheses) or $rs$ in an index. We also follow Notation \ref{NotCom}, but the reader should be aware that the index of $\Delta$ will now be a 1$\times$2 vector in $I^2$ as we will work with partial coalgebras over $I^2$.

\begin{Def}\label{DefPartBiAlg} A \emph{partial bialgebra} $\mathscr{A}=(\mathscr{A},M,\Delta)$ consists of a set $I$ (the \emph{object set}) and a collection of vector spaces $A(K)$ for $K\in M_2(I)$ such that 
\begin{itemize}
\item[$\bullet$] the $\Grs{A}{K_l}{K_r}$ form an $I^2$-partial algebra,
\item[$\bullet$] the $\Grt{A}{K_u}{K_d}$ form an $I^2$-partial coalgebra,
\end{itemize} 
and for which the following compatibility relations are satisfied.
\begin{enumerate}[label=(\arabic*)]
\item\label{Propa} (Comultiplication of Units) For all $k,l,l',m\in I$, one has 
\[\Delta_{l,l'}(\UnitC{k}{m}) = \delta_{l,l'} \UnitC{k}{l}\otimes \UnitC{l}{m}.\]  
\item\label{Propb} (Counit of Multiplication) For all $K,L\in M_2(I)$ with $K_r = L_l$ and all $a\in A(K)$ and $b\in A(L)$, \[\epsilon(ab) = \epsilon(a)\epsilon(b).\]
\item\label{Propc} (Non-degeneracy) For all $k\in I$, $\epsilon(\UnitC{k}{k})=1$. 
\item\label{Propd} (Finiteness) For each $K\in M_2(I)$ and each $a\in A(K)$, the assignment $(r,s)\rightarrow \Delta_{rs}(a)$ is rcf.
\item\label{Prope} (Comultiplication is multiplicative) For all $a\in A(K)$ and $b\in A(L)$ with $K_r= L_l$,  \[\Delta_{rs}(ab) = \sum_t \Delta_{rt}(a)\Delta_{ts}(b).\]
\end{enumerate}
\end{Def}

\begin{Rem}\label{RemBA}\begin{enumerate}[label=(\arabic*)]
\item By assumption \ref{Propd}, the sum on the right hand side in condition \ref{Prope} is in fact finite and hence well-defined. 
\item Note that the object set of the above $\mathscr{A}$ as a partial bialgebra is $I$, but the object set of its underlying partial algebra (or coalgebra) is $I^2$.
\item Properties \ref{Propa}, \ref{Propd} and \ref{Prope} simply say that $\Delta$ is a homomorphism $\mathscr{A}\rightarrow \mathscr{A}\otimes \mathscr{A}$ of partial algebras based over the assignment $I^2\rightarrow \mathscr{P}(I^2\times I^2)$, the power set of $I^2\times I^2$, such that \[(I^2\times I^2)_{{\tiny \begin{pmatrix} k\\m \end{pmatrix}}} = \{\left(\begin{pmatrix} k \\ l \end{pmatrix},\begin{pmatrix} l \\ m \end{pmatrix}\right)\mid l\in I\}.\] 
\end{enumerate}
\end{Rem}

We relate the notion of a partial bialgebra to the recently introduced
notion of a weak multiplier bialgebra \cite{Boh1}. Let us first
introduce the following notation, using the notion introduced in
Example \ref{ExaMult} (2).

\begin{Not}
If $\mathscr{A}$ is an $I$-partial bialgebra, we write \[\lambda_k = \sum_l \UnitC{k}{l},\qquad \rho_l = \sum_k\UnitC{k}{l} \qquad \in M(A).\]
\end{Not}

\begin{Rem} From Property \ref{Propc} of Definition \ref{DefPartBiAlg}, it follows that $\lambda_k\neq 0\neq \rho_k$ for any $k\in I$. 
\end{Rem} 

To show that the total algebra of a partial bialgebra becomes a weak
multiplier bialgebra, we will need some easy lemmas. The first one is
an immediate consequence of Remark \ref{RemBA} (3) and Lemma \ref{LemPAMor}:

\begin{Lem} Let $\mathscr{A}$ be an $I$-partial bialgebra. Then for each $a\in A$, there exists a unique multiplier $\Delta(a) \in M(A\otimes A)$ such that \begin{align}\label{EqDel}
    \begin{aligned}
      \Delta_{rs}(a) &= (1\otimes \lambda_r)\Delta(a)(1\otimes
      \lambda_s) \\ &= (\rho_r\otimes 1)\Delta(a)(\rho_s\otimes 1)
    \end{aligned}
\end{align}  for all $r,s\in I$, all $K\in M_2(I)$ and all $a\in A(K)$. 

The resulting map \[\Delta:A\rightarrow M(A\otimes A),\quad a\mapsto \Delta(a)\] is a homomorphism.
\end{Lem} 

We will refer to $\Delta: A\rightarrow M(A\otimes A)$ as the
\emph{total comultiplication} of $\mathscr{A}$. We will then also use
the suggestive Sweedler notation for this map, \[\Delta(a) =
a_{(1)}\otimes a_{(2)}.\]  Note for example that \[\Delta(\UnitC{k}{m}) = \sum_{l}\UnitC{k}{l}\otimes \UnitC{l}{m} = \sum_l \lambda_k\rho_l\otimes \lambda_l\rho_m.\]

\begin{Lem} The element $E = \sum_{k,l,m} \UnitC{k}{l}\otimes \UnitC{l}{m}= \sum_l \rho_l\otimes \lambda_l$ is a well-defined idempotent in $M(A\otimes A)$, and satisfies \[\Delta(A)(A\otimes A)=E(A\otimes A),\quad (A\otimes A)\Delta(A)= (A\otimes A)E.\]
\end{Lem} 
\begin{proof} Clearly the sum defining $E$ is strictly convergent, and makes $E$ into an idempotent. It is moreover immediate that $E\Delta(a)=\Delta(a) = \Delta(a)E$ for all $a\in A$. Since \[E(\UnitC{k}{l}\otimes \UnitC{m}{n}) = \Delta(\UnitC{k}{n})(\UnitC{k}{l}\otimes \UnitC{m}{n}) \] by the property \ref{Propa} of Definition \ref{DefPartBiAlg}, and analogously for multiplication with $E$ on the right, the lemma is proven. 
\end{proof} 

By \cite[Proposition A.3]{VDW2}, there is a unique homomorphism $\Delta:M(A)\rightarrow M(A\otimes A)$ extending $\Delta$ and such that $\Delta(1) = E$.  Similarly the maps $\id\otimes \Delta$ and $\Delta\otimes \id$ extend to maps from $M(A\otimes A)$ to $M(A\otimes A\otimes A)$. 

For example, note that
\begin{align} \label{eq:delta-lambda-rho} \Delta(\lambda_{k}) &=
  (\lambda_{k} \otimes 1)\Delta(1), & \Delta(\rho_{m}) &= (1 \otimes \rho_{m})\Delta(1).
\end{align}

The following proposition gathers the properties of $\Delta$, $\epsilon$ and $\Delta(1)$ which guarantee that $(A,\Delta)$ forms a weak multiplier bialgebra in the sense of \cite[Definition 2.1]{Boh1}. We will call it the \emph{total weak multiplier bialgebra} associated to $\mathscr{A}$.

\begin{Prop} Let $\mathscr{A}$ be a partial bialgebra with total algebra $A$, total comultiplication $\Delta$ and counit $\epsilon$. Then the following properties are satisfied.
\begin{enumerate}[label={(\arabic*)}]
\item Coassociativity: $(\Delta\otimes \id)\Delta = (\id\otimes \Delta)\Delta$ (as maps $M(A)\rightarrow M(A^{\otimes 3})$).
\item Counitality: $(\epsilon\otimes \id)(\Delta(a)(1\otimes b)) = ab = (\id\otimes \epsilon)((a\otimes 1)\Delta(b))$ for all $a,b\in A$.
\item Weak Comultiplicativity of the Unit: \[(\Delta(1)\otimes 1)(1\otimes \Delta(1)) = (\Delta\otimes \id)\Delta(1) = (\id\otimes \Delta)\Delta(1) = (1\otimes \Delta(1))(\Delta(1)\otimes 1).\]
\item \label{WMC} Weak Multiplicativity of the Counit: For all $a,b,c\in A$, one has \[(\epsilon\otimes \id)(\Delta(a)(b\otimes c)) = (\epsilon\otimes \id)((1\otimes a)\Delta(1)(b\otimes c))\] and 
\[(\epsilon\otimes \id)((a\otimes b)\Delta(c)) = (\epsilon\otimes \id)((a\otimes b)\Delta(1)(1\otimes c)).\]
\item Strong multiplier property: For all $a,b\in A$, one has \[\Delta(A)(1\otimes A)\cup (A\otimes 1)\Delta(A)\subseteq  A\otimes A.\] 
\end{enumerate}
\end{Prop}

\begin{proof} Most of these properties follow immediately from the definition of a partial bialgebra. For demonstrational purposes, let us check the first identity of property \ref{WMC}. Let us choose $a\in A(K)$, $b\in A(L)$ and $c\in A(M)$. Then \[\Delta(a)(b\otimes c) = \delta_{K_{ru},L_{lu}}\delta_{M_{lu},L_{ld}} \sum_r \Delta_{r,L_{ld}}(a)(b\otimes c).\]  Applying $(\epsilon\otimes \id)$ to both sides, we obtain by Proposition \ref{Propb} of Definition \ref{DefPartBiAlg} and counitality of $\epsilon$ that \[(\epsilon \otimes \id)(\Delta(a)(b\otimes c)) = \delta_{K_{ru},L_{lu},L_{ld},M_{lu}} \epsilon(b) ac.\] On the other hand, \begin{eqnarray*} (1\otimes a)\Delta(1)(b\otimes c) &=& \sum_{r,s,t} \UnitC{r}{s} b \otimes a\UnitC{s}{t}c \\ &=& \delta_{L_{ld},K_{ru},M_{lu}} b \otimes ac.\end{eqnarray*} Applying $(\epsilon\otimes \id)$, we find \begin{eqnarray*} (\epsilon\otimes \id)( (1\otimes a)\Delta(1)(b\otimes c) ) &=&  \delta_{L_{ld},K_{ru},M_{lu}}\delta_{L_{lu},L_{ld}}\delta_{L_{ru},L_{rd}} \epsilon(b)ac \\ &=&  \delta_{L_{ld},L_{lu},K_{ru},M_{lu}} \epsilon(b)ac,\end{eqnarray*} which agrees with the expression above.
\end{proof} 

\begin{Rem} 
Since also the expressions $\Delta(a)(b\otimes 1)$ and $(1\otimes a)\Delta(b)$ are in $A\otimes A$ for all $a,b\in A$, we see that $(A,\Delta)$ is in fact a \emph{regular} weak multiplier bialgebra \cite[Definition 2.3]{Boh1}.
\end{Rem} 

Recall from \cite[Section 3]{Boh1} that a regular weak multiplier
bialgebra admits four projections $A\rightarrow M(A)$, given
by \begin{align*} \bar{\Pi}^L(a) = (\epsilon\otimes \id)((a\otimes
  1)\Delta(1)),\quad & \bar{\Pi}^R(a) = (\id\otimes
  \epsilon)(\Delta(1)(1\otimes a)),\\ \Pi^L(a) = (\epsilon\otimes
  \id)(\Delta(1)(a\otimes 1)),\quad& \Pi^R(a) =
  (\id\otimes\epsilon)((1\otimes a)\Delta(1)),\end{align*} where the
right hand side expressions are interpreted as multipliers in the
obvious way. The relation  $\Delta(1)=\sum_{p} \rho_{p} \otimes \lambda_{p}$ and  condition \ref{Propc} in Definition \ref{DefPartBiAlg} imply
\begin{align*}
  \bar \Pi^{L}(A) &=\mathrm{span}\{\lambda_{p}:p\in I\} =  \Pi^{L}(A), &
  \bar \Pi^{R}(A) &= \mathrm{span}\{\rho_{p}:p\in I\} =\Pi^{R}(A).
\end{align*}
The \emph{base algebra} of $(A,\Delta)$ is therefore just the algebra
$\Fun_{\fin}(I)$ of finitely supported functions on $I$, and the
comultiplication of $A$ is (left and right) \emph{full} (meaning
roughly that the legs of $\Delta(A)$ span $A$) by \cite[Theorem 3.13]{Boh1}.  

 The maps $\Pi^{L}$ and $\Pi^{R}$ can also
be written in the form
\begin{align} \label{eq:pi} 
    \Pi^L(a) & = \sum_{p}\epsilon(\lambda_{p}a)\lambda_p, & \Pi^R(a) & =    \sum_{p}\epsilon(a \rho_{p}) \rho_p
\end{align}
because $\epsilon(\lambda_{k}\rho_{m} a \lambda_{l}\rho_{n})=0$  if $(k,l)\neq(m,n)$. These relations and  \eqref{EqDel}, \eqref{eq:delta-lambda-rho} imply
\begin{align} \label{eq:pi-l-delta}
  (\Pi^{L} \otimes \id)(\Delta(a)) &= \sum_{p} \lambda_{p}\otimes \lambda_{p}a, &
  (\id \otimes \Pi^{L})(\Delta(a)) &= \sum_{p} \rho_{p}a \otimes \lambda_{p}, & \\ \label{eq:pi-r-delta}
  (\Pi^{R} \otimes \id)(\Delta(a)) &= \sum_{p} \rho_{p} \otimes a\lambda_{p}, &
  (\id \otimes \Pi^{R})(\Delta(a)) &= \sum_{p} a\rho_{p} \otimes \rho_{p}.
\end{align}

Let us now show a converse. If $(A,\Delta)$ is a regular weak multiplier bialgebra, let us write $A^L = \Pi^L(A) = \bar{\Pi}^L(A)\subseteq M(A)$ and $A^R = \Pi^R(A)= \bar{\Pi}^R(A)\subseteq M(A)$ for the base algebras, where the identities follow from \cite[Theorem 3.13]{Boh1}. Then if moreover $(A,\Delta)$ is left and right full, we have that $A^L$ is (canonically) anti-isomorphic to $A^R$ by the map \[\sigma: A^L \rightarrow A^R, \quad \bar{\Pi}^L(a) \rightarrow \Pi^R(a), \qquad a\in A,\] by \cite[Lemma 4.8]{Boh1}. We then simply refer to $A^L$ as \emph{the} base algebra. 

\begin{Rem}\label{RemNak} We could also have used the map $\bar{\sigma}(\Pi^L(a)) = \bar{\Pi}^R(a)$ to identify $A^L$ and $A^R$. As it turns out, $\bar{\sigma}^{-1}\sigma$ is the (unique) Nakayama automorphism for some functional $\varepsilon$ on $A^L$, cf. \cite[Proposition 4.9]{Boh1}. Hence if $A^L$ is commutative, it follows that $\sigma = \bar{\sigma}$.
\end{Rem} 

\begin{Prop}\label{PropCharPBA} Let $(A,\Delta)$ be a left and right full regular weak multiplier bialgebra whose base algebra is isomorphic to $\Fun_{\fin}(I)$ for some set $I$, and such that moreover $A^LA^R \subseteq A$. Then $(A,\Delta)$ is the total weak multiplier bialgebra of a uniquely determined partial bialgebra $\mathscr{A}$ over $I$.
\end{Prop} 

\begin{Rem} The condition $A^LA^R \subseteq A$ is of course essential, as we want $A$ to behave locally as a bialgebra, not a multiplier bialgebra. Indeed, in case $A^L= \C$, the condition simply says that $A$ is unital. In general, it should be considered as a \emph{properness} condition. 
\end{Rem} 

\begin{proof} Let us write the standard generators (Dirac functions) of $A^L$ as $\lambda_k$ for $k\in I$, and write $\sigma(\lambda_k) = \rho_k\in A^R$. By assumption, $\UnitC{k}{l} = \lambda_k\rho_l\in A$. Further $A= AA^R = AA^L = A^LA=A^RA$, cf.~ the proof of \cite[Theorem 3.13]{Boh1}. Hence the $\UnitC{k}{l}$ make $A$ into the total algebra of an $I\times I$-partial algebra, as $A^L$ and $A^R$ pointwise commute by \cite[Lemma 3.5]{Boh1}. 

Define \[\Delta_{rs}(a) = (\rho_r\otimes \lambda_r)\Delta(a)(\rho_s\otimes \lambda_s).\] From \cite[Lemma 3.3]{Boh1}, it follows that $\Delta_{rs}$ is a map from $\Gr{A}{k}{l}{m}{n}$ to $\Gr{A}{k}{l}{r}{s}\otimes \Gr{A}{r}{s}{m}{n}$. That same lemma, together with the coassociativity of $\Delta$, show that the $\Delta_{rs}$ form a coassociative family.  

Now by \cite[Lemma 3.9]{Boh1}, we have $(\rho_k\otimes 1)\Delta(a) = (1\otimes \lambda_k)\Delta(a)$ for all $a$. By that same lemma and Remark \ref{RemNak}, we have as well $\Delta(a)(\rho_k\otimes 1) = \Delta(a)(1\otimes \lambda_k)$. Hence we may as well write \begin{eqnarray*} \Delta_{rs}(a) &=& (\rho_r\otimes 1)\Delta(a)(\rho_s\otimes 1) \\ &=& (1\otimes \lambda_r)\Delta(a)(1\otimes \lambda_s)\end{eqnarray*}  It is now straightforward that the counit map of $(A,\Delta)$ also provides a counit for the $\Delta_{rs}$, hence the $\Gr{A}{k}{l}{m}{n}$ also form a partial coalgebra. 

As $\Delta(a)(1\otimes \lambda_s)$ and $(1\otimes \lambda_r)\Delta(a)$ are already in $A\otimes A$, it is also clear that $\Delta_{rs}(a)$ is rcf for each $a$. The multiplicativity of the $\Delta_{rs}$ is then immediate from the multiplicativity of $\Delta$.

To show that $\Delta_{ll'}(\UnitC{k}{m}) = \delta_{l,l'} \UnitC{k}{l}\otimes \UnitC{l}{m}$, it suffices to show that $\Delta(1) = \sum_k \rho_k\otimes \lambda_k$. Now as $\Delta(1)(A\otimes A)  = \Delta(A)(A\otimes A)$, and as clearly $\Delta(a) = \sum_{r,s}\Delta_{rs}(a)$ in the strict topology for all $a\in A$, it follows that \[\Delta(1) = \left(\sum_k \rho_k\otimes \lambda_k\right)\Delta(1).\]  Similarly, $\Delta(1) = \Delta(1)\left(\sum_k\rho_k\otimes \lambda_k\right)$. On the other hand, by \cite[Lemma 4.10]{Boh1} it follows that we can then write \[\Delta(1) = \sum_{k\in I'} \rho_k\otimes \lambda_k\] for some subset $I'\subseteq I$. As by definition $\bar{\Pi}^L(A) = \Fun_{\fin}(I)$, we deduce that $I=I'$. 

For $a\in \Gr{A}{k}{l}{p}{q}$ and $b\in \Gr{A}{l}{m}{q}{r}$, we then have $\epsilon(ab) = \epsilon(a\UnitC{l}{q}b) = \epsilon(a)\epsilon(b)$ by \cite[Proposition 2.6.(4)]{Boh1}, which shows the partial multiplicativity of $\epsilon$. 

Finally, assume that $k$ was such that $\epsilon(\UnitC{k}{k})=0$. Then by the partial multiplication law, $\epsilon$ is zero on all $\Gr{A}{k}{l}{k}{l}$. Applying $\Delta_{kl}$ to $\Gr{A}{k}{l}{m}{n}$ and using the counit property on the first leg, it follows that $\Gr{A}{k}{l}{m}{n}=0$ for all $l,m,n$. In particular, $\UnitC{k}{m}=0$ for all $m$. But this entails $\lambda_k=0$, a contradiction. Hence $\epsilon(\UnitC{k}{k})\neq 0$. From the partial multiplication law, it follows that $\epsilon(\UnitC{k}{k})^2 = \epsilon(\UnitC{k}{k})$, hence $\epsilon(\UnitC{k}{k})=1$.

This concludes the proof that $(A,\Delta)$ determines a partial bialgebra $\mathscr{A}$. It is immediate that $(A,\Delta)$ is in fact the total weak multiplier bialgebra of $\mathscr{A}$. 
\end{proof} 


\subsection{Partial Hopf algebras}

We now formulate the notion of a partial Hopf algebra, whose total form will correspond to a weak multiplier Hopf algebra \cite{Boh1,VDW2,VDW1}. We will mainly refer to \cite{Boh1} for uniformity.

 Let us write $\circ$ for the inverse of $\wmult$, and $\bullet$ for the inverse of $\bmult$, so \[\begin{pmatrix} k & l \\ m & n \end{pmatrix}^{\circ} = \begin{pmatrix} l & k \\ n & m \end{pmatrix},\quad \begin{pmatrix} k & l \\ m & n \end{pmatrix}^{\bullet} = \begin{pmatrix} m & n \\ k & l \end{pmatrix},\quad \begin{pmatrix} k & l \\ m & n \end{pmatrix}^{\circ \bullet} = \begin{pmatrix} n & m \\ l & k \end{pmatrix}.\] The notation $\circ$ (resp. $\bullet$) will also be used for row vectors (resp. column vectors).

\begin{Def}\label{DefPartBiAlgAnt} An \emph{antipode} for an
  $I$-partial bialgebra $\mathscr{A}$ consists of linear
maps \[S:A(K)\rightarrow A(K^{\circ\bullet})\]
  such that the following property holds: for all $M,P\in M_2(I)$ and
  all $a\in A(M)$, \begin{align} \label{eq:antipode-pi-l}\underset{K\wmult
      L^{\circ\bullet}=P}{\sum_{K\bmult L = M}} a_{(1)K}S(a_{(2)L})&=
    \delta_{P_l,P_r}\epsilon(a)\mathbf{1}(P_l),
    \\ \label{eq:antipode-pi-r}
    \underset{K^{\circ\bullet}\wmult L=P}{\sum_{K\bmult L = M}}
    S(a_{(1)K})a_{(2)L}&=
    \delta_{P_l,P_r}\epsilon(a)\mathbf{1}(P_r).\end{align}

A partial bialgebra $\mathscr{A}$ is called a \emph{partial Hopf algebra} if it admits an antipode.
\end{Def} 

\begin{Rem} Note that condition \ref{Propd} of Definition \ref{DefPartBiAlg} again guarantees that the above sums are in fact finite.
\end{Rem}

If $S$ is an antipode for a partial bialgebra, we can extend $S$ to a
linear map \[S:A\rightarrow A\] on the total algebra $A$.  Conditions
\eqref{eq:antipode-pi-l} and \eqref{eq:antipode-pi-r} then take the
following simple form:
\begin{Lem} \label{lemma:antipode}
  A family of maps $S \colon A(K) \to A(K^{\circ\bullet})$ satisfies
  \eqref{eq:antipode-pi-l} and \eqref{eq:antipode-pi-r} if and only if
  the total map $S\colon A \to A$ satisfies 
  \begin{align} \label{eq:total-antipode}
 a_{(1)}S(a_{(2)}) &= \Pi^{L}(a), &
 S(a_{(1)})a_{(2)} &= \Pi^{R}(a)
  \end{align}
for all $a\in A$.
\end{Lem}

Note that these should be considered a priori as equalities of left (resp. right) multipliers on $A$.

\begin{proof}
For $M,P\in M_{2}(I)$ and $a\in A(M)$, the left and the right hand side  of \eqref{eq:antipode-pi-l} are the $P$-homogeneous components of $ a_{(1)}S(a_{(2)})$ and $\Pi^{L}(a)=\sum_{p} \epsilon(\lambda_{p}a)\lambda_{p}$, respectively.
\end{proof}

\begin{Lem}\label{LemAntiUnit} Let $\mathscr{A}$ be a partial Hopf algebra with antipode $S$. For all $k,l\in I$, $S(\UnitC{k}{l}) = \UnitC{l}{k}$.
\end{Lem}
\begin{proof} For example the first identity in Equation \eqref{eq:total-antipode} of Lemma \ref{lemma:antipode} applied to $\UnitC{k}{k}$ gives \[\sum_l S(\UnitC{l}{k}) = \sum_l \UnitC{k}{l}S(\UnitC{l}{k}) = \lambda_k,\] as $S(\UnitC{l}{k}) \in \Gr{A}{k}{k}{l}{l}$ and $\Pi^{L}(\UnitC{k}{k}) = \lambda_k$. This implies the lemma.
\end{proof} 
\begin{Rem} \label{remark:index-equivalence}
  Let $\mathscr{A}$ be an $I$-partial Hopf algebra. Then the relation
  on $I$ defined by
  \begin{align*}
    k \sim l \Leftrightarrow \UnitC{k}{l} \neq 0
  \end{align*}
is an equivalence relation. Indeed, it is reflexive and transitive by
assumptions (3) and (1) in Definition \ref{DefPartBiAlg}, and
symmetric by the preceding result. We call the set $\mathscr{I}$ of equivalence classes the \emph{hyperobject} set of $\mathscr{A}$. 
\end{Rem}
The existence of an antipode is closely related to partial invertibility of
the maps $T_{1},T_{2} \colon A \otimes A \to A\otimes A$ given by
\begin{align} \label{eq:wt-12}
  T_{1} (a\otimes b)&= \Delta(a)(1 \otimes b), &
  T_{2} (a\otimes b)&= (a \otimes 1)\Delta(b).
 \end{align}
The precise formulation involves the linear maps $E_{i},G_{i}
 \colon A\otimes A\to A\otimes A$ given by
\begin{align} \label{eq:e1g1}
  G_{1}(a\otimes b) &=
 \sum_{p} a\rho_{p} \otimes \rho_{p}b, &  E_{1}(a \otimes b) &=\Delta(1)(a\otimes b)=\sum_{p} \rho_{p}a\otimes \lambda_{p}b, \\ \label{eq:e2g2}
 G_{2}(a \otimes b) &= \sum_{p} a\lambda_{p} \otimes
    \lambda_{p}b, &
E_{2}(a\otimes b) &= (a\otimes b)\Delta(1)=\sum_{p} a\rho_{p} \otimes b\lambda_{p}.
\end{align}
\begin{Prop} \label{prop:riti}
  Let $\mathscr{A}$ be a partial Hopf algebra with total algebra $A$,
  total comultiplication $\Delta$ and antipode  $S$. Then the maps
  $R_{1},R_{2} \colon A \otimes A \to M(A \otimes A)$ given by
  \begin{align*}
    R_{1}(a \otimes b) &= a_{(1)}\otimes S(a_{(2)})b, &
    R_{2}(a\otimes b) &= aS(b_{(1)})\otimes b_{(2)}
  \end{align*}
  take values in $A\otimes A$ and satisfy for $i=1,2$ the relations
  \begin{align} \label{eq:riti}
    T_{i}R_{i}&=E_{i}, & R_{i}T_{i}&= G_{i}, & T_{i}R_{i}T_{i}&= T_{i}, & R_{i}T_{i}R_{i} &= R_{i}.
  \end{align}
\end{Prop}
\begin{proof}
  The map $R_{1}$ takes values in $A\otimes A$ because
  \begin{align*}
  a_{(1)} \otimes
  S(a_{(2)})\lambda_{k}\rho_{l} =  a_{(1)} \otimes S(\rho_{l}\lambda_{k}a_{(2)}) \in A
  \otimes A
  \end{align*}
  for all $a\in A$, and Lemma
  \ref{lemma:antipode},  Equation \eqref{eq:pi-l-delta} and Lemma \ref{LemAntiUnit} imply
  \begin{align*}
    T_{1}R_{1}(a \otimes b)&= a_{(1)} \otimes a_{(2)}S(a_{(3)})b =
    a_{(1)} \otimes \Pi^{L}(a_{(2)})b =
    \sum_{p} \rho_{p}a \otimes \lambda_{p}b, \\
    R_{1}T_{1}(a \otimes b) &= a_{(1)} \otimes S(a_{(2)})a_{(3)}b =
    a_{(1)} \otimes \Pi^{R}(a_{(2)})b = \sum_{p} a\rho_{p}\otimes
    \rho_{p}b.
  \end{align*}
 The relations $T_{1}R_{1}T_{1}=T_{1}$ and
$R_{1}T_{1}R_{1}=R_{1}$ follow easily from  \eqref{EqDel} and
\eqref{eq:delta-lambda-rho}. The assertions concerning $R_{2}$ and
$T_{2}$ follow similarly.
\end{proof}
\begin{Theorem}  \label{theorem:partial-hopf-algebra}
  Let $\mathscr{A}$ be a partial bialgebra with total algebra $A$,
  total comultiplication $\Delta$ and counit $\epsilon$. Then the
  following conditions are equivalent:
  \begin{enumerate}[label={(\arabic*)}]
  \item\label{tph1} $\mathscr{A}$ is a partial Hopf algebra.
  \item\label{tph2} There exist linear maps $R_{1},R_{2} \colon A\otimes A\to
    A\otimes A$ satisfying  \eqref{eq:riti}.
  \item\label{tph3} $(A,\Delta,\epsilon)$  is a weak multiplier Hopf algebra in the sense of \cite{VDW1}.
  \end{enumerate}
  If these conditions hold, then the total  antipode of $\mathscr{A}$ coincides with the antipode of $(A,\Delta,\epsilon)$.
\end{Theorem}
\begin{proof}
\ref{tph1} implies \ref{tph2} by Proposition \ref{prop:riti}. \ref{tph2} is equivalent to \ref{tph3} by Definition
1.14 in \cite{VDW1}. Indeed, the maps $G_{1},G_{2}$ defined in \eqref{eq:e1g1} and \eqref{eq:e2g2} satisfy
\begin{align*}
  G_{1}(a_{(1)} \otimes b) \otimes a_{(2)}c &= \sum_{p} a_{(1)} \otimes \rho_{p}b
  \otimes a_{(2)}\lambda_{p}c, \\
  ac_{(1)} \otimes G_{2}(b\otimes c_{(2)}) &=\sum_{p} a\rho_{p}c_{(1)} \otimes b\lambda_{p} \otimes c_{(2)}
\end{align*}
and therefore coincide with the maps introduced in Proposition 1.14 in
\cite{VDW1}.  Finally, assume \ref{tph3}. Then
 Lemma 6.14 and equation (6.14) in \cite{Boh1} imply that the antipode
$S$ of $(A,\Delta)$ satisfies $S(A(K))\subseteq A(K^{\circ\bullet})$ and relation \eqref{eq:total-antipode}.  Now, Lemma \ref{lemma:antipode} implies \ref{tph1}.
\end{proof}

From \cite[Proposition 3.5 and Proposition 3.7]{VDW1} or \cite[Theorem
6.12 and Corollary 6.16]{Boh1}, we can conclude that the antipode of a
partial Hopf algebra reverses the multiplication and
comultiplication. Denote by $\Delta^{\op}$ the composition of
$\Delta$ with the flip map.

\begin{Cor} \label{corollary:antipode} Let $\mathscr{A}$ be a partial
  Hopf algebra. Then the total antipode $S:A\rightarrow A$ is uniquely determined and satisfies
  $S(ab) = S(b)S(a)$ and $\Delta(S(a)) = (S\otimes S)\Delta^{\op}(a)$
  for all $a,b\in A$.
\end{Cor} 
\begin{proof} Uniqueness of the antipode follows from the identities \eqref{eq:total-antipode}, see also \cite[Remark 2.8.(ii)]{VDW1}. 
\end{proof} 

We will need the following relation between $\epsilon$ and $S$ at some point.

\begin{Lem}\label{LemCoAnt} Let $(\mathscr{A},\Delta)$ be a partial Hopf algebra. Then $\epsilon\circ S = \epsilon$ on each $\Gr{A}{k}{l}{m}{n}$.
\end{Lem}

\begin{proof} Using the notation in Proposition \ref{prop:riti} and the discussion preceding it, we have that \[T_1: \sum_p(A\rho_p\otimes \rho_p A)\rightarrow \Delta(1)(A\otimes A)\] is a bijection with $R_1$ as inverse. As one easily verifies that $(\id\otimes \epsilon)T_1 = \id\otimes \epsilon$ by the partial multiplicativity and counit property of $\epsilon$, it follows that also $(\id\otimes \epsilon)R_1 = \id\otimes \epsilon$ on $\Delta(1)(A\otimes A)$. Applying both sides to $a\otimes \UnitC{k}{k}$ with $a\in \Gr{A}{k}{l}{k}{l}$, we find \[(\id\otimes (\epsilon\circ S))\Delta_{kl}(a) = a.\] Applying $\epsilon$ to this identity, we find $\epsilon\circ S = \epsilon$ on each $\Gr{A}{k}{l}{k}{l}$, and hence on all $\Gr{A}{k}{l}{m}{n}$.
\end{proof}

In practice, it is convenient to have an \emph{invertible} antipode around. Although the invertibility often comes for free in case extra structure is around, we will mostly just impose it to make life easier. The following definition follows the terminology of \cite{VDae1}. 

\begin{Def} Let $\mathscr{A}$ be a partial Hopf algebra. We call $\mathscr{A}$ a \emph{regular} partial Hopf algebra if the antipode maps on $\mathscr{A}$ are invertible.
\end{Def}

From the uniqueness of the antipode, it follows immediately that $S^{-1}$ is then an antipode for $(\mathscr{A},\Delta^{\op})$. Conversely, if both $(\mathscr{A},\Delta)$ and $(\mathscr{A},\Delta^{\op})$ have antipodes, then $(\mathscr{A},\Delta)$ is a regular partial Hopf algebra. 

\subsection{Invariant integrals}


\begin{Def}
  Let $\mathscr{A}$ be an $I$-partial bialgebra.  We call a family of
  functionals
\begin{align} \label{eq:functionals}
  \phic{k}{m} \colon A\pmat{k}{k}{m}{m} \to \C
\end{align}
a \emph{left invariant} \emph{integral} if
 $\phic{k}{k}(\UnitC{k}{k})=1$ for all $k\in
I$ and
\begin{align}
  \label{eq:integral}
   (\id \otimes \phic{l}{m})(\Delta_{ll}(a)) 
&= \delta_{k,p} \phic{k}{m}(a)
  \UnitC{k}{l} 
\end{align}
 for all $k,l,m,p\in I$, $a \in A\pmat{k}{p}{m}{m}$. 
 
 We call them a \emph{right invariant}  \emph{integral} if instead one has \begin{align}
  (\phic{k}{l} \otimes
  \id)(\Delta_{ll}(a))&= \delta_{m,p} \phic{k}{m}(a) \UnitC{l}{m}\end{align}
 for all $k,l,m,p\in I$, $a \in A\pmat{k}{k}{m}{p}$. 
 
 A left integral which is at the same time a right invariant integral will simply be called an \emph{invariant integral}.
\end{Def}

As before, we can extend a (left or right) invariant integral to a functional $\phi$ on $A$ by linearity and by putting $\phi=0$ on $\Gr{A}{k}{l}{m}{n}$ if $k\neq l$ or $m\neq n$. The total form of the invariance conditions
\eqref{eq:integral}  reads as follows. 

\begin{Lem} \label{lemma:total-integral}
  A family of functionals  as in   \eqref{eq:functionals}
  is left invariant
  if and only if
for all $a,b\in A$,
  \begin{align*}
(\id\otimes \phi)((b\otimes 1)\Delta(a)) &= \sum_{k}\phi(\lambda_{k}a)b\lambda_k.
      \end{align*}
      It defines a right invariant functional if and only if 
   \begin{align*}   (\phi\otimes \id)(\Delta(a)(1\otimes b)) &= \sum_{n}
\phi(\rho_{n} a)\rho_n b.\end{align*}
\end{Lem}

\begin{proof}
  Straightforward.
\end{proof}

We have the following form of \emph{strong invariance}.

\begin{Lem} \label{lemma:strong-invariance}
  Let $\mathscr{A}$ be a partial Hopf algebra with left invariant integral $\phi$. Then
  for all $a\in A$,
  \begin{align*}
    S\left(( \id\otimes
    \phi)(\Delta(b)(1 \otimes a))\right) &= (\id \otimes \phi)((1 \otimes b)\Delta(a)).
  \end{align*}
  Similarly, if $\mathscr{A}$ is a partial Hopf algebra with right invariant integral $\phi$, then 
   \begin{align*} S\left((\phi \otimes
    \id)((a\otimes 1)\Delta(b))\right) &= (\phi \otimes \id)(\Delta(a)(b\otimes 1)).\end{align*}
\end{Lem}
\begin{proof}
 The counit property, the relations \eqref{EqDel} and
 \eqref{eq:total-antipode} and Lemma \ref{lemma:total-integral} imply
  \begin{align*}
    a_{(1)}\phi(ba_{(2)}) &= \sum_{n}
    a_{(1)}\phi(\epsilon(b_{(1)}\rho_{n})b_{(2)}\lambda_{n}a_{(2)}) \\
&= \sum_{n} \epsilon(b_{(1)}\rho_{n})\rho_{n}a_{(1)}\phi(b_{(2)}a_{(2)})
\\
&= S(b_{(1)})b_{(2)}a_{(1)}\phi(b_{(3)}a_{(2)}) =
S(b_{(1)})\phi(b_{(2)}a)
  \end{align*}
for all $a,b \in A$. The second equation 
follows similarly.
\end{proof}

\begin{Lem} Assume that $\mathscr{A}$ is a regular  $I$-partial Hopf algebra which admits a left invariant integral $\phi$. Then the following hold.
\begin{enumerate}[label = {(\arabic*)}]
\item\label{LI1} $\phi(\UnitC{k}{m})=1$ for all $k,m\in I$ with $\UnitC{k}{m}\neq 0$.
\item\label{LI2} $\phi$ is uniquely determined.
\item\label{LI3} $\phi=\phi S$.
\item\label{LI4} $\phi$ is invariant.
\end{enumerate}
\end{Lem}

\begin{proof} 
To see \ref{LI1}, take $a=\UnitC{k}{k}$ in \eqref{eq:integral}. 

Now by Corollary \ref{corollary:antipode}, we have that $\phi S$ is right invariant. But assume that $\psi$ is any
 right invariant integral.     Then for all $k,l,m\in I$, $a\in A\pmat{k}{k}{m}{m}$,
    \begin{align*}
      \phic{k}{m}(a)  &= (\Grt{\psi}{k}{k} \otimes
      \phic{k}{m})(\Delta_{kk}(a)) = \Grt{\psi}{k}{m}(a)\Grt{\phi}{k}{m}(\UnitC{k}{m}) = \Grt{\psi}{k}{m}(a) .
    \end{align*}
  This proves \ref{LI2}, \ref{LI3} and \ref{LI4}.  
     \end{proof}

We will need the following lemma at some point, cf.~ \cite[Proposition 3.4]{VDae2}.

\begin{Lem}\label{LemFaith} Let $\mathscr{A}$ be a regular partial Hopf
  algebra with an invariant integral $\phi$. Then
  $\phi$ is faithful in the following sense: if $a\in A$ and
  $\phi(ab) =0$ (resp. $\phi(ba)=0$) for all $b\in A$, then
  $a=0$.
\end{Lem} 

\begin{proof} Suppose $a\in A$ and $\phi(ba)=0$ for all $b\in A$. By the support condition of $\phi$, we may suppose $a$ is homogeneous, $a\in \Gr{A}{k}{l}{m}{n}$.

We will first show that necessarily $\epsilon(a)=0$, for which we may already assume $k=m$ and $l=n$. Indeed, the condition on $a$ implies also $(\id\otimes \phi)(\Delta_{lk}(b)(1\otimes a))=0$ for all $b\in \Gr{A}{s}{r}{l}{k}$. Applying the strong invariance identity, we deduce \begin{equation}\label{EqSwitch}(\id\otimes \phi)((1\otimes b)\Delta_{rs}(a))=0,\qquad \forall b\in \Gr{A}{s}{r}{l}{k}.\end{equation} Writing $\Delta_{rs}(a) = \sum_i p_i\otimes q_i$ with the $p_i$ linearly independent, we deduce $\phi(bq_i)=0$ for all $i$ and $b$, and so also $\sum_i \phi(S(p_i)q_i)=0$. Hence $0=\sum_r \phi(S(a_{(1){\tiny \begin{pmatrix} k & l \\ r & l \end{pmatrix}}})a_{(2){\tiny \begin{pmatrix} r & l \\ k & l\end{pmatrix}}}) = \phi(\epsilon(a) \UnitC{l}{l}) = \epsilon(a)$.

Note now that from \eqref{EqSwitch}, it follows that for any functional $\omega$ on $\Gr{A}{k}{l}{m}{n}$, also $a'=(\omega\otimes \id)\Delta_{mn}(a)$ satisfies $\phi(ba')=0$ for all $b\in A$. Hence, by what we have just shown, $\epsilon(a')=0$, i.e.~ $\omega(a)=0$. As $\omega$ was arbitrary, we deduce $a=0$.

The other case follows similarly, or by considering the opposite comultiplication.
\end{proof}

\subsection{Partial compact quantum groups}

Our main objects of interest are partial Hopf algebras with involutions and invariant integrals.
\begin{Def} A \emph{partial $*$-algebra} $\mathscr{A}$ is a partial
  algebra whose total algebra $A$ is equipped with an antilinear,
  antimultiplicative involution $*\colon A\rightarrow A$, $ a\mapsto
  a^*$,  such that the $\mathbf{1}_k$ are selfadjoint for all $k$ in
  the object set. 
\end{Def} 

One can of course give an alternative definition directly in terms of the partial algebra structure by requiring that we are given antilinear maps $A(k,l)\rightarrow A(l,k)$ satisfying the obvious antimultiplicativity and involution properties.

\begin{Def} A \emph{partial $*$-bialgebra} $\mathscr{A}$ is a
 partial bialgebra whose underlying partial algebra has been
  endowed with a partial $*$-algebra structure such that
$\Delta_{rs}(a)^* = \Delta_{sr}(a^*)$ for all $a \in \Gr{A}{k}{l}{m}{n}$.
A \emph{partial Hopf $*$-algebra} is a partial bialgebra which is at the same time a partial $*$-bialgebra and a partial Hopf algebra.
\end{Def} 
Thus, a partial bialgebra is a partial
$*$-bialgebra if and only if the underlying weak multiplier bialgebra
 is a weak multiplier $*$-bialgebra.

From Theorem \ref{theorem:partial-hopf-algebra} and \cite{Boh1},
\cite{VDW1}, we can deduce:
\begin{Cor} \label{cor:involutive}
  An $I$-partial $*$-bialgebra $\mathscr{A}$ is an $I$-partial Hopf
  $*$-algebra if and only if the weak multiplier $*$-bialgebra
  $(A,\Delta)$ is a weak multiplier Hopf $*$-algebra. In that case,
  the counit and antipode satisfy
  $\epsilon(a^{*})=\overline{\epsilon(a)}$ and $S(S(a)^{*})^{*}=a$ for
  all $a\in A$. In particular, the total antipode is bijective.
\end{Cor}
\begin{proof}
  The if and only if part follows immediately from  Theorem
  \ref{theorem:partial-hopf-algebra}, the relation for the counit  from
uniqueness of the counit  \cite[Theorem 2.8]{Boh1}, and the relation
for the antipode from \cite[Proposition 4.11]{VDW1}.
\end{proof}

We are finally ready to formulate our main definition.
\begin{Def} A \emph{partial compact quantum group} $\mathscr{G}$ is a
  partial Hopf $*$-algebra $\mathscr{A} = P(\mathscr{G})$ with an invariant integral  $\phi$ that is positive in the sense  that $\phi(a^*a)\geq 0$ for all $a\in A$. We also say that $\mathscr{G}$ is the partial compact quantum group \emph{defined by} $\mathscr{A}$.
\end{Def} 

\begin{Rem} It will follow from our Proposition \ref{prop:rep-cosemisimple} and  \cite[Theorem 3.3 and Theorem 4.4]{Hay1} that for $I$ finite, a partial compact quantum group is precisely a compact quantum group of face type \cite[Definition 4.1]{Hay1}. However, we feel that terminology could be misleading if the object set is not finite. When referring to partial compact quantum groups, we feel that it is better reflected that only the \emph{parts} of this object are to be considered compact, not the total object. 
\end{Rem} 

\section{Partial tensor categories}

The notion of a partial algebra has a nice categorification. Recall first that the appropriate (vertical) categorification of a unital $\C$-algebra is a $\C$-linear additive tensor category. From now on, by `category' we will by default mean a $\C$-linear additive category. 

\begin{Def} A \emph{partial tensor category} $\CatCC$ over a set $\mathscr{I}$ consists of 
\begin{itemize}
\item[$\bullet$] a collection of (small) categories $\mathcal{C}_{\alpha\beta}$ with $\alpha,\beta\in \mathscr{I}$, 
\item[$\bullet$] $\C$-bilinear functors \[\otimes: \CatC_{\alpha\beta}\times \CatC_{\beta\gamma}\rightarrow \CatC_{\alpha\gamma},\] 
\item[$\bullet$] natural isomorphisms \[ a_{X,Y,Z}: (X\otimes Y)\otimes Z \rightarrow X\otimes (Y\otimes Z),\qquad X \in \CatC_{\alpha\beta},Y\in \CatC_{\beta\gamma},Z\in \CatC_{\gamma\delta},\] 
\item[$\bullet$] non-zero objects $\Unitb_{\alpha} \in \CatC_{\alpha\alpha}$,
\item[$\bullet$] natural isomorphisms \[\lambda_X^{(\alpha)}:\Unitb_\alpha\otimes X \rightarrow X,\qquad \rho_X^{(\beta)}:X\otimes \Unitb_\beta\rightarrow X, \qquad X\in \CatC_{\alpha\beta},\]
\end{itemize}
satisfying the obvious associativity and unit constraints. 
\end{Def}

\begin{Rem} In true analogy with the partial algebra case, we could let the $\Unitb_\alpha$ also be zero objects, but this generalisation will not be needed in the following. 
\end{Rem}

The corresponding total notion is as follows. 

\begin{Def} A \emph{tensor category with local units (indexed by $\mathscr{I}$)} consists of
\begin{itemize}
\item[$\bullet$] a (small) category $\CatC$, 
\item[$\bullet$] a $\C$-bilinear functor $\otimes: \CatC\times \CatC \rightarrow \CatC$ with compatible associativity constraint $a$, 
\item[$\bullet$]\label{FinSup} a collection $\{\Unitb_\alpha\}_{\alpha\in \mathscr{I}}$ of objects such that 
\begin{enumerate}[label=(\arabic*)] 
\item $\Unitb_\alpha\otimes \Unitb_\beta \cong 0$ for each $\alpha\neq \beta$, and
\item for each object $X$,  $\Unitb_\alpha\otimes X \cong 0 \cong X\otimes \Unitb_\alpha$ for all but a finite set of $\alpha$,
\end{enumerate}
\item[$\bullet$]\label{UnCon} natural isomorphisms $\lambda_X:\oplus_\alpha (\Unitb_\alpha\otimes X) \rightarrow X$ and $\rho_X:\oplus_\alpha(X\otimes \Unitb_\alpha)\rightarrow X$ satisfying the obvious unit conditions. 
\end{itemize} 
\end{Def}

Note that the condition \ref{UnCon} makes sense because of the local support condition in \ref{FinSup}. 

\begin{Rem} \begin{enumerate}[label=(\arabic*)]
\item There is no problem in modifying Mac Lane's coherence theorem, and we will henceforth assume that our partial tensor categories and tensor categories with local units are strict, just to lighten notation. 
\item One can also see the global tensor category $\CatC$ as an inductive limit of (unital) tensor categories. 
\end{enumerate}
\end{Rem}

\begin{Not} If $(\CatC,\otimes,\{\Unitb_\alpha\})$ is a tensor category with local units, and $X\in \CatC$, we define \[X_{\alpha\beta} = \Unitb_\alpha\otimes X \otimes \Unitb_\beta,\] and we denote by \[\eta_{\alpha\beta}:X_{\alpha\beta} \rightarrow \oplus_{\gamma,\delta} \left(\Unitb_\gamma \otimes X \otimes \Unitb_\delta\right) \cong X\] the natural inclusion maps. 
\end{Not}

\begin{Lem} Up to equivalence, there is a canonical one-to-one correspondence between partial tensor categories and tensor categories with local units. 
\end{Lem}

The reader can easily cook up the definition of equivalence referred to in this lemma.

\begin{proof} Let $(\CatC,\otimes,\{\Unitb_\alpha\}_{\alpha\in \mathscr{I}})$ be a tensor category with local units indexed by $\mathscr{I}$. Then the $\CatC_{\alpha\beta} = \{X \in \CatC\mid X_{\alpha\beta} \underset{\eta_{\alpha\beta}}{\cong} X\}$, seen as full subcategories of $\CatC$, form a partial tensor category upon restriction of $\otimes$.

Conversely, let $\CatCC$ be a partial tensor category. Then we let $\CatC$ be the category formed by formal finite direct sums $\oplus X_{\alpha\beta}$ with $X_{\alpha\beta}\in \CatC_{\alpha\beta}$, and with \[\Mor(\oplus X_{\alpha\beta},\oplus Y_{\alpha\beta}) := \oplus_{\alpha\beta} \Mor(X_{\alpha\beta},Y_{\alpha\beta}).\] The tensor product can be extended to $\CatC$ by putting $X_{\alpha\beta} \otimes X_{\gamma\delta} = 0$ when $\beta\neq \gamma$. The associativity constraints can then be summed to an associativity constraint for $\CatC$. It is evident that the $\Unitb_\alpha$ provide local units for $\CatC$. 
\end{proof}

\begin{Rem} Another global viewpoint is to see the collection of
  $\CatC_{\alpha\beta}$ as a 2-category with 0-cells indexed by the
  set $\mathscr{I}$, the objects of the $C_{\alpha\beta}$ as 1-cells,
  and  the morphisms of the $C_{\alpha\beta}$ as 2-cells. As for
  partial algebras vs.~ linear categories, we will not emphasize this
  way of looking at our structures, as this viewpoint is not
  compatible with the notion of a monoidal functor between partial tensor categories.
\end{Rem} 

Continuing the analogy with the algebra case, we define the enveloping \emph{multiplier tensor category} of a tensor category with local units. 

\begin{Def} Let $\CatCC$ be a partial tensor category over $\mathscr{I}$ with total tensor category $\CatC$. The \emph{multiplier tensor category} $M(\CatC)$ of $\CatC$ is defined to be the category consisting of formal sums $\oplus_{\alpha,\beta\in \mathscr{I}} X_{\alpha\beta}$ which are rcf, and with \[\Mor(\oplus X_{\alpha\beta},\oplus Y_{\alpha\beta}) = \left(\prod_\beta\bigoplus_\alpha  \Mor(X_{\alpha\beta},Y_{\alpha\beta}) \right) \cap \left(\prod_\alpha\bigoplus_\beta \Mor(X_{\alpha\beta},Y_{\alpha\beta})\right),\] the composition of morphisms being entry-wise (`Hadamard product'). 
\end{Def}

\begin{Rem} Because of the rcf condition on objects, we could in fact have written simply $\Mor(\oplus X_{\alpha\beta},\oplus Y_{\alpha\beta}) = \prod_{\alpha\beta} \Mor(X_{\alpha\beta},Y_{\alpha\beta})$. 
\end{Rem} 

The tensor product of $\CatC$ can be extended to $M(\CatC)$ by putting \[\left(\oplus X_{\alpha\beta}\right)\otimes \left(\oplus Y_{\alpha\beta}\right) = \oplus_{\alpha,\beta,\gamma} \left(X_{\alpha\beta}\otimes Y_{\beta\gamma}\right),\] and similarly for morphism spaces. This makes sense because of the rcf condition of the objects of $M(\CatC)$. The associativity constraints of the $\CatC_{\alpha\beta}$ can be summed to an associativity constraint for $M(\CatC)$, while $\Unitb := \oplus_{\alpha\in \mathscr{I}} \Unitb_\alpha$ becomes a unit for $M(\CatC)$, rendering $M(\CatC)$ into an ordinary tensor category (with unit object).

\begin{Rem} With some effort, a more intrinsic construction of the
  multiplier tensor category can be given in terms of couples of
  endofunctors, in the same vein as the construction of the multiplier
  algebra of  a non-unital algebra.
\end{Rem} 

\begin{Exa}\label{ExaVectBiGr} Let $I$ be a set. We can consider the partial tensor category $\CatCC = \{\Vect_{\fd}\}_{i,j\in I}$ where each $\CatC_{ij}$ is a copy of the category of finite-dimensional vector spaces $\Vect_{\fd}$, and with each $\otimes$ the ordinary tensor product. The total category $\CatC$ can then be identified with the category $\Vectif$ of finite-dimensional bi-graded vector spaces with the `balanced' tensor product over $I$. More precisely, the tensor product of $V$ and $W$ is $V\itimes W$ with components \[\Gru{(}{k}{}V\itimes W\Gru{)}{}{m} = \oplus_l \;(\Gru{V}{k}{l}\otimes \Gru{W}{l}{m})\subseteq V\otimes W.\] The multiplier category $M(\Vectif)$ equals $\Vectrcf$, the category of bigraded vector spaces which are rcfd (i.e.~ finite-dimensional on each row and column).
\end{Exa}

We now formulate the appropriate notion of a functor between partial tensor categories. Let us first give an auxiliary definition.

\begin{Def} Let $\CatCC$ be a partial tensor category over $\mathscr{I}$. If $\mathscr{J}\subseteq \mathscr{I}$, we call $\CatDD = \{\CatC_{\alpha\beta}\}_{\alpha,\beta\in \mathscr{J}}$ a \emph{restriction} of $\CatCC$. 
\end{Def} 

\begin{Def} Let $\CatCC$ and $\CatDD$ be partial tensor categories over respective sets $\mathscr{I}$ and $\mathscr{J}$, and let \[\phi:\mathscr{J}\rightarrow \mathscr{I},\quad k \mapsto k'\] determine a decomposition $\mathscr{J} = \{\mathscr{J}_\alpha\mid \alpha\in \mathscr{I}\}$ with $k\in \mathscr{J}_\alpha \iff \phi(k)=\alpha$. 

A \emph{unital morphism} from $\CatCC$ to $\CatDD$ (based on $\phi$)
consists of $\C$-linear functors \[F_{kl}: \CatC_{k'l'}\rightarrow
\CatD_{kl},\quad X\mapsto F_{kl}(X) = \Gru{F(X)}{k}{l}\] natural
monomorphisms \[\iota^{(klm)}_{X,Y}: \GrDA{F(X)}{k}{l} \otimes
\GrDA{F(Y)}{l}{m} \hookrightarrow \GrDA{F(X\otimes Y)}{k}{m}, \quad
X\in \CatC_{k'l'},Y\in \CatD_{l'm'},\] and isomorphisms \[\mu_{k}:
\Unitb_k \cong \GrDA{F(\Unitb_{k'})}{k}{k}\] satisfying the following conditions. \begin{enumerate}[label=(\arabic*)]
\item (Unitality)  $\GrDA{F(\Unitb_{\alpha})}{k}{l}= 0$ if $k\neq l$ in $\mathscr{J}_\alpha$.
\item (Local finiteness) For each $\alpha,\beta\in \mathscr{I}$ and $X\in \CatC_{\alpha\beta}$, the application $(k,l)\mapsto \GrDA{F(X)}{k}{l}$ is rcf on $\mathscr{J}_{\alpha}\times \mathscr{J}_{\beta}$. 
\item (Multiplicativity) For all $X\in \CatC_{k'\beta}$ and $Y\in \CatC_{\beta m'}$, one has\[\bigoplus_{l\in \mathscr{J}_\beta} \iota^{(klm)}_{X,Y}: \left(\bigoplus_{l\in \mathscr{J}_\beta} \GrDA{F(X)}{k}{l} \otimes \GrDA{F(Y)}{l}{m}\right) \cong \GrDA{F(X\otimes Y)}{k}{m}.\]
\item (Coherence) The $\iota^{(klm)}$ satisfy the 2-cocycle condition making \[\xymatrix{F_{kl}(X)\otimes F_{lm}(Y)\otimes F_{mn}(Z) \ar[rr]^{\id\otimes \iota^{(lmn)}_{Y,Z}} \ar[d]_{\iota^{(klm)}_{X,Y}\otimes\id}&& F_{kl}(X)\otimes F_{ln}(Y\otimes Z)\ar[d]^{\iota^{(kln)}_{X,Y\otimes Z}}\\ F_{km}(X\otimes Y)\otimes F_{mn}(Z) \ar[rr]_{\iota^{(kmn)}_{X\otimes Y,Z}}&& F_{kn}(X\otimes Y \otimes Z)}\] commute for all $X\in \CatC_{k'l'},Y\in \CatC_{l'm'}, Z\in \CatC_{m'n'}$, and the $\mu_k$ satisfy the commutation relations \[\xymatrix{ \GrDA{F(X)}{k}{l}\otimes \Unitb_l \ar[r]^{\!\!\!\!\id\otimes \mu_l} \ar@{=}[d] & \GrDA{F(X)}{k}{l} \otimes \GrDA{F(\Unitb_{l'})}{l}{l} \ar[d]^{\iota^{(kll)}_{X, \Unitb_{l'}}} \\ \GrDA{F(X)}{k}{l} & \ar@{=}[l] \GrDA{F(X\otimes \Unitb_{l'})}{k}{l}} \qquad \xymatrix{  \Unitb_k\otimes \GrDA{F(X)}{k}{l}\ar[r]^{\!\!\!\!\mu_k\otimes \id} \ar@{=}[d] & \GrDA{F(\Unitb_{k'})}{k}{k} \otimes \GrDA{F(X)}{k}{l} \ar[d]^{\iota^{(kkl)}_{\Unitb_{k'},X}} \\ \GrDA{F(X)}{k}{l} & \ar@{=}[l] \GrDA{ F(\Unitb_{k'}\otimes X)}{k}{l}} \]
\end{enumerate}

A \emph{morphism} from $\CatCC$ to $\CatDD$ is a unital morphism from $\CatCC$ to a restriction of $\CatDD$. 
\end{Def}

The corresponding global notion (of a unital morphism) is as follows.

\begin{Lem} Let $\CatCC$ and $\CatDD$ be partial tensor categories over respective sets $\mathscr{I}$ and $\mathscr{J}$. Fix an application \[\phi: \mathscr{J}\rightarrow \mathscr{I}\] inducing a disjoint decomposition $\{\mathscr{J}_\alpha\mid \alpha\in \mathscr{I}\}$. Then there is a one-to-one correspondence between unital morphisms $\CatCC\rightarrow \CatDD$ based on $\phi$ and functors $F:\CatC \rightarrow M(\CatD)$ with isomorphisms \[\iota_{X,Y}:F(X)\otimes F(Y)\cong F(X\otimes Y),\qquad \mu_\alpha:\oplus_{k\in \mathscr{J}_\alpha} \Unitb_k \cong F(\Unitb_\alpha)\] satisfying the natural coherence conditions. 
\end{Lem} 
\begin{Rem} If $J_{\alpha}=\emptyset$, the global functor $F$ sends
  $\Unitb_{\alpha}$ to the zero object in $M(\mathcal{D})$.
\end{Rem} 

The reader has already furnished for himself the notion of 
equivalence of partial tensor categories. There is a closely related
but weaker notion of equivalence corresponding to chopping up a partial tensor category into smaller pieces (or, vice versa, gluing certain blocks of a partial tensor category together). Let us formalize this in the following definition.

\begin{Def} Let $\CatCC$ and $\CatDD$ be partial tensor categories. We say $\CatDD$ is a \emph{partitioning} of $\CatCC$ (or $\CatCC$ a \emph{globalisation} of $\CatDD$) if there exists a unital morphism $\CatCC\rightarrow \CatDD$ inducing an equivalence of categories $\CatC\rightarrow \CatD$.
\end{Def}

The partial tensor categories that we will be interested in will be required to have some further structure. 

\begin{Def} A partial tensor category $\CatCC$ is called \emph{semi-simple} if all $\CatC_{\alpha\beta}$ are semi-simple. 

A partial tensor category is said to have \emph{indecomposable units} if all units $\Unitb_\alpha$ are indecomposable. 
\end{Def}

It is easy to see that any semi-simple partial tensor category can be
partitioned into a semi-simple partial tensor category with indecomposable units.  Hence we will from now on only consider semi-simple partial tensor categories with indecomposable units.

The following definition introduces the notion of duality for partial tensor categories.

\begin{Def} Let $\CatCC$ be a partial tensor category. 

An object $X\in \CatC_{\alpha\beta}$ is said to admit a \emph{left dual} if there exists an object $Y=\hat{X} \in \CatC_{\beta\alpha}$ and morphisms $\ev_{X}: Y\otimes X \rightarrow \Unitb_\beta$ and $\coev_X: \Unitb_\alpha\rightarrow X\otimes Y$ satisfying the obvious snake identities.

We say $\CatCC$ \emph{admits left duality} if each object of each $\CatC_{\alpha\beta}$ has a left dual.
\end{Def}

Similarly, one defines right duality $X\rightarrow \check{X}$ and (two-sided) duality $X\rightarrow \bar{X}$. As for tensor categories with unit, if $X$ admits a (left or right) dual, it is unique up to isomorphism. 

\begin{Lem}\label{LemMorDua}
\begin{enumerate}[label=(\arabic*)]
\item Let $\CatCC$ be a partial tensor category. If $X$ has left dual $\hat{X}$, then $X$ is a right dual to $\hat{X}$. 
\item Let $F$ be a morphism $\CatCC\rightarrow \CatDD$ based over
  $\phi:\mathscr{J}\rightarrow \mathscr{I}$. If $X\in \CatC_{k'l'}$
  has a left dual $\hat X$, then $F_{lk}(\hat{X})$ is a left dual to $F_{kl}(X)$.
 \end{enumerate}
\end{Lem}
\begin{proof}
We can consider the restriction $\CatCC'$ of $\CatCC$ to any two-element subset $\mathscr{I}'$ of $\mathscr{I}$, and the first property then follows from the usual argument inside the global (unital) tensor category $\CatCC'$. For the second property, consider also the associated restriction $\mathscr{D}'$ to $\phi^{-1}(\mathscr{I})$. We can then again apply the usual arguments to the associated global category $\CatC'$ and global unital morphism $F:\CatC'\rightarrow M(\CatD')$ to see that $F(\hat{X})\cong \widehat{F(X)}$. Using that local units are evidently self-dual and that duality behaves anti-multiplicatively w.r.t.~ tensor products, we can cut down with unit objects on both sides to obtain the statement in the lemma.
\end{proof}

A final ingredient which will be needed is an analytic structure on our partial tensor categories.

\begin{Def} A \emph{partial fusion C$^*$-category} is a partial tensor category $(\CatC,\otimes)$ with duality such that all $\CatC_{\alpha\beta}$ are semi-simple C$^*$-categories,  all functors $\otimes$ are $^*$-functors (in the sense that $(f\otimes g)^* = f^*\otimes g^*$ for morphisms), and the associativity and unit constraints are unitary.
\end{Def} 

\begin{Rem}
\begin{enumerate}[label=(\arabic*)]
\item If $\CatCC$ is a partial tensor C$^*$-category, the total category $\CatC$ only has pre-C$^*$-algebras as endomorphism spaces, as the morphisms spaces need not be closed in the C$^*$-norm. On the other hand, $M(\CatC)$ only has $^*$-algebras as endomorphism spaces, since we did not restrict our direct products. 
\item The notion of duality for a partial tensor C$^*$-category is the
  same as in the absence of a C$^*$-structure. However, because of the
  presence of the $^*$-structure, any left dual is automatically a
  two-sided dual, and the dual object of $X$ is then simply denoted by $\overline{X}$.  
\item We slightly abuse the terminology `fusion', as strictly speaking this would require there to be only a finite set of mutually non-equivalent irreducible objects in each $\CatC_{\alpha\beta}$.
\item In the same vein, the total C$^*$-category with local units associated to a partial fusion C$^*$-category could be called a \emph{multiplier fusion C$^*$-category}.
\end{enumerate}
\end{Rem}

\begin{Exa} Let $I$ be a set. Then we can consider the partial fusion C$^*$-category $\CatCC = \{\Hilb_{\fd}\}_{I\times I}$ of finite-dimensional Hilbert spaces, with all $\otimes$ the ordinary tensor product. The associated global category is the category $\Hilbif$ of finite-dimensional bi-graded Hilbert spaces. The dual of a Hilbert space $\Hsp \in \CatC_{kl}$ is just the ordinary dual Hilbert space $\Hsp^* \cong \overline{\Hsp}$, but considered in the category $\CatC_{lk}$. 
\end{Exa}

The notion of a morphism for partial semi-simple tensor C$^*$-categories has to be adapted in the following way.

\begin{Def} Let $\CatCC$ and $\CatDD$ be partial fusion
  C$^*$-categories over respective sets $\mathscr{I}$ and
  $\mathscr{J}$, and let $\phi:\mathscr{J}\rightarrow \mathscr{I}$.  A
  \emph{morphism} from $\CatCC$ to $\CatDD$ (based on $\phi$) is a
  $\phi$-based morphism $(F,\iota,\mu)$ from $\CatCC$ to $\CatDD$ as
  partial tensor categories, with the added requirement that all
  $F_{kl}$ are $^*$-functors and all $\iota$- and $\mu$-maps are
  isometric.
\end{Def} 

\begin{Rem} If a morphism of partial fusion C$^*$-categories is based over a \emph{surjective} map $\varphi: \mathscr{J}\rightarrow \mathscr{I}$, then it is automatically faithful. Indeed, by semi-simplicity a non-faithful morphism would send some irreducible object to zero. However, by the duality assumption this would mean that some irreducible unit is sent to zero, which is excluded by surjectivity of $\varphi$ and the definition of morphism.
\end{Rem}

\section{Representations of partial compact quantum groups}

In this section, the representation theory of partial compact quantum
groups is investigated.

\subsection{Corepresentations of partial bialgebras}

Let $\mathscr{A}$ be an $I$-partial bialgebra. We will now write its
homogeneous components in the form $A(K) = \eGr{A}{k}{l}{m}{n}
\Gr{A}{k}{l}{m}{n}$.

We denote by $\Hom_\C(V,W)$  the vector space of linear
maps between two vector spaces $V$ and $W$.

Let $I$ be a set. As in Example \ref{ExaVectBiGr}, an $I^{2}$-graded
vector space $V=\bigoplus_{k,l\in I} \Gru{V}{k}{l}$ will be called
\emph{row-and column finite-dimensional} (rcfd) if the $\oplus_l
V_{kl}$ (resp.~ $\oplus_k V_{kl}$) are finite-dimensional for each $k$
(resp.~ $l$) fixed, and $\Vectrcf$ denotes the category whose objects are rcfd
$I^{2}$-graded vector spaces. Morphisms are linear maps $T$ that
preserve the grading and therefore can be written $T=\prod_{k,l\in I}
\Gru{T}{k}{l}$.

\begin{Def} \label{definition:corep} Let $\mathscr{A}$ be an
  $I$-partial bialgebra and let $V=\bigoplus_{k,l} \Gru{V}{k}{l}$
   be
an rcfd $I^{2}$-graded vector space.  A \emph{corepresentation}
  $\mathscr{X}=(\Gr{X}{k}{l}{m}{n})_{k,l,m,n}$ of $\mathscr{A}$ on $V$
  is a family of elements
 \begin{align} \label{eq:rep-blocks}
   \Gr{X}{k}{l}{m}{n} \in \Gr{A}{k}{l}{m}{n} \otimes
  \Hom_\C(\Gru{V}{m}{n},\Gru{V}{k}{l})
 \end{align}
 satisfying 
 \begin{align}
   \label{eq:rep-comultiplication}
    (\Delta_{pq} \otimes
    \id)(\Gr{X}{k}{l}{m}{n}) &=
    \Big{(}\Gr{X}{k}{l}{p}{q}\Big{)}_{13}\Big{(}\Gr{X}{p}{q}{m}{n}\Big{)}_{23},
    \\ \label{eq:rep-counit}
(\epsilon \otimes
  \id)(\Gr{X}{k}{l}{m}{n})&=\delta_{k,m}\delta_{l,n}\id_{\Gru{V}{k}{l}}
 \end{align}
  for all possible indices. We also call $(V,\mathscr{X})$ a
  \emph{corepresentation}.
\end{Def}
Here, we use here the standard leg numbering notation, e.g.~ $a_{23}=1\otimes a$.
\begin{Exa} \label{example:rep-triv} Equip the vector space
  $\C^{(I)}=\bigoplus_{k\in I} \C$ with the diagonal
  $I^{2}$-grading. Then the family $\mathscr{U}$ given by
  \begin{align} \label{eq:rep-triv}
    \Gr{U}{k}{l}{m}{n} = \delta_{k,l}\delta_{m,n} \UnitC{k}{m} \in
    \Gr{A}{k}{l}{m}{n}
  \end{align}
is a corepresentation of $\mathscr{A}$ on $\C^{(I)}$. We call it the
\emph{trivial corepresentation}.
\end{Exa}

\begin{Exa} \label{example:rep-regular}
  Assume  given an rcfd family of subspaces
  \begin{align*}
    \Gru{V}{m}{n} \subseteq \bigoplus_{k,l} \Gr{A}{k}{l}{m}{n}
  \end{align*}
  satisfying
  \begin{align} \label{eq:rep-regular-inclusion}
    \Delta_{pq}(\Gru{V}{m}{n}) &\subseteq \Gru{V}{p}{q} \otimes
    \Gr{A}{p}{q}{m}{n}.
  \end{align}
Then the  elements $\Gr{X}{k}{l}{m}{n} \in \Gr{A}{k}{l}{m}{n} \otimes
  \Hom_{\C}(\Gru{V}{m}{n},\Gru{V}{k}{l})$ defined by 
  \begin{align*}
    \Gr{X}{k}{l}{m}{n}(1 \otimes b) &= \Delta^{\op}_{kl}(b) \in
    \Gr{A}{k}{l}{m}{n} \otimes \Gru{V}{k}{l} \quad
    \text{for all } b\in \Gru{V}{m}{n}
  \end{align*}
  form a corepresentation $\mathscr{X}$ of $\mathscr{A}$ on
  $V$. Indeed, 
  \begin{align*}
    (\Delta_{pq} \otimes \id)(\Gr{X}{k}{l}{m}{n})(1 \otimes 1 \otimes
    b) &=(\Delta_{pq}\otimes \id)(\Delta^{\op}_{kl}(b)) =
    \Big{(}\Gr{X}{k}{l}{p}{q}\Big{)}_{13}\Big{(}\Gr{X}{p}{q}{m}{n}\Big{)}_{23}(1
    \otimes 1 \otimes b), \\
    (\epsilon \otimes \id)(\Gr{X}{k}{l}{m}{n})b &= (\epsilon \otimes
    \id)(\Delta^{\op}_{kl}(b)) = \delta_{k,m}\delta_{l,n}b
  \end{align*}
  for all $b\in \Gru{V}{m}{n}$.  We call $\mathscr{X}$ the
  \emph{regular corepresentation on $V$}. 
\end{Exa}

Morphisms of corepresentations are defined as follows.
\begin{Def}
  Let $\mathscr{A}$ be an $I$-partial bialgebra.  A \emph{morphism}
  $T$ between  corepresentations
  $(V,\mathscr{X})$ and $(W,\mathscr{Y})$ of $\mathscr{A}$ is a family
  of linear maps
  \[\Gru{T}{k}{l} \in
  \Hom_\C(\Gru{V}{k}{l},\Gru{W}{k}{l})\] satisfying \[(1 \otimes
  \Gru{T}{k}{l})\Gr{X}{k}{l}{m}{n} = \Gr{Y}{k}{l}{m}{n}(1 \otimes
  \Gru{T}{m}{n})\]
\end{Def}

We denote the category of all corepresentations of $\mathscr{A}$ on rcfd $I^2$-graded vector spaces by $\Corep_{\rcf}(\mathscr{A})$.

We next consider the total form of a corepresentation.

Let $\mathscr{A}$ be a partial bialgebra with total algebra $A$, and
let $V$ be an rcfd $I^{2}$-graded vector space.
Denote by $\lambda^{V}_{k},\rho^{V}_{l} \in \Hom_{\C}(V)$ the
projections onto the summands $\Gru{V}{k}{} = \bigoplus_{q}
\Gru{V}{k}{q}$ and $\Gru{V}{}{l}=\bigoplus_{p}\Gru{V}{p}{l}$
respectively, and identify $\Hom_{\C}(\Gru{V}{m}{n},\Gru{V}{k}{l})$ with
$\lambda^{V}_{k}\rho^{V}_{l}\Hom_{\C}(V)\lambda^{V}_{m}\rho^{V}_{n}$. Denote by $\Hom_{\C}^{0}(V) \subseteq \Hom_{\C}(V)$ the algebraic sum of all
these subspaces. Then we can define a homomorphism
\begin{align*}
  \Delta \otimes \id \colon M(A \otimes \Hom_{\C}^{0}(V)) \to M(A
  \otimes A \otimes \Hom_{\C}^{0}(V))
\end{align*}
similarly as we defined $ \Delta \colon A \to M(A\otimes A)$.
\begin{Lem} \label{lemma:rep-multiplier}
  Let  $\mathscr{A}$ be an $I$-partial bialgebra and $V$  an rcfd $I^{2}$-graded vector space.  If $\mathscr{X}$ is a
  corepresentation of  $\mathscr{A}$ on $V$, then the sum
  \begin{align}
    \label{eq:rep-multiplier}
  X:=\sum_{k,l,m,n} \Gr{X}{k}{l}{m}{n} \in  M(A
  \otimes \Hom_{\C}^{0}(V))
  \end{align}
 converges strictly and satisfies the following conditions:
  \begin{enumerate}[label=(\arabic*)] 
  \item\label{repma} $(\lambda_{k}\rho_{m} \otimes \id){X}(\lambda_{l}\rho_{n}
    \otimes \id) = (1 \otimes \lambda^{V}_{k}\rho^{V}_{l}){X}(1 \otimes
    \lambda^{V}_{m}\rho^{V}_{n}) = \Gr{X}{k}{l}{m}{n}$,
  \item\label{repmb} $(A \otimes 1){X}$, $ {X}(A \otimes 1)$ and $(1 \otimes
    \Hom^{0}_{\C}(V))X(1 \otimes \Hom^{0}_{\C}(V))$ lie in $A \otimes \Hom_{\C}^{0}(V)$,
  \item\label{repmc} $(\Delta\otimes \id)(X)=X_{13}X_{23}$, 
  \item\label{repmd} the sum $(\epsilon \otimes \id)({X}) :=\sum (\epsilon \otimes
    \id)(\Gr{X}{k}{l}{m}{n})$ converges in $M(\Hom^{0}_{\C}(V))$ strictly
    to $\id_{V}$.
  \end{enumerate}
  Conversely, if $ X \in M(A \otimes \Hom_{\C}^{0}(V))$ satisfies
  \ref{repma}--\ref{repmd} with $\Gr{X}{k}{l}{m}{n}$ defined by \ref{repma}, then
  $\mathscr{X}=(\Gr{X}{k}{l}{m}{n})_{k,l,m,n}$ is a corepresentation
  of $\mathscr{A}$ on $V$.
\end{Lem}
\begin{proof}
 Straightforward.
\end{proof}

\begin{Def} If $\mathscr{X}$ and $X$ are as in Lemma \ref{lemma:rep-multiplier}, we will call $X$ the \emph{corepresentation multiplier} of $\mathscr{X}$. 
\end{Def}

Let us relate the notion of a corepresentation multiplier to the
notion of a full comodule for a weak multiplier bialgebra introduced in \cite[Definition 2.2 and Definition 4.2]{Boh2}. Recall first from \cite[Theorem 4.5]{Boh2} that if $(A,\Delta)$ is a weak multiplier bialgebra, then any full comodule over $A$ carries the structure of a firm bimodule over the base algebra. In particular, if $A$ arises from a partial bialgebra, any comodule is bigraded over the object set. 

\begin{Prop} Let $\mathscr{A}$ be a partial bialgebra with corepresentation $X$ on $V = \oplus_{m,n} \Gru{V}{m}{n}$. Then the couple \[\lambda_X: V\otimes A \rightarrow V\otimes A,\quad v\otimes a \rightarrow X_{21}(v\otimes a),\] 
\[\rho_X: V\otimes A\rightarrow V\otimes A,\quad v\otimes a \mapsto (1\otimes a)X_{21}(v\otimes 1)\] is well-defined and defines a full comodule for the weak multiplier bialgebra $(A,\Delta)$. Conversely, any full comodule which is rcfd for the induced bigrading arises in this way.
\end{Prop} 

\begin{proof} Well-definedness of the couple $(\lambda_X,\rho_X)$ is immediate from the local support condition, and it is clear then that $(1\otimes a)\lambda_X(v\otimes b) = \rho_X(v\otimes a)(1\otimes b)$.  The conditions (2.11) and (2.12) in \cite[Definition 2.12]{Boh2} are then easily seen to follow from the identity $(\Delta\otimes \id)(X) = X_{13}X_{23}$. Finally, as for $v\in \Gru{V}{m}{n}$ one has $(\id\otimes \varepsilon)(X_{21}(v\otimes \UnitC{n}{n}) = (\varepsilon\otimes \id)(\Gr{X}{m}{n}{m}{n}) v  = v$, it follows that $(\lambda_X,\rho_X)$ is full.

Assume now conversely that $(\lambda,\rho)$ defines a full comodule structure on $V = \oplus_{m,n} \Gru{V}{m}{n}$. From the definition of the grading, it follows that we obtain maps \[\Gru{V}{m}{n}\rightarrow \Gru{V}{k}{l}\otimes \Gr{A}{k}{l}{m}{n},\quad v \mapsto (1\otimes \UnitC{k}{m})\lambda(v\otimes \UnitC{l}{n}).\] As the $\Gru{V}{m}{n}$ are finite-dimensional, there hence exists $\Gr{X}{k}{l}{m}{n} \in \Gr{A}{k}{l}{m}{n}\otimes \Hom_{\C}(\Gru{V}{m}{n},\Gru{V}{k}{l})$ such that  
\[(\Gr{X}{k}{l}{m}{n})_{21}(v\otimes 1)=(1\otimes \UnitC{k}{m})\lambda(v\otimes \UnitC{l}{n}).\] As $(\lambda,\rho)$ form a multiplier, it is then moreover immediate that in fact \[(\Gr{X}{k}{l}{m}{n})_{21}(v\otimes a) = (1\otimes \UnitC{k}{m})\lambda(v\otimes a).\] From (2.12) in \cite[Definition 2.12]{Boh2}, it is then immediate that the $\Gr{X}{k}{l}{m}{n}$ satisfy \eqref{eq:rep-comultiplication}. Moreover, from the proof of \cite[Theorem 4.5]{Boh2} it follows that for $v\in \Gru{V}{m}{n}$, one has \[v = (\id\otimes \epsilon)(X_{21}(v\otimes \UnitC{n}{n}),\] hence \eqref{eq:rep-counit} holds, and $\mathscr{X}$ forms a corepresentation. 
\end{proof}

We present some more general constructions for corepresentations of partial bialgebras. Given an rcfd $I^{2}$-graded vector space $V=\bigoplus_{k,l} \Gru{V}{k}{l}$
and a family of subspaces $\Gru{W}{k}{l} \subseteq \Gru{V}{k}{l}$, we
denote by $\iota^{W}\colon W\to V$ and $\pi^{W} \colon V \to
V/W=\bigoplus_{k,l} \Gru{V}{k}{l}/\Gru{W}{k}{l}$ the embedding and the
quotient map.
\begin{Def} Let $(V,\mathscr{X})$ be a
  corepresentation of a partial bialgebra $\mathscr{A}$.  We call a
  family of subspaces $\Gru{W}{k}{l} \subseteq \Gru{V}{k}{l}$
  \emph{invariant (w.r.t.\ $\mathscr{X}$)} if
 \begin{align} \label{eq:rep-invariant} (1\otimes
   \Gr{\pi}{}{W}{k}{l})\Gr{X}{k}{l}{m}{n}(1 \otimes
   \Gr{\iota}{}{W}{m}{n}) =0.
  \end{align}
We call $(V,\mathscr{X})$ 
 \emph{irreducible} if the only invariant families of subspaces are
 $(0)_{k,l}$ and $(\Gru{V}{k}{l})_{k,l}$.
\end{Def}

The next lemmas deal with restriction, factorisation and Schur's lemma. We skip their proofs which are straightforward.

\begin{Lem}
  Let $(V,\mathscr{X})$ be a corepresentation
  of a partial bialgebra and let $\Gru{W}{k}{l}
  \subseteq \Gru{V}{k}{l}$ be an invariant family of subspaces. Then
  there exist unique  corepresentations
  $(W,(\iota^{W})^{*}\mathscr{X})$ and $(V/W,\pi^{W}_{*}\mathscr{X})$ 
  such that $\iota^{W}$  and  $\pi^{W}$  are  morphisms  $(W,(\iota^{W})^{*}\mathscr{X}) \to (V,\mathscr{X}) \to (V/W,\pi^{W}_{*}\mathscr{X})$.
\end{Lem}

\begin{Lem} Let $T$ be a morphism of 
  corepresentations $(V,\mathscr{X})$ and $(W,\mathscr{Y})$ of a
  partial bialgebra. Then the families of subspaces $\ker
  \Gru{T}{k}{l} \subseteq \Gru{V}{k}{l}$ and $\img\Gru{T}{k}{l}
  \subseteq \Gru{W}{k}{l}$ are invariant.  In particular, if
  $(V,\mathscr{X})$ and $(W,\mathscr{Y})$ are irreducible, then either
  all $\Gru{T}{k}{l}$ are zero or all $\Gru{T}{k}{l}$ are
  isomorphisms.
\end{Lem}

Given corepresentations $\mathscr{X}$ and $\mathscr{Y}$ of
a partial bialgebra $\mathscr{A}$ on respective rcfd $I^{2}$-graded vector spaces $V$ and $W$,
we  obtain an $I^{2}$-graded vector space $V\oplus W$ by taking
component-wise direct sums, and use the canonical embedding 
\begin{align*}
  \Hom(\Gru{V}{m}{n},\Gru{V}{k}{l}) \oplus
  \Hom(\Gru{W}{m}{n},\Gru{W}{k}{l}) \hookrightarrow
  \Hom(\Gru{V}{m}{n} \oplus \Gru{W}{m}{n},\Gru{V}{k}{l} \oplus
  \Gru{W}{k}{l})
\end{align*}
to define the \emph{direct sum} $\mathscr{X} \oplus \mathscr{Y}$,
which is a corepresentation of $\mathscr{A}$ on $V\oplus W$. Then the
natural embeddings from $V$ and $W$ into $V\oplus W$ and the
projections onto $V$ and $W$ are evidently morphisms of
corepresentations.  More generally, given a family of  corepresentations
$((V_{\alpha},\mathscr{X}_{\alpha}))_{\alpha}$ such that the sum
$\bigoplus_{\alpha} V_{\alpha}$ is rcfd again, one
can form the direct sum $\bigoplus_{\alpha} \mathscr{X}_{\alpha}$,
which is a corepresentation on $\bigoplus_{\alpha} V_{\alpha}$.
\begin{Prop}
  Let $\mathscr{A}$ be an $I$-partial bialgebra. Then $\Corep_{\rcf}(\mathscr{A})$
  is a $\C$-linear abelian category, and the forgetful functor
  $\Corep_{\rcf}(\mathscr{A}) \to \Vectrcf$ lifts kernels, cokernels and biproducts.
\end{Prop}
\begin{proof}
  The preceding considerations show that the forgetful functor lifts
  kernels, cokernels and biproducts. Moreover, in
  $\Corep_{\rcf}(\mathscr{A})$, every monic is a kernel
  and every epic is a cokernel because the same is true in $\Vectrcf$
  and because kernels and cokernels lift.
\end{proof}

\subsection{Corepresentations of partial Hopf algebras}

If $\mathscr{A}$ is a partial Hopf algebra,  then every
corepresentation multiplier has a generalized inverse.
\begin{Lem} \label{lemma:rep-invertible}
  Let $(V,\mathscr{X})$ be a corepresentation of a partial Hopf
  algebra $\mathscr{A}$. Then with $\Gr{Z}{k}{l}{m}{n} = (S\otimes \id)(\Gr{X}{n}{m}{l}{k})$, we have $\Gr{Z}{k}{l}{m}{n}\in \Gr{A}{k}{l}{m}{n}\otimes \Hom_{\C}(\Gru{V}{l}{k},\Gru{V}{n}{m})$ and
  \begin{align*}
    \Gr{X}{k}{l}{m}{n}  \cdot \Gr{Z}{l}{k'}{n}{m'} &=0 \text{ if } m'\neq m &
      \sum_{n} \Gr{X}{k}{l}{m}{n} \cdot \Gr{Z}{l}{k'}{n}{m} &= \delta_{k,k'}\UnitC{k}{m} \otimes
      \id_{\Gru{V}{k}{l}}, \\
      \Gr{Z}{n}{m}{l}{k}\cdot \Gr{X}{m}{n'}{k}{l'} &= 0
      \text{ if } n\neq n' & 
      \sum_{m} \Gr{Z}{n}{m}{l}{k}\cdot \Gr{X}{m}{n}{k}{l'} &=
      \delta_{l,l'} \UnitC{n}{l} \otimes \id_{\Gru{V}{k}{l}}.
  \end{align*}
  In particular, the multiplier $Z:=     (S \otimes
  \id)(X) \in M(A \otimes \Hom_{\C}^{0}(V))$
  satisfies
  \begin{align} \label{eq:rep-generalized-inverse}
    XZ &= \sum_{k} \lambda_{k} \otimes \lambda^{V}_{k}, &
    ZX &= \sum_{l} \rho_{l} \otimes \rho^{V}_{l},
  \end{align}
  and is a generalized inverse of $X$ in the sense that $XZX=X$ and $ZXZ=Z$.
\end{Lem}
\begin{proof}
  The grading property of $\Gr{Z}{k}{l}{m}{n}$ follows from  $S(\Gr{A}{p}{q}{r}{s})\subseteq \Gr{A}{s}{r}{q}{p}$, and then the upper left hand identity is immediate.  To
  verify the upper right hand one, we use identities \eqref{eq:rep-comultiplication}, \eqref{eq:rep-counit} and \eqref{eq:antipode-pi-l}. Namely, with $M_{A}$ denoting the multiplication of $A$, we find
  \begin{align*}
      \sum_{n} \Gr{X}{k}{l}{m}{n} \cdot (S \otimes
      \id)(\Gr{X}{m}{n}{k'}{l}) &= \sum_{n} (M_{A}  (\id \otimes S)
      \otimes \id)((\Gr{X}{k}{l}{m}{n})_{13}(\Gr{X}{m}{n}{k'}{l})_{23})
 \\ &= \sum_{n} (M_{A} (\id \otimes S)  \Delta_{m,n} \otimes
      \id)(\Gr{X}{k}{l}{k'}{l}) \\
      &= \delta_{k,k'} \UnitC{k}{l} \otimes (\epsilon \otimes
      \id)(\Gr{X}{k}{l}{k'}{l})
      \\ &=
\delta_{k,k'}\UnitC{k}{m} \otimes
      \id_{\Gru{V}{k}{l}}. \end{align*} The other
equations follow similarly, and the assertions concerning $Z$ are
direct consequences.
\end{proof}
\begin{Def}
  Let $\mathscr{X}$ be a corepresentation of a  partial Hopf
  algebra.  We  denote the generalized inverse $(S \otimes \id)(X)$
  of $X$  by $X^{-1}$ and let
  \begin{align*}
   \Gr{(X^{-1})}{k}{l}{m}{n}=(S \otimes \id)(\Gr{X}{n}{m}{l}{k}) \in
   \Gr{A}{k}{l}{m}{n} \otimes \Hom_{\C}(\Gru{V}{l}{k},\Gru{V}{n}{m})
  \end{align*}
\end{Def}
For completeness, we mention the following converse to Lemma \ref{lemma:rep-invertible}.
\begin{Lem}
  Let $\mathscr{A}$ be an $I$-partial bialgebra, $V$ an rcfd
  $I^{2}$-graded vector space and $X,Z \in M(A \otimes
  \Hom_{\C}^{0}(V))$. If conditions \ref{repma}--\ref{repmc} in Lemma
  \ref{lemma:rep-multiplier} and \eqref{eq:rep-generalized-inverse}
  hold, then the corresponding family
  $\mathscr{X}=(\Gr{X}{k}{l}{m}{n})_{k,l,m,n}$ is a corepresentation
  of $\mathscr{A}$ on $V$.
\end{Lem}
\begin{proof}
  We have to verify condition \ref{repmd} in Lemma
  \ref{lemma:rep-multiplier}.  If $(k,l) \neq (p,q)$, then
  $\epsilon(\Gr{A}{k}{l}{p}{q})=0$ and hence $(\epsilon
  \otimes \id)(\Gr{X}{k}{l}{p}{q}) =0$. The counit property and condition
  \ref{repmc} in Lemma \ref{lemma:rep-multiplier} imply 
\begin{align*}
  \Gr{X}{k}{l}{m}{n} &= ((\epsilon\otimes \id)  \Delta \otimes
  \id)(  \Gr{X}{k}{l}{m}{n}) 
\\ &  = \sum_{p,q} (\epsilon\otimes \id \otimes
  \id)\left((\Gr{X}{k}{l}{p}{q})_{13}(\Gr{X}{p}{q}{m}{n})_{23}\right)
  =  (1 \otimes \Gru{T}{k}{l})\Gr{X}{k}{l}{m}{n},
\end{align*}
where $\Gru{T}{k}{l}=(\epsilon \otimes \id)(\Gr{X}{k}{l}{k}{l}) \in
\Hom_{\C}(\Gru{V}{k}{l})$.  Therefore,  $T=\prod_{k,l} T_{k,l}$  satisfies $(1 \otimes T)X =
X$. Multiplying on the right by $Z$, we find
$T\lambda^{V}_{k}=\lambda^{V}_{k}$ for all $k$. Thus, $T=\id_{V}$.
\end{proof}

\begin{Lem} \label{lemma:rep-total-morphism}
 A bigraded map $T$ defines a morphism from 
    $(V,\mathscr{X})$ to $(W,\mathscr{Y})$ if and only if one of the following relations hold:
    \begin{align*}
      Y^{-1}(1 \otimes T)X&=\sum_{m,n} \rho_{n} \otimes \Gru{T}{m}{n},
      &
    Y(1\otimes T)X^{-1} &=\sum_{k,l} \lambda_{k} \otimes \Gru{T}{k}{l}.
    \end{align*}
\end{Lem}

\subsection{Tensor product and duality}

Recall from Example \ref{ExaVectBiGr} that the category $\Vectrcf$ is a tensor category. The tensor product of morphisms is the
restriction of the ordinary tensor product.  We will interpret this product as being strictly associative.  The unit for this product is the vector
space $\C^{(I)}=\bigoplus_{k\in I} \C$. 

Given $V$ and $W$ in $\Vectrcf$, we identify $\Hom_\C(\Gru{V}{m}{n},\Gru{V}{k}{l})\otimes
   \Hom_\C(\Gru{W}{n}{q},\Gru{W}{l}{p})$ with a subspace of
\begin{align*}
   \Hom_\C(\Gru{V}{m}{n}\otimes
   \Gru{W}{n}{q},\Gru{V}{k}{l}\otimes \Gru{W}{l}{p})\subseteq
   \Hom_\C(\Gru{(}{m}{}V\itimes
     W\Gru{)}{}{q},\Gru{(}{k}{}V\itimes W\Gru{)}{}{p}).
\end{align*}

We can now construct a product of corepresentations as follows.
\begin{Lem} Let $\mathscr{X}$ and $\mathscr{Y}$ be copresentations of
  $\mathscr{A}$ on respective  rcfd $I^{2}$-graded vector spaces $V$ and
  $W$. Then the sum
  \begin{align} \label{eq:rep-product-blocks}
     \Gr{(X\Circt Y)}{k}{p}{m}{q} := \sum_{l,n}
    \left(\Gr{X}{k}{l}{m}{n}\right)_{12}\left(\Gr{Y}{l}{p}{n}{q}\right)_{13}
  \end{align}
  has only finitely many non-zero terms, and the elements
 \[\Gr{(X\Circt
    Y)}{k}{p}{m}{q}\in \Gr{A}{k}{p}{m}{q} \otimes
  \Hom_\C(\Gru{(}{m}{}V\itimes W\Gru{)}{}{q},\Gru{(}{k}{}V\itimes W\Gru{)}{}{p})
\]
define a  corepresentation $\mathscr{X} \Circt \mathscr{Y}$ of
$\mathscr{A}$ on $V\itimes W$. 
\end{Lem} 
\begin{proof}
  The sum \eqref{eq:rep-product-blocks} is finite because $V$ and
  $W$ are  rcfd. Using the identification above, we
  see that
 \[
  \left(\Gr{X}{k}{l}{m}{n}\right)_{12}\left(\Gr{Y}{l}{p}{n}{q}\right)_{13}\in \Gr{A}{k}{p}{m}{q} \otimes \Hom_\C(\Gru{(}{m}{}V\itimes
    W\Gru{)}{}{q},\Gru{(}{k}{}V\itimes W\Gru{)}{}{p}).\] Now,   the fact that $\Gr{(X\Circt
    Y)}{k}{p}{m}{q}$ is a corepresentation follows easily
  from the multiplicativity of $\Delta$ and the weak multiplicativity
  of $\epsilon$.
\end{proof}
\begin{Rem} \label{remark:rep-tensor-multiplier}
  The corepresentation multiplier associated to $\mathscr{X}\Circt
  \mathscr{Y}$   is  just $X_{12}Y_{13}$.
\end{Rem}

\begin{Prop} \label{prop:rep-tensor} Let $\mathscr{A}$ be an
  $I$-partial bialgebra. Then  $\Corep_{\rcf}(\mathscr{A})$ carries the
  structure of a strict tensor category such that the product of  corepresentations $(V,\mathscr{X})$ and
  $(W,\mathscr{Y})$ is the corepresentation $(V\itimes
  W,\mathscr{X}\Circt \mathscr{Y})$, the unit is the trivial
  corepresentation $(\C^{(I)},\mathscr{U})$, and the forgetful functor
  $\Corep_{\rcf}(\mathscr{A}) \to \Vectrcf$ is a strict tensor functor.
\end{Prop}
\begin{proof}
This is clear.
\end{proof}

Given a  corepresentation of a partial Hopf algebra, one can use the
antipode to define a contragredient corepresentation on a dual space.
Denote the dual of vector spaces $V$ and  linear maps $T$ by
$\dual{V}$ and $\dualop{T}$, respectively, and define the dual of an
$I^{2}$-graded vector space $V=\bigoplus_{k,l} \Gru{V}{k}{l}$ to be
the space
\begin{align*}
  \dual{V}=\bigoplus_{k,l} \Gru{(\dual{V})}{k}{l}, \quad \text{where }
\Gru{(\dual{V})}{k}{l} = \dual{(\Gru{V}{l}{k})}.
\end{align*}

\begin{Prop}
  Let $\mathscr{A}$ be an $I$-partial Hopf algebra with antipode $S$
  and  let $(V,\mathscr{X})$ be a 
  corepresentation of $\mathscr{A}$. Then $\dual{V}$ and the family
  $\dualco{\mathscr{X}}$ given by
   \begin{align} \label{eq:rep-left-dual}
\Gr{\dualco{X}}{k}{l}{m}{n}   :=  (S \otimes \dualop{-})(\Gr{X}{n}{m}{l}{k}) 
   \end{align} 
   form a  corepresentation of $\mathscr{A}$  which is a left dual of $(V,\mathscr{X})$. If the antipode
   $S$ of $\mathscr{A}$ is bijective, then $\dual{V}$ and the family
   $\dualcor{\mathscr{X}}$ given by 
   \begin{align} \label{eq:rep-right-dual}
 \Gr{\dualcor{X}}{k}{l}{m}{n} :=(S^{-1}
   \otimes \dualop{-})(\Gr{X}{n}{m}{l}{k})    
   \end{align}
 form a  corepresentation
 of $\mathscr{A}$ which is a
   right dual of $(V,\mathscr{X})$.
  \end{Prop}
  \begin{proof}
    We only prove the assertion concerning
    $(\dual{V},\dualco{\mathscr{X}})$. To see that this is a corepresentation, note that the element
    \eqref{eq:rep-left-dual} belongs to $\Gr{A}{k}{l}{m}{n} \otimes
    \Hom_{\C}(\Gru{(\dual{V})}{m}{n},\Gru{(\dual{V})}{k}{l})$ and use
    the relations $\Delta \circ S = (S \otimes S)\Delta^{\op}$ and
    $\epsilon \circ S = \epsilon$ from Corollary
    \ref{corollary:antipode} and Lemma \ref{LemCoAnt}.  
    Let us show that $(\dual{V},\dualco{\mathscr{X}})$ is a left dual
    of $(V,\mathscr{X})$.

    Given a finite-dimensional vector space $W$, denote by $\ev_{W}
    \colon \dual{W} \otimes W \to \C$ the evaluation map and by $\coev_{W}
    \colon \C \to W \otimes \dual{W}$ the coevaluation map, given by
    $1\mapsto \sum_{i} w_{i} \otimes \dual{w_{i}}$ if $(w_{i})_{i}$
    and $(\dual{w_{i}})_{i}$ are dual bases of $W$ and
    $\dual{W}$. With respect to these maps, $\dual{W}$ is a left dual
    of $W$. If $F\colon W_{1}\to W_{2}$ is a linear map between
    finite-dimensional spaces, then
\begin{align} \label{eq:coev-vee} (\id_{W_{2}} \otimes F^{\tr}) \circ \coev_{W_{2}} &= (F \otimes \id_{W_{1}^{*}})\circ
  \coev_{W_{1}}, &
\ev_{W_{1}}(F^{\tr}
  \otimes \id_{W_{2}})&=  \ev_{W_{2}}(\id_{W_{2}^{*}} \otimes F).
\end{align}

Now, define morphisms $\coev \colon \C^{(I)} \to V\itimes \dual{V}$ and
$\ev \colon \dual{V} \itimes V \to \C^{(I)}$ by
\begin{align*}
  \Gru{\coev}{k}{l} &= \delta_{k,l} \sum_{p} \coev_{\tiny\Gru{V}{k}{p}} \colon
  \C \to 
    \Gru{(}{k}{}V\itimes \dual{V}\Gru{)}{}{l}, &
  \Gru{\ev}{k}{l} &= \delta_{k,l} \sum_{p} \ev_{\Gru{V}{p}{k}} \colon
    \Gru{(}{k}{}V\itimes \dual{V}\Gru{)}{}{l} \to \C.
\end{align*}
One easily checks that with respect to these maps, $\dual{V}$ is a
left dual of $V$ in $\Vectrcf$. 

We therefore only need to show that $\ev$ is a morphism from
$\dualco{\mathscr{X}}\Circt\mathscr{X}$ to $\mathscr{U}$ and that $\coev$ is
a morphism from $\mathscr{U}$ to
$\mathscr{X}\Circt\dualco{\mathscr{X}}$.  But \eqref{eq:coev-vee} and
Lemma \ref{lemma:rep-invertible} imply
  \begin{align*}
    (1\otimes \Gru{\ev}{k}{k})
 \sum_{l,n}  \big(
\Gr{\dualco{X}}{k}{l}{m}{n}\big)_{12}
\big(\Gr{X}{l}{k}{n}{q}\big)_{13} &=
    (1\otimes \Gru{\ev}{k}{k})
 \sum_{l,n} 
(S \otimes \dualop{-})(\Gr{X}{n}{m}{l}{k})_{12}
    (\Gr{X}{l}{k}{n}{q})_{13} \\ &=
(1\otimes \Gru{\ev}{m}{m})  \sum_{l,n}
      (S \otimes \id)(\Gr{X}{n}{m}{l}{k})_{13}(\Gr{X}{l}{k}{n}{q})_{13} \\
    &= \delta_{m,q}\UnitC{k}{q}\otimes \Gru{\ev}{m}{m} \\
    &= \Gr{U}{k}{k}{m}{q}(1 \otimes \Gru{\ev}{m}{m}).
  \end{align*}
A similar  calculation shows that also $\coev$ is a morphism, whence the claim follows.
\end{proof}
\begin{Cor} \label{cor:rep-tensor-duality}
  Let $\mathscr{A}$ be a partial Hopf algebra. Then
  $\Corep_{\rcf}(\mathscr{A})$ is a tensor category with left
  duals and, if the antipode of $\mathscr{A}$ is invertible, with right duals.
\end{Cor}

Let $\mathscr{A}$ be an $I$-partial Hopf algebra.  Then the tensor
unit in $\Corep_{\rcf}(\mathscr{A})$, which is the trivial corepresentation
$\mathscr{U}$ on $\C^{(I)}$, need not be irreducible. Instead, it decomposes
into irreducible corepresentations indexed by the hyperobject set $\mathscr{I}$ of equivalence
classes for the relation $\sim$ on $I$ given by  $k \sim l \iff
  \UnitC{k}{l}\neq 0$ (see Remark
\ref{remark:index-equivalence}).
\begin{Lem}
  Let $\mathscr{A}$ be an $I$-partial Hopf algebra and let
  $(I_{\alpha})_{\alpha\in \mathscr{I}}$ be a labelled partition of $I$ into
  equivalence classes for the relation $\sim$.  Then for each $\alpha\in \mathscr{I}$, the subspace
  $\C^{(I_{\alpha})} \subseteq \C^{(I)}$ is invariant, and the restriction
  $\mathscr{U_{\alpha}}$ of $\mathscr{U}$ to $\C^{(I_{\alpha})}$ is
  irreducible. In particular, $\mathscr{U}=\bigoplus_{\alpha\in\mathscr{I}}
  \mathscr{U_{\alpha}}$ is a decomposition into irreducible corepresentations.
\end{Lem}
\begin{proof}
Immediate from the fact that $\Gr{U}{k}{k}{m}{m} = 
  \UnitC{k}{m}$  is $1$  if $k\sim m$  and $0$ if $k\not\sim m$. 
\end{proof}

\begin{Def} We denote by $\Corep(\mathscr{A})$ the category of  corepresentations $(V,\mathscr{X})$ for which there exists a finite subset of the hyperobject set $\mathscr{I}$ such that $\Gru{V}{k}{l}=0$ for the equivalence classes of $k,l$ outside this subset.
\end{Def}

It is easily seen that $\Corep(\mathscr{A})$ is then a tensor category with local units indexed by $\mathscr{I}$. We will use the same notation for the associated partial tensor category.

\subsection{Decomposition into irreducibles}

When there is an invariant integral around, one can average morphisms of vector spaces to obtain morphisms of corepresentations. 
\begin{Lem} \label{lem:rep-average}  Let $(V,\mathscr{X})$ and
  $(W,\mathscr{Y})$ be  corepresentations of  a partial
  Hopf algebra $\mathscr{A}$ with an invariant integral $\phi$, and let
  $\Gru{T}{k}{l} \in \Hom_{\C}(\Gru{V}{k}{l},\Gru{W}{k}{l})$ for all $k,l\in I$. Then for each $m,n$ fixed, the families
  \begin{align*}
    \Gr{\check T}{m}{n}{k}{l} &:= (\phi \otimes
    \id)(\Gr{(Y^{-1})}{n}{m}{l}{k}(1\otimes
    \Gru{T}{m}{n})\Gr{X}{m}{n}{k}{l}), \\
    \Gr{\hat T}{m}{n}{k}{l} &:=(\phi \otimes
    \id)(\Gr{Y}{k}{l}{m}{n}(1\otimes
    \Gru{T}{m}{n})\Gr{(X^{-1})}{l}{k}{n}{m})
  \end{align*} 
form  morphisms $\Grd{\check{T}}{m}{n}$ and $\Grd{\hat{T}}{m}{n}$ from $(V,\mathscr{X})$ to $(W,\mathscr{Y})$. 
\end{Lem} 
\begin{proof} Clearly, we may suppose that $T$ is supported only on the component at index $(m,n)$, and we may then drop the upper indices and simply write $\Gru{\check{T}}{k}{l}$ and $\Gru{\hat{T}}{k}{l}$. Then 
 in total form, $\check{T}=(\phi \otimes \id)(Y^{-1}(1 \otimes T)X)$
  and $\hat{T}=(\phi \otimes \id)(Y(1 \otimes T)X^{-1})$.  Now, Lemma
  \ref{lemma:rep-multiplier} and Lemma \ref{lemma:total-integral} 
  imply
  \begin{align*}
    Y^{-1}(1 \otimes \check{T})X &= (\phi \otimes \id \otimes
    \id)((Y^{-1})_{23}(Y^{-1})_{13}(1 \otimes 1
    \otimes T)X_{13}X_{23})  \\
    &= ((\phi \otimes\id)  \Delta  \otimes \id)(Y^{-1}(1 \otimes T)X) \\
    &= \sum_{l} \rho_{l} \otimes (\phi \otimes \id)((\rho_{l} \otimes
    1)Y^{-1}(1 \otimes T)X)  \\
    &= \sum_{k,l} \rho_{l} \otimes \Gru{\check T}{k}{l},
  \end{align*}
  whence $\check{T}$ is a morphism from $\mathscr{X}$ to $\mathscr{Y}$
  by Lemma \ref{lemma:rep-total-morphism}. The assertion for $\hat
  T$ follows similarly.
\end{proof}

\begin{Lem}
  Let $\mathscr{A}$ be an $I$-partial Hopf algebra with an invariant integral $\phi$.
  Let $(V,\mathscr{X})$ be a  corepresentation
  and $\Gru{W}{k}{l} \subseteq \Gru{V}{k}{l}$ an invariant family of
  subspaces. Then there exists an idempotent endomorphism $T$ of
  $(V,\mathscr{X})$ such that $\Gru{W}{k}{l}=\img\Gru{T}{k}{l}$ for
  all $k,l$.
\end{Lem}
\begin{proof}
By a direct sum decomposition, we may assume that $V$ is in a fixed component $\Corep(\mathscr{A})_{\alpha\beta}$. For all $k\in I_{\alpha},l\in I_{\beta}$, choose idempotent endomorphisms $\Gru{T}{k}{l}$ of $\Gru{V}{k}{l}$
  with image $\Gru{W}{k}{l}$.   Let
  $\mathscr{Y}$ be the restriction of $\mathscr{X}$ to $W$.  By Lemma
  \ref{lem:rep-average}, we obtain morphisms $\Grd{\check{T}}{m}{n}$
  of $(V,\mathscr{X})$ to $(W,\mathscr{Y})$, which we can also
  interpret as endomorphisms of $(V,\mathscr{X})$.  Fix $n\in
  I_{\beta}$ and write $\Grd{\check{T}}{}{n} = \sum_m
  \Grd{\check{T}}{m}{n}$ (using column-finiteness of $V$). We claim
  that $W$ is the image of $\Grd{\check{T}}{}{n}$.

    In
  total form, invariance of $W$ implies  \[(1 \otimes T)X(1
  \otimes T)=X(1\otimes T).\] Applying
 $(S \otimes \id)$, we get   \[(1 \otimes T)X^{-1}(1
  \otimes T)=X^{-1}(1\otimes T).\]
Now choose $n\in I_{\beta}$ and write $\Grd{\check{T}}{}{n} = \sum_m \Grd{\check{T}}{m}{n}$, which makes sense because of column-finiteness of $V$. We combine  Lemma
  \ref{lemma:rep-multiplier}, Lemma \ref{lemma:rep-invertible} and
  normalisation of $\phi$, and find
  \begin{align*}
    \Grd{\check{T}}{}{n} T &= (\phi \otimes \id)(X^{-1}(1 \otimes
    \rho_{n}^{V}T)X(1 \otimes T)) \\  &= 
     (\phi \otimes \id)(X^{-1}(1 \otimes
    \rho_{n}^{V})X(1 \otimes T)) \\
    &=
  \sum_l \phi(\UnitC{n}{l}) \rho^{V}_{l}T \\& =T,
  \end{align*}
 since we only have to sum over $l\in I_{\beta}$ as $n \in I_{\beta}$ by assumption. 
 
 Now as $W$ is invariant and $T$ sends $V$ into $W$, we have that
 $\Gr{\check{T}}{}{n}{k}{l}$ sends $\Gru{V}{k}{l}$ into
 $\Gru{W}{k}{l}$. Hence it follows that $\img{\check{T}^{n}}=\img T$,
 and $\check{T}^{n}$ is the desired intertwiner.
\end{proof}

\begin{Cor}  \label{cor:rep-cosemisimple}
  Let $\mathscr{A}$ be a partial Hopf algebra with an invariant integral.  Then
  every  corepresentation of $\mathscr{A}$ decomposes into a (possibly infinite) direct
  sum of irreducible corepresentations.
\end{Cor} 
\begin{proof} 
The preceding lemma shows that  every non-zero corepresentation is either
irreducible or the direct sum of two non-zero corepresentations, and we can apply Zorn's lemma.
\end{proof}

We can now prove that the category $\Corep(\mathscr{A})$ of a partial Hopf algebra with invariant integral is semi-simple, that is, any object is a finite direct sum of irreducible objects. If one allows a more relaxed definition of semisimplicity allowing infinite direct sums, this will be true also for the potentially bigger category $\Corep_{\rcf}(\mathscr{A})$.

We will first state a lemma which will also be convenient at other occasions.

\begin{Lem}\label{LemInjMor}  Let $\mathscr{A}$ be a partial Hopf algebra and fix $\alpha,\beta$ in the hyperobject set.  Then if $T$ is a morphism in $\Corep(\mathscr{A})_{\alpha\beta}$ and $\sum_{k\in I_\alpha} \Gru{T}{k}{l}=0$ for some $l \in I_\beta$, then $T=0$.
\end{Lem} 

\begin{proof} This follows from the equations in Lemma \ref{lemma:rep-total-morphism}
\end{proof}

\begin{Prop}\label{prop:rep-cosemisimple} Let $\mathscr{A}$ be a partial Hopf algebra with an invariant integral.   Then the components of the partial tensor category $\Corep(\mathscr{A})$ are semi-simple.
\end{Prop}
\begin{proof} 

Let $V$ be in any object of $\Corep(\mathscr{A})_{\alpha\beta}$ for $\alpha,\beta\in \mathscr{I}$.  From Lemma \ref{LemInjMor}, we see that for $T$ a morphism in $\Corep(\mathscr{A})_{\alpha\beta}$, the map $T\mapsto \sum_{k\in I_\alpha} \Gru{T}{k}{l}$ is injective for any choice of $l\in I_\beta$. It follows by column-finiteness of $V$ that the algebra of self-intertwiners of $V$ is finite-dimensional. We then immediately conclude from Corollary \ref{cor:rep-cosemisimple} that $V$ is a finite direct sum of irreducible invariant subspaces.
\end{proof} 

\subsection{Matrix coefficients of irreducible corepresentations}

Our next goal is to obtain the analogue of Schur's orthogonality
relations for matrix coefficients of corepresentations.

Given finite-dimensional vector spaces $V$ and $W$, the dual space of
$\Hom_{\C}(V,W)$ is linearly spanned by functionals of the form
\begin{align*}
  \omega_{f,v} \colon \Hom_{\C}(V,W) \to \C, \quad T \mapsto  (f|Tv),
\end{align*}
where $v\in V$, $f\in \dual{W}$, and $(-|-)$ denotes the natural
pairing of $\dual{W}$ with $W$.
\begin{Def} Let $\mathscr{A}$ be a partial bialgebra. The space of
  \emph{matrix coefficients} $\mathcal{C}(\mathscr{X})$ of a 
  corepresentation $(V,\mathscr{X})$ is the sum of the subspaces
\begin{align*}
  \Gr{\mathcal{C}(\mathscr{X})}{k}{l}{m}{n} &= \Span \left\{ (\id \otimes
    \omega_{f,v})(\Gr{X}{k}{l}{m}{n}) \mid v\in \Gru{V}{m}{n}, f \in
    \dual{(\Gru{V}{k}{l})} \right\} \subseteq \Gr{A}{k}{l}{m}{n}.
\end{align*}
\end{Def}
Let $(V,\mathscr{X})$ be  a  corepresentation of a partial bialgebra
$\mathscr{A}$.  Condition \eqref{eq:rep-comultiplication} in Definition \ref{definition:corep}
implies
\begin{align} \label{eq:rep-matrix-delta}
  \Delta_{pq}(\Gr{\mathcal{C}(\mathscr{X})}{k}{l}{m}{n}) \subseteq
  \Gr{\mathcal{C}(\mathscr{X})}{k}{l}{p}{q} \otimes
  \Gr{\mathcal{C}(\mathscr{X})}{p}{q}{m}{n}.
\end{align}
Thus, the $\Gr{\mathcal{C}(\mathscr{X})}{k}{l}{m}{n}$ form a partial
coalgebra with respect to $\Delta$ and $\epsilon$.  Moreover, for each
$k,l$, the $I^{2}$-graded vector  space
\begin{align*}
  \Grd{\mathcal{C}(\mathscr{X})}{k}{l}:=\bigoplus_{m,n }
  \Gr{\mathcal{C}(\mathscr{X})}{k}{l}{m}{n}
\end{align*}
is rcfd, and the inclusion above shows that one can
form the regular corepresentation on this space.
\begin{Lem} \label{lemma:rep-regular-embedding}
  Let $(V,\mathscr{X})$ be a  corepresentation
  of a partial bialgebra and let $f\in
  \dual{(\Gru{V}{k}{l})}$. Then the family of maps
  \begin{align*}
    \Gr{T}{}{(f)}{m}{n} \colon \Gru{V}{m}{n} \to
    \Gr{\mathcal{C}(\mathscr{X})}{k}{l}{m}{n}, \ w \mapsto (\id
    \otimes \omega_{f,w})(\Gr{X}{k}{l}{m}{n})=(\id \otimes
    f)(\Gr{X}{k}{l}{m}{n}(1 \otimes w)),
  \end{align*}
  is a morphism from $\mathscr{X}$ to the regular corepresentation on
  $\Grd{\mathcal{C}(\mathscr{X})}{k}{l}$. 
\end{Lem}
\begin{proof}
  Denote by $\mathscr{Y}$ the regular corepresentation on
  $\bigoplus_{m,n } \Gr{\mathcal{C}(\mathscr{X})}{k}{l}{m}{n}$. Then
  \begin{align*}
 \Gr{Y}{p}{q}{m}{n}    (1\otimes \Gr{T}{}{(f)}{m}{n}(v)) &= 
(\Delta^{\op}_{pq} \otimes \omega_{f,v})( \Gr{X}{k}{l}{m}{n}) 
\\ & = (\id \otimes \id \otimes
f)((\Gr{X}{k}{l}{p}{q})_{23}(\Gr{X}{p}{q}{m}{n})_{13}(1 \otimes 1
 \otimes v)) \\ &=(1 \otimes \Gr{T}{}{(f)}{p}{q})\Gr{X}{p}{q}{m}{n}(1 \otimes v)
  \end{align*}
for all $v \in \Gru{V}{m}{n}$.
\end{proof}
As before, we denote by $\dual{V}$ the dual of a vector space $V$.
\begin{Lem} \label{lemma:regular-corep} Let $\mathscr{A}$ be a partial
  Hopf algebra.
  \begin{enumerate}[label=(\arabic*)]
  \item  Let $a \in \bigoplus_{k,l} \Gr{A}{k}{l}{m}{n}$. Then the family of
  subspaces
  \begin{align} \label{eq:element-reg-corep}
    \Gru{V}{p}{q} = \{ (\id\otimes f)(\Delta_{pq}(a)) : f \in
    \dual{(\Gr{A}{p}{q}{m}{n})}\}
  \end{align}
  is rcfd and satisfies $\Delta_{rs}(\Gru{V}{p}{q}) \subseteq
  \Gru{V}{r}{s} \otimes \Gr{A}{r}{s}{p}{q}$ so that one can form the
  restriction of the regular corepresentation
  $(V,\mathscr{X})$. Moreover, $a \in \Gru{V}{m}{n}$.
\item Let $(V,\mathscr{X})$ be an irreducible restriction of the
  regular corepresentation. Then \eqref{eq:element-reg-corep} holds
  for any non-zero $a \in \Gru{V}{m}{n}$.
  \end{enumerate}
\end{Lem}
\begin{proof}
(1)  Taking $f=\epsilon$, one finds $a \in \Gru{V}{m}{n}$. Next, write
  \begin{align*}
    \Delta_{pq}(a)=\sum_{i} b_{pq}^{i} \otimes c^{i}_{pq}
  \end{align*}
  with linearly independent $(c_{pq}^{i})_{i}$. Then $ \Gru{V}{p}{q} =
  \mathrm{span}\{b_{pq}^{i} : i \}$, and  $\Delta_{rs}(\Gru{V}{p}{q}) \subseteq
  \Gru{V}{r}{s} \otimes \Gr{A}{r}{s}{p}{q}$ because
  \begin{align*}
 \sum_{i}
    \Delta_{rs}(b^{i}_{pq}) \otimes c^{i}_{pq} =
    (\Delta_{rs} \otimes \id)\Delta_{pq}(a) = (\id \otimes
    \Delta_{pq}) \Delta_{rs}(a) = \sum_{j} b^{j}_{rs} \otimes
    \Delta_{pq}(c^{j}_{rs}).
  \end{align*}
(2)  If $a\in \Gru{V}{m}{n}$ is non-zero, then the right hand
  sides of \eqref{eq:element-reg-corep} form a non-zero invariant
  family of subspaces of $\Gru{V}{p}{q}$ by (1).
\end{proof}
\begin{Prop} \label{prop:rep-weak-pw} Let $\mathscr{A}$ be a partial
  Hopf algebra with an invariant integral. Then the total algebra $A$ is the sum
  of the matrix coefficients of irreducible  corepresentations.
\end{Prop}
\begin{proof} 
  Let $a \in \Gr{A}{k}{l}{m}{n}$, define $\Gru{V}{p}{q}$ as in
  \eqref{eq:element-reg-corep} and form the restriction of the regular
  corepresentation $(V,\mathscr{X})$. Then
  \begin{align*}
    a = (\id \otimes \epsilon)(\Delta^{\op}_{kl}(a)) =
    (\id \otimes \epsilon)(\Gr{X}{k}{l}{m}{n}(1 \otimes a)) \in
    \Gr{\mathcal{C}(\mathscr{X})}{k}{l}{m}{n}.
  \end{align*}
  Decomposing $(V,\mathscr{X})$, we find that
  $a$ is contained in the sum of matrix coefficients of irreducible
  corepresentations.
\end{proof}

The first part of the orthogonality relations concerns matrix
coefficients of inequivalent irreducible corepresentations. 
\begin{Prop} \label{prop:rep-orthogonality-1} Let $\mathcal{A}$ be a
  partial Hopf algebra with an invariant integral $\phi$ and inequivalent
  irreducible corepresentations $(V,\mathscr{X})$ and
  $(W,\mathscr{Y})$.  Then  for all
  $a\in \mathcal{C}(X), b \in \mathcal{C}(Y)$,
  \[\phi(S(b)a) = \phi(bS(a))=0.\]
\end{Prop}
\begin{proof}
Since $\phi$ vanishes on $S(\Gr{A}{k}{l}{m}{n})\Gr{A}{p}{q}{r}{s}$ and
on $\Gr{A}{p}{q}{r}{s}S(\Gr{A}{k}{l}{m}{n})$ unless
$(p,q,r,s) = (m,n,k,l)$, it suffices to prove the assertion for  elements of the form
\begin{align*}
  a&=(\id \otimes \omega_{f,v})(\Gr{X}{k}{l}{m}{n})  && \text{and} &
  b&=(\id \otimes \omega_{g,w})(\Gr{Y}{m}{n}{k}{l})
\end{align*}
where $f\in \dual{(\Gru{V}{k}{l})}, v \in \Gru{V}{m}{n}$ and $g \in
\dual{(\Gru{W}{m}{n})}, w \in \Gru{W}{k}{l}$.  Lemma
\ref{lem:rep-average}, applied to the family
  \begin{align*}
    \Gru{T}{p}{q} \colon \Gru{V}{p}{q} \to \Gru{W}{p}{q}, \quad u
    \mapsto  \delta_{p,k}\delta_{q,l}  f(u)w,
  \end{align*}
  yields morphisms $\Grd{\check{T}}{k}{l},\Grd{\hat{T}}{k}{l}$ from $(V,\mathscr{X})$ to
  $(W,\mathscr{Y})$ which necessarily are $0$. Inserting the
  definition of $\Grd{\check{T}}{k}{l}$, we find
  \begin{align*}
    \phi(S(b)a) &= \phi\big((S \otimes
    \omega_{g,w})(\Gr{Y}{m}{n}{k}{l}) \cdot (\id \otimes
    \omega_{f,v})(\Gr{X}{k}{l}{m}{n})\big) \\ &= (\phi \otimes \omega_{g,v})\left(\Gr{(Y^{-1})}{l}{k}{n}{m}(1 \otimes
      \Gru{T}{k}{l} )     \Gr{X}{k}{l}{m}{n}\right) 
    = \omega_{g,v}( \Gr{\check{T}}{k}{l}{m}{n}) = 0.
  \end{align*}
  
  A similar calculation involving $\hat{T}$ shows that
  $\phi(bS(a))=0$.  
\end{proof}

\begin{Theorem} \label{thm:rep-orthogonality} Let $\mathcal{A}$ be a
 partial Hopf algebra with an invariant integral $\phi$. Let $\alpha,\beta\in \mathscr{I}$, and let $(V,\mathscr{X})$
  be an irreducible corepresentation of $\mathscr{A}$ inside $\Corep(\mathscr{A})_{\alpha\beta}$. Suppose
  $F=F_{\mathscr{X}}$ is an isomorphism from $(V,\mathscr{X})$ to
  $(V,\hat{\hat{\mathscr{X}}})$ with inverse
  $ G=F^{-1}$. Then the following hold.
  \begin{enumerate}[label=(\arabic*)]
  \item The numbers $d_G:=\sum_{k} \Tr (\Gru{G}{k}{l})$ and $d_F:=\sum_{n} \Tr (\Gru{F}{m}{n})$ are non-zero and do not depend on the choice of $l \in I_\beta$ or $m\in I_\alpha$.
    \item  For all $k,m \in I_\alpha$ and $l,n\in I_\beta$,
    \begin{align*}
      (\phi \otimes \id)(\Gr{(X^{-1})}{l}{k}{n}{m}\Gr{X}{k}{l}{m}{n})
      &=d_G^{-1}\Tr(\Gru{G}{k}{l})
      \id_{\Gru{V}{m}{n}}, \\
      (\phi \otimes \id)(\Gr{X}{k}{l}{m}{n}\Gr{(X^{-1})}{l}{k}{n}{m})
      &=d_F^{-1}\Tr(\Gru{F}{m}{n})
      \id_{\Gru{V}{k}{l}}.
    \end{align*}
  \item Denote by $\Sigma_{klmn}$ the flip map $\Gru{V}{k}{l}
    \otimes \Gru{V}{m}{n} \to \Gru{V}{m}{n}
    \otimes \Gru{V}{k}{l}$. Then
 \begin{align*}
   (\phi \otimes \id \otimes
   \id)((\Gr{(X^{-1})}{l}{k}{n}{m})_{12}(\Gr{X}{k}{l}{m}{n})_{13}) &=
   d_G^{-1}
   (\id_{\Gru{V}{m}{n}} \otimes \Gru{G}{k}{l})
   \circ \Sigma_{klmn}, \\
   (\phi \otimes \id \otimes
   \id)((\Gr{X}{k}{l}{m}{n})_{13}(\Gr{(X^{-1})}{l}{k}{n}{m})_{12}) &= d_F^{-1} (\Gru{F}{m}{n}
   \otimes \id_{\Gru{V}{k}{l}}) \circ \Sigma_{klmn}.
 \end{align*}
\end{enumerate}
  \end{Theorem}
\begin{proof}
  We prove the assertions and equations involving $d_G$ in (1), (2)
  and (3)  simultaneously; the assertions involving $d_F$  follow similarly.

  Consider
  the following endomorphism $F_{mnkl}$ of $\Gru{V}{m}{n}\otimes \Gru{V}{k}{l}$, 
  \begin{align*}
    F_{mnkl}
    &:=(\phi \otimes \id \otimes \id)\left((\Gr{(X^{-1})}{l}{k}{n}{m})_{12}(\Gr{X}{k}{l}{m}{n})_{13}\right)
    \circ \Sigma_{mnkl} \\ &= (\phi \otimes \id \otimes
    \id)\left((\Gr{(X^{-1})}{m}{n}{k}{l})_{12}
      \Sigma_{klkl,23}(\Gr{X}{k}{l}{m}{n})_{12}\right).
  \end{align*}
  By applying Lemma \ref{lem:rep-average} with respect to the flip map $\Sigma_{klkl}$, we see that the family $(F_{mnkl})_{m,n}$ is
  an endomorphism of $(V \otimes \Gru{V}{k}{l}, X\otimes \id)$ and hence
  \begin{align}
    F_{mnkl} &= \id_{\Gru{V}{m}{n}} \otimes \Gru{R}{k}{l} \label{eq:rep-orthogonal-1}
  \end{align}
  with some $\Gru{R}{k}{l} \in \Hom_{\C}(\Gru{V}{k}{l})$ not
  depending on $m,n$. 
  
  On the other hand, since $\phi = \phi S$,
  \begin{align*}
    F_{mnkl} &= (\phi \otimes \id \otimes \id)((S \otimes
    \id)(\Gr{X}{m}{n}{k}{l})_{12}(\Gr{X}{k}{l}{m}{n})_{13})
    \circ \Sigma_{mnkl} \\
    &= (\phi \otimes \id \otimes \id)\left(((S \otimes
      \id)(\Gr{X}{k}{l}{m}{n}))_{13}
      ((S^{2} \otimes \id)(\Gr{X}{m}{n}{k}{l}))_{12}\right)     \circ \Sigma_{mnkl}\\
    &= (\phi \otimes \id \otimes
    \id)\left((\Gr{(X^{-1})}{k}{l}{m}{n})_{13} (\Sigma_{mnmn})_{23}
      (\Gr{(\dual{\dual{X}{}\!})}{m}{n}{k}{l})_{13}\right).
  \end{align*}
  Hence we can again apply Lemma \ref{lem:rep-average} and
  find that the family $(F_{mnkl})_{k,l}$ is a morphism \[(F_{mnkl})_{k,l}:
  (\Gru{V}{m}{n} \otimes V, \hat{\hat{X}}_{13})\rightarrow (\Gru{V}{m}{n} \otimes V,
 X_{13}).\] Therefore,
  \begin{align}
    F_{mnkl} &= \Gru{T}{m}{n} \otimes \Gru{G}{k}{l} \label{eq:rep-orthogonal-2}
  \end{align}
  with some $\Gru{T}{m}{n} \in \mathcal{\Hom_{\C}}(\Gru{V}{m}{n})$
  not depending on $k,l$. Combining \eqref{eq:rep-orthogonal-1} and
  \eqref{eq:rep-orthogonal-2}, we conclude that, for some $\lambda\in \C$, \[F_{mnkl} = \lambda
  (\id_{\Gru{V}{m}{n}} \otimes \Gru{G}{k}{l})\]
  
  Choose dual  bases
  $(v_{i})_{i}$ for $\Gru{V}{k}{l}$ and $(f_{i})_{i}$ for  $\dual{(\Gru{V}{k}{l})}$. Then
  \begin{align*}
    \lambda   \Tr( \Gru{G}{k}{l}) \id_{\Gru{V}{m}{n}}
 &= \sum_{i} (\id \otimes
    \omega_{f_{i},v_{i}})(F_{mnkl}) = (\phi \otimes
    \id)(\Gr{(X^{-1})}{l}{k}{n}{m} \Gr{X}{k}{l}{m}{n}).
  \end{align*}
  Take now $n=l$.  By Lemma \ref{LemInjMor}, we can choose $m\in I_{\alpha}$ with $\Gru{V}{m}{n}\neq 0$.   Then summing the previous relation over $k$, the relations $\sum_{k}
  \Gr{(X^{-1})}{l}{k}{n}{m} \Gr{X}{k}{l}{m}{n} = \UnitC{l}{n}
  \otimes \id_{\Gru{V}{m}{n}}$ and
  $\phi(\UnitC{l}{l})=1$ give
\begin{align*}
\lambda \cdot  \sum_{k} \Tr(\Gru{G}{k}{l}) = 1.  
\end{align*}
Now all assertions in (1)--(3) concerning $d_G$ follow.
\end{proof}

\begin{Rem} For semi-simple tensor categories with duals, it is known
  that any object is isomorphic to its left bidual \cite[Proposition
  2.1]{ENO1}, hence there always exists an isomorphism $F_{\mathscr{X}}$ as in the previous Theorem. In fact, from the faithfulness of $\phi$ and Proposition \ref{prop:rep-orthogonality-1}, it follows that not all $F_{mnkl}$ in the previous proof are zero. Hence $G=F_{\mathscr{X}}^{-1}$ is a non-zero morphism and thus an isomorphism from the left bidual of $\mathscr{X}$ to $\mathscr{X}$.  
\end{Rem}

\begin{Cor}\label{CorOrth}
  Let $\mathscr{A}$ be a partial Hopf algebra with an invariant integral $\phi$, let
  $(V,\mathscr{X})$ be an irreducible corepresentation of
  $\mathscr{A}$, let $F=F_{\mathscr{X}}$ be an isomorphism from
  $(V,\mathscr{X})$ to $(V,\dualco{\dualco{\mathscr{X}}})$ and
  $G=F^{-1}$, and let $a=(\id \otimes
  \omega_{f,v})(\Gr{X}{k}{l}{m}{n})$ and $b=(\id \otimes
  \omega_{g,w})(\Gr{X}{m}{n}{k}{l})$, where 
  $f \in   \dual{(\Gru{V}{k}{l})}$, $v \in\Gru{V}{m}{n}$, $g \in
  \dual{(\Gru{V}{m}{n})}$, $w \in  \Gru{V}{k}{l}$.  Then
\begin{align*}
  \phi(S(b)a) &= \frac{(g|v)(f|Gw)}{\sum_{r}
    \Tr(\Gru{G}{r}{n})}, & \phi(aS(b)) = \frac{(g|Fv)(f|w)}{\sum_{s}
    \Tr(\Gru{F}{m}{s})}.
\end{align*}
\end{Cor}
\begin{proof}
Apply $\omega_{g,w} \otimes
    \omega_{f,v}$ to the formulas in  Theorem
    \ref{thm:rep-orthogonality}.(c).
\end{proof}
\begin{Cor} \label{cor:rep-pw}
  Let $\mathscr{A}$ be a partial Hopf algebra with an invariant integral and let
  $((V^{(a)},\mathscr{X}_{a}))_{a \in \mathcal{I}}$ be a maximal family of mutually non-isomorphic irreducible corepresentations of
  $\mathscr{A}$. Then the map
  \begin{align*}
    \bigoplus_{a} \bigoplus_{k,l,m,n}
    (\dual{(\Gr{V}{}{(a)}{k}{l})} \otimes
    \Gr{V}{}{(a)}{m}{n}) \to A
  \end{align*}
  that sends $f \otimes w \in
  \dual{(\Gr{V}{}{(a)}{k}{l})} \otimes
  \Gr{V}{}{(a)}{m}{n}$ to $ (\id \otimes
  \omega_{f,w})(\Gr{(X_{a})}{k}{l}{m}{n})$,
  is a linear isomorphism. 
\end{Cor}
\begin{proof} This follows from Proposition \ref{prop:rep-weak-pw}, Proposition \ref{prop:rep-orthogonality-1} and Corollary \ref{CorOrth}.
\end{proof}
\begin{Cor} \label{cor:rep-pw-morphisms}
  Let $\mathscr{A}$ be a regular partial Hopf algebra with an invariant integral, let
  $((V^{(a)},\mathscr{X}_{a}))_{a\in \mathcal{I}}$ be a maximal
  family of mutually non-isomorphic irreducible corepresentations of $\mathscr{A}$,
  fix $a \in \mathcal{I}$ and $k,l\in I$, and denote by $\Gr{\mathscr{Y}}{k}{l}{}{a}$
  the regular corepresentation on
  $\Grd{\mathcal{C}(\mathscr{X}_a)}{k}{l}$. Then there exists a
  linear isomorphism
  \begin{align*}
    \dual{( \Gr{V}{}{(a)}{k}{l})} \to
    \Mor((V^{(a)},\mathscr{X}_{a}),
    (\Grd{\mathcal{C}(\mathscr{X}_a)}{k}{l},\Gr{\mathscr{Y}}{k}{l}{}{a}))
  \end{align*}
  assigning to each $f\in     \dual{( \Gr{V}{}{(a)}{k}{l})}$ the morphism
  $T^{(f)}$ of Lemma \ref{lemma:rep-regular-embedding}.
\end{Cor}

\subsection{Unitary corepresentations of partial compact quantum groups}

Let us now enhance our partial Hopf algebras to partial compact
quantum groups. We write $B(\Hsp,\mathcal{G})$ for the linear space of
bounded morphisms between Hilbert spaces $\Hsp$ and $\mathcal{G}$. 

\begin{Def} Let $\mathscr{A}$ define a partial compact quantum group
  $\mathscr{G}$. We call a  corepresentation $\mathscr{X}$ of
  $\mathscr{A}$ on a collection of Hilbert spaces $\Gru{\Hsp}{k}{l}$
  \emph{unitary}
  if \[\Gr{(X^{-1})}{k}{l}{m}{n}=(\Gr{X}{l}{k}{n}{m})^{*}\quad
  \textrm{in }\Gr{A}{k}{l}{m}{n}\otimes
  B(\Gru{\Hsp}{l}{k},\Gru{\Hsp}{n}{m}).\] 
\end{Def}

\begin{Rem}
The total object $\Hsp$ will then only be a pre-Hilbert space, but as the local components are finite-dimensional, this will not be an issue.
\end{Rem}

\begin{Exa}\label{example:rep-trivial-unitary}
  Regard $\C^{(I)}$ as a direct sum of the trivial Hilbert spaces $\C$. Then the
  trivial corepresentation $\mathscr{U}$ on $\C^{(I)}$ is unitary.
\end{Exa}

The tensor product of corepresentations lifts to a tensor product
of unitary corepresentations as follows.  We define the tensor product
of rcfd $I^{2}$-graded Hilbert spaces similarly as for rcfd
$I^{2}$-graded vector spaces and pretend it to be strict again.  Let
$(\Hsp,\mathscr{X})$ and $(\mathcal{G},\mathscr{Y})$ be unitary rcfd
corepresentations. Then the tensor product $(\Hsp \itimes
\mathcal{G},\mathscr{X} \Circt \mathscr{Y})$ is unitary again. Indeed,
in  total form,  $(X\Circt Y)^{-1} = Y_{13}^{-1}X_{12}^{-1}
  =Y_{13}^{*}X_{12}^{*} = (X \Circt Y)^{*}$ by Remark \ref{remark:rep-tensor-multiplier}.
We hence obtain a tensor C$^*$-category $\Corep_{u,\rcf}(\mathscr{A})$ of unitary corepresentations. We denote again by $\Corep_u(\mathscr{A})$ the subcategory of all corepresentations with finite support on the hyperobject set. It is the total tensor C$^*$-category with local units of a semi-simple partial tensor C$^*$-category.

Our aim now is to show that every (irreducible) corepresentation is
equivalent to a unitary one. We show this by embedding the
corepresentation into a restriction of the regular corepresentation.
\begin{Lem} \label{lemma:rep-regular-unitary}
  Let $\mathscr{A}$ define a partial compact quantum group with
positive invariant  integral $\phi$, and let $\Gru{V}{m}{n} \subseteq
\bigoplus_{k,l} \Gr{A}{k}{l}{m}{n}$ be subspaces such that
$\Delta_{pq}(\Gru{V}{m}{n}) \subseteq \Gru{V}{p}{q} \otimes
    \Gr{A}{p}{q}{m}{n}$ and $V=\bigoplus_{k,l} \Gru{V}{k}{l}$ is rcfd. Then each $\Gru{V}{k}{l}$ is a Hilbert space with
    respect to the inner product given by $\langle
    a|b\rangle:=\phi(a^{*}b)$, and the regular corepresentation
    $\mathscr{X}$ on $V$ is unitary.
\end{Lem}
\begin{proof} 
By Lemma \ref{lemma:rep-invertible},  it suffices to show that
  \begin{equation}\label{EqUnit} \sum_{k}
    (\Gr{X}{k}{l}{m}{n'})^* \Gr{X}{k}{l}{m}{n} =
    \delta_{n,n'}\UnitC{l}{n}\otimes
    \id_{\Gru{V}{m}{n}}.
  \end{equation} 
Let  $a\in \Gru{V}{m}{n}$, $b\in \Gru{V}{m}{n'}$ and define $\omega_{b,a} \colon
\Hom_{\C}(\Gru{V}{m}{n},\Gru{V}{m}{n'}) \to \C$ by $T
\mapsto \langle b|Ta\rangle$. Then
\begin{eqnarray*}
\sum_{k }(\id \otimes \omega_{b,a})
((\Gr{X}{k}{l}{m}{n'})^* \Gr{X}{k}{l}{m}{n}))  &=& \sum_k
(\id\otimes \phi)(\Delta_{kl}^{\op}(b)^*\Delta_{kl}^{\op}(a))\\
  &=& \sum_k (\phi\otimes
  \id)(\Delta_{lk}(b^*)\Delta_{kl}(a)) \\ &=& (\phi\otimes
  \id)(\Delta_{ll}(b^*a)) \\ &=& \phi(b^*a)\UnitC{l}{n} \\&=&
  \delta_{n',n} \UnitC{l}{n} \otimes \langle b|a\rangle.
\end{eqnarray*} 
This proves \eqref{EqUnit}.
\end{proof} 

\begin{Prop} \label{prop:rep-unitarisable}   Let $\mathscr{A}$ define
  a partial compact quantum group. Then every  
  corepresentation of  $\mathscr{A}$  is
  isomorphic to a unitary one.
\end{Prop}
\begin{proof}
  By Proposition \ref{prop:rep-cosemisimple} and Corollary
  \ref{cor:rep-pw}, every corepresentation is isomorphic to a direct
  sum of irreducible regular corepresentations, which are unitary by
  Lemma \ref{lemma:rep-regular-unitary}.
\end{proof}
\begin{Cor} The partial C$^*$-tensor category $\Corep_u(\mathscr{A})$ is a partial fusion C$^{*}$-category.
\end{Cor}

\begin{Rem}
If $\mathscr{A}$ defines a partial compact quantum group $\mathscr{G}$, we will also write $\Corep_u(\mathscr{A})= \Rep_u(\mathscr{G})$, and talk of (unitary) representations of $\mathscr{G}$.
 \end{Rem}  
   
Let now $\mathscr{X}$ be a unitary corepresentation of $\mathscr{A}$. Then there exists an isomorphism from $\mathscr{X}$ to $\dualco{\dualco{\mathscr{X}}} = (S^2\otimes \id)\mathscr{X}$. The following proposition shows that it can be implemented by positive operators.

\begin{Prop} \label{prop:rep-unitary-bidual}
  Let $\mathscr{A}$ define a partial compact quantum group and let
  $(\Hsp,\mathscr{X})$ be an irreducible unitary corepresentation of
  $\mathscr{A}$.  Then there exists an isomorphism $F=F_{\mathscr{X}}$
  from $(\Hsp,\mathscr{X})$ to 
  $(\Hsp,(S^{2} \otimes \id)(\mathscr{X}))$ in $\Corep(\mathscr{A})$ such
  that each $\Gru{F}{k}{l}$ is positive.
\end{Prop}
\begin{proof}
 By Proposition \ref{prop:rep-unitarisable}, there exists an
  isomorphism $T \colon \dualco{\mathscr{X}} \to \mathscr{Y}$ for some
  unitary corepresentation $\mathscr{Y}$ on $\dual{\Hsp}$, so that in total form,
  $(1\otimes T)\dualco{X} = Y(1 \otimes T)$.
We  apply   $S \otimes -^{\tr}$ and $-^{*} \otimes -^{*\tr}$,
respectively to find 
\begin{align*}
 \dualco{\dualco{X}}(1 \otimes \dualop{T}) &= (1 \otimes
  \dualop{T})\dualco{Y}, & (1 \otimes T^{*\tr})X=\dualco{Y}(1\otimes T^{*\tr}).
\end{align*}

Combining both equations, we
find $\dualco{\dualco{X}}(1 \otimes \dualop{T}T^{*\tr})=(1 \otimes
\dualop{T}T^{*\tr})X$. Thus, we can take
$F:=\dualop{T}T^{*\tr}$.
\end{proof}

The Schur orthogonality relations in Corollary \ref{CorOrth} can be
rewritten using the involution instead of the antipode as follows.
Let $(\Hsp,\mathscr{X})$ be a unitary corepresentation of
$\mathscr{A}$. Since $(S\otimes \id)(X)=X^{-1}=X^{*}$, the space of
matrix coefficients $\mathcal{C}(\mathscr{X})$ satisfies
\begin{align} \label{eq:rep-unitary-matrix-coefficients}
  S(\Gr{\mathcal{C}(\mathscr{X})}{k}{l}{m}{n}) &=
  (\Gr{\mathcal{C}(\mathscr{X})}{m}{n}{k}{l})^{*} \subseteq \Gr{A}{n}{m}{l}{k}.
\end{align}
More precisely, let $v \in \Gru{\Hsp}{k}{l}$, $v' \in \Gru{\Hsp}{m}{n}$
and denote by $\omega_{v,v'}$ the functional
given by $T \mapsto \langle v|Tv'\rangle$. Then
\begin{align*}
  S((\id \otimes \omega_{v,v'})(\Gr{X}{k}{l}{m}{n})) &=
  (\id \otimes \omega_{v,v'}) (\Gr{(X^{-1})}{n}{m}{l}{k})) \\ & =
  (\id \otimes \omega_{v,v'})( (\Gr{X}{m}{n}{k}{l})^{*}) =
  (\id \otimes \omega_{v',v})(\Gr{X}{m}{n}{k}{l})^{*}.
\end{align*}
This equation, Proposition \ref{prop:rep-orthogonality-1},  Lemma
\ref{LemFaith} and
Corollary \ref{CorOrth} imply the following
corollaries:
\begin{Cor} \label{cor:rep-unitary-orthogonality-1}   Let $\mathscr{A}$ define a partial compact quantum group with
  positive invariant integral $\phi$ and let   $(V,\mathscr{X})$ and
  $(W,\mathscr{Y})$ be inequivalent irreducible unitary
  corepresentations of $\mathcal{A}$.  Then for all $a\in
  \mathcal{C}(X), b \in \mathcal{C}(Y)$,
  \[\phi(b^{*}a) = \phi(ba^{*})=0.\] In particular, $\mathcal{C}(X)
  \cap \mathcal{C}(Y)=0$.
\end{Cor}
\begin{Cor}\label{cor:rep-unitary-schur-orthogonality}
  Let $\mathscr{A}$ define a partial compact quantum group with
  positive invariant integral $\phi$, let $(\Hsp,\mathscr{X})$ be an irreducible
  unitary corepresentation of $\mathscr{A}$, let $F=F_{\mathscr{X}}$ be a positive
  isomorphism from $(\Hsp,\mathscr{X})$ to
  $(\Hsp,\dualco{\dualco{\mathscr{X}}})$ and
  $G=F^{-1}$, and let $a=(\id \otimes
  \omega_{v,v'})(\Gr{X}{k}{l}{m}{n})$ and $b=(\id \otimes
  \omega_{w,w'})(\Gr{X}{k}{l}{m}{n})$, where $v,w \in
  \Gru{\Hsp}{k}{l}$ and $v',w' \in \Gru{\Hsp}{m}{n}$.  Then
\begin{align*}
  \phi(b^{*}a) &= \frac{\langle w|v'\rangle\langle v|Gw'\rangle}{\sum_{r}
    \Tr(\Gru{G}{r}{n})}, & \phi(ab^{*}) = \frac{\langle
    w|Fv'\rangle \langle v|w'\rangle}{\sum_{s}
    \Tr(\Gru{F}{m}{s})}.
\end{align*}
\end{Cor}
As a consequence of Proposition \ref{prop:rep-weak-pw} and Proposition
\ref{prop:rep-unitarisable} or Lemma \ref{lemma:rep-regular-unitary},
the matrix coefficients of irreducible unitary corepresentations
span $\mathscr{A}$, and in the Corollary \ref{cor:rep-pw}, we may
assume the irreducible corepresentations
$(V^{i},\mathscr{X}_{i})$ to be unitary if $\mathscr{A}$
defines a partial compact quantum group.

\begin{Rem}\label{RemPos} In fact, Proposition \ref{prop:rep-unitary-bidual} and Corollary \ref{cor:rep-unitary-schur-orthogonality} show the following. Let $\mathscr{A}$ be a partial Hopf $^*$-algebra admitting an invariant integral $\phi$, which a priori we do not assume to be positive. Suppose however that each irreducible corepresentation of $\mathscr{A}$ is equivalent to a unitary corepresentation. Then $\phi$ is necessarily positive.
\end{Rem} 

\subsection{Analogues of Woronowicz's  characters}

Let $\mathscr{A}$ be a partial bialgebra, and $a\in \Gr{A}{k}{l}{m}{n}$. Then for $\omega \in A^*$, we can define
\begin{align*} \omega \aste{p,q} a = (\id \otimes \omega) (\Delta_{pq}(a)),&&a \aste{r,s}
\omega:=(\omega \otimes \id)(\Delta_{rs}(a)),\end{align*}and this defines a bimodule structure with respect to the natural $I\times I$-partial convolution algebra structure on $\oplus \left(\Gr{A}{k}{l}{m}{n}\right)^*$. When $\omega$ has support on $\sum_{k,l}\Gr{A}{k}{l}{k}{l}$, it is meaningful to define \begin{align*} \omega\ast a := \sum_{p,q} \omega\aste{p,q}a && a\ast \omega = \sum_{r,s} a\aste{r,s}\omega \end{align*}

We recall that an entire function $f$ has \emph{exponential growth
  on the right half-plane} if there exist $C,d>0$ such that $|f(x+iy)|\leq
C\mathrm{e}^{dx}$  for all $x,y\in \R$ with $x>0$. 

\begin{Theorem} \label{thm:rep-characters} Let $\mathscr{A}$ be a
  partial Hopf algebra with an invariant integral $\phi$.  Then there
  exists a unique family of linear functionals $f_{z} \colon A\to \C$
  such that
\begin{enumerate}[label={(\arabic*)}]
  \item $f_z$ vanishes on $A(K)$ when $K_u\neq K_d$.
  \item for each $a\in A$, the function $z\mapsto f_{z}(a)$ is entire
    and of exponential growth on the right half-plane.
  \item $f_{0} = \epsilon$ and $(f_{z} \otimes f_{z'}) \circ 
    \Delta= f_{z+z'}$ for all $z,z' \in \C$.
  \item $\phi(ab)=\phi(b(f_{1} \ast a \ast f_{1}))$ for all $a,b\in A$.
  \end{enumerate}
  This family furthermore satisfies
  \begin{enumerate}[label={(\arabic*)}]\setcounter{enumi}{4}
  \item $f_z(ab) = f_z(a)f_z(b)$ for $a\in A(K)$ and $b\in A(L)$ with $K_r = L_l$. 
  \item $S^{2}(a)=f_{-1} \ast a \ast f_{1}$ for all $a\in A$.
  \item $f_{z}(\UnitC{l}{n})=\delta_{l,n}$ and $f_{z} \circ S = f_{-z}$ for all $a\in A$.
  \item $\bar{f}_{z}=f_{-\overline{z}}$ if $\mathscr{A}$ is a partial
    Hopf $^*$-algebra and $\phi$ is positive.
\end{enumerate}
\end{Theorem}

Note that conditions (3), (4) and (6) are meaningful by condition (1).

\begin{proof}
  We first prove uniqueness.  Assume that $(f_{z})_{z}$ is a family of
  functionals satisfying (1)--(4).  Since $\phi$ is faithful, the map
  $\sigma\colon a \mapsto f_{1} \ast a \ast f_{1}$ is uniquely
  determined by $\phi$, and one easily sees that it is a homomorphism. Using
  (3), we find that $\epsilon \circ \sigma^n=f_{2n}$, which uniquely determines these functionals. Using (2) and the
  fact that every entire function of exponential growth on the right
  half-plane is uniquely determined by its values at $\N \subseteq \C$, we can conclude that the family $f_{z}$ is uniquely determined. Moreover, since the property (5) holds for $z = 2n$, we also conclude by the same argument as above that it holds for all $z\in \C$.

  Let us now prove existence.  By Theorem \ref{thm:rep-orthogonality}, Corollary \ref{cor:rep-pw} and Proposition \ref{prop:rep-unitary-bidual}, we can
  define for each $z\in \C$ a functional $f_{z} \colon A \to \C$ such
  that for every 
  irreducible corepresentation
  $(V,\mathscr{X})$ in $\Corep(\mathscr{A})$,
    \begin{align*}
      f_{z}((\id \otimes \omega_{\xi,\eta})(\Gr{X}{k}{l}{m}{n})) &=
      \delta_{k,m}\delta_{l,n}
      \omega_{\xi,\eta}((\Gru{F}{k}{l})^{z}) \quad \text{for all }
      \xi \in \Gru{V}{k}{l},\eta \in
      \Gru{V}{m}{n},
    \end{align*}
    or, equivalently,
    \begin{align*}
      (f_{z} \otimes \id)(\Gr{X}{k}{l}{m}{n}) =
      \delta_{k,m}\delta_{l,n} (\Gru{F}{k}{l})^{z},
    \end{align*}
    where $F=F_{\mathscr{X}}$ is a non-zero operator implementing a morphism from $(V,\mathscr{X})$ to
    $(V, \dualco{\dualco{\mathscr{X}}})$, scaled such that
    \begin{align*}
      d_{\mathscr{X}}:= \sum_{r} \Tr(\Gru{(F^{-1})}{r}{l}) = \sum_{s}
      \Tr(\Gru{F}{m}{s})
    \end{align*}
    for all $l$ in the right and all $m$ in the left hyperobject support of $\mathscr{X}$. By
    construction, (1) and (2) hold. We show that the $(f_{z})_{z}$ satisfy the
    assertions (3)--(7). 

    Throughout the following arguments, let $(V,\mathscr{X})$ and $F$ be as
    above.

    We first prove property (3). This follows from the relations
    \begin{align*}
      (f_{0}  \otimes \id)(\Gr{X}{k}{l}{m}{n}) &=
      \delta_{k,m}\delta_{l,n} \id_{\Gru{V}{k}{l}} =
      (\epsilon \otimes \id)(\Gr{X}{k}{l}{m}{n})
    \end{align*}
    and
    \begin{align*}
      (((f_{z}\otimes f_{z'})\circ \Delta) \otimes
      \id)(\Gr{X}{k}{l}{m}{n}) &=  \delta_{k,m}\delta_{l,n}(f_{z} \otimes f_{z'} \otimes
      \id)\big((\Gr{X}{k}{l}{k}{l})_{13}
      (\Gr{X}{k}{l}{k}{l})_{23}\big) \\
      &=  \delta_{k,m}\delta_{l,n}(\Gru{F}{k}{l})^{z}  \cdot (\Gru{F}{k}{l})^{z'} \\
      &= (f_{z+z'} \otimes \id)(\Gr{X}{k}{l}{m}{n}).
    \end{align*}
    Applying slice maps of the form $\id
    \otimes \omega_{\xi,\xi'}$ and invoking Theorem \ref{thm:rep-orthogonality}, this proves (3).

To prove (4), write $ \Delta^{(2)} = (
    \Delta \otimes \id)\circ  \Delta = (\id \otimes 
    \Delta) \circ \Delta$, and put \[\theta_{z,z'}:=(f_{z'} \otimes \id
    \otimes f_{z})\circ  \Delta^{(2)}.\] Then
    \begin{align*}
      (\theta_{z,z'} \otimes \id)(\Gr{X}{k}{l}{m}{n}) &= (f_{z'} \otimes
      \id \otimes f_{z} \otimes
      \id)((\Gr{X}{k}{l}{k}{l})_{14}(\Gr{X}{k}{l}{m}{n})_{24}(\Gr{X}{m}{n}{m}{n})_{34})
      \\
      &= (1 \otimes (\Gru{F}{k}{l})^{z'}) \Gr{X}{k}{l}{m}{n} (1
      \otimes (\Gru{F}{m}{n})^{z}).
    \end{align*}
    We take $z=z'=1$, use Theorem \ref{thm:rep-orthogonality}, where
    now $d_F= d_G=d_{\mathscr{X}}$ by our scaling of $F$, and obtain
    \begin{eqnarray*}
     && \hspace{-2cm} (\phi \otimes \id \otimes
      \id)((\Gr{(X^{-1})}{l}{k}{n}{m})_{12}((\theta_{1,1} \otimes
      \id)(\Gr{X}{k}{l}{m}{n}))_{13})\\ && =d_{\mathscr{X}}^{-1}(\id \otimes
      \Gru{F}{k}{l}) (\id \otimes \Gru{(F^{-1})}{k}{l})
      \Sigma_{k,l,m,n} (\id \otimes
      \Gru{F}{m}{n}) \\
      &&=d_{\mathscr{X}}^{-1}(\Gru{F}{m}{n} \otimes \id) \Sigma_{klmn} \\
      &&= (\phi \otimes \id \otimes
      \id)((\Gr{X}{k}{l}{m}{n})_{13}(\Gr{(X^{-1})}{l}{k}{n}{m})_{12}).
    \end{eqnarray*}
    To conclude the proof of assertion (4), apply again slice maps of the form
    $\omega_{\xi,\xi'} \otimes \omega_{\eta,\eta'}$.

We have then already argued that the property (5) automatically holds. To show the property (6), note that by Proposition \ref{prop:rep-unitary-bidual} and the calculation above,
    \begin{align*}
      (S^{2} \otimes \id)(\Gr{X}{k}{l}{m}{n}) &= (1
      \otimes\Gru{F}{k}{l})
      \Gr{X}{k}{l}{m}{n}(1 \otimes \Gru{F}{m}{n})^{-1} 
      =(\theta_{-1,1}  \otimes \id)(\Gr{X}{k}{l}{m}{n}).
    \end{align*}
     Assertion (6) follows again by applying slice maps.
    
     To check, (7), note that (1), (2) and (4) immediately imply
     $f_{z}(\UnitC{k}{m})=\delta_{k,m}$. As both $z \rightarrow
     f_{-z}$ and $z\rightarrow f_z\circ S$ satisfy the conditions
     (1)--(4) for $\mathscr{A}$ with the opposite product and
     coproduct (using the partial character property (5) and the
     invariance of $\phi$ with respect to $S$), we find $f_{-z} =
     f_{z} \circ S$.

     Finally, we assume that $\mathscr{A}$ is a partial Hopf
     $^*$-algebra with positive invariant integral $\phi$ and prove
     (8).  By Proposition \ref{prop:rep-unitary-bidual}, we can assume
     $\Gru{F}{k}{l}$ to be positive.  Write
     $\bar{f}_z(a) = \overline{f_z(a^*)}$. Using the relations $
     (\Gr{X}{k}{l}{k}{l})^{*}=(S \otimes \id)(\Gr{X}{k}{l}{k}{l})$,
     $f_{z} \circ S=f_{-z}$  and
     positivity of $\Gru{F}{k}{l}$, we conclude
     \begin{align*}
       (\bar{f}_z \otimes
       \id)(\Gr{X}{k}{l}{k}{l})
&=       \left((f_{z} \otimes
       \id)((\Gr{X}{k}{l}{k}{l})^{*})\right)^{*} \\
& = \left((f_{-z} \otimes \id)(\Gr{X}{k}{l}{k}{l})\right)^{*} 
 =
((\Gru{F}{k}{l})^{-z})^{*} 
=       (\Gru{F}{k}{l})^{-\overline{z}} = (f_{-\overline{z}}
\otimes \id)(\Gr{X}{k}{l}{k}{l}),
     \end{align*}
whence $\bar{f}_z(a) = f_{-\overline{z}}(a)$ for all $a\in
\Gr{\mathcal{C}(\mathscr{X})}{k}{l}{k}{l}$. Since $f_{z}$ and
$f_{-\overline{z}}$ vanish on $\Gr{A}{k}{l}{m}{n}$ if $(k,l)\neq
(m,n)$ and the matrix coefficients of unitary 
corepresentations span $A$, we can conclude $\bar{f}_{z}=f_{-\overline{z}}$.
\end{proof}

Note that our formula for the Woronowicz characters is slightly different from the one in \cite{Hay1}, as we are using a different normalisation of the Haar functional.


\section{Tannaka-Kre$\breve{\textrm{\i}}$n-Woronowicz duality for partial compact quantum groups}


In the previous section, we showed how any partial compact quantum group gave rise to a partial fusion C$^*$-category with a unital morphism into a partial tensor C$^*$-category of finite-dimensional Hilbert spaces. In this section we reverse this construction, and show that the two structures are in duality with each other. The proof does not differ much from the usual Tannaka-Kre$\breve{\textrm{\i}}$n reconstruction process, but one has to pay some extra care to the well-definedness of certain constructions. Implicitly, we build our reconstruction process by passing first through the construction of the discrete dual of a partial compact quantum group, which we however refrain from formally introducing.

Let us at first fix a semi-simple partial tensor category $\CatCC$
with indecomposable units over a base set $\mathscr{I}$. We will again view the tensor product of $\CatC$ as being strict, for notational convenience. 

Assume that we also have another set $I$ and a partition $I = \{I_\alpha\mid \alpha\in \mathscr{I}\}$ with associated \emph{surjective} function \[\varphi:I\rightarrow \mathscr{I}, \quad k\mapsto k'.\] Let $F: \CatCC\rightarrow \{\Vect_{\fd}\}_{I\times I}$ be a morphism based on $\varphi$, cf.~ Example \ref{ExaVectBiGr}.  We will again denote by $F_{kl}:\CatCC_{k'l'}\rightarrow \Vect_{\fd}$ the components of $F$ at index $(k,l)$, and by $\iota$ and $\mu$ resp.~ the product and unit constraints.  For $X\in \CatC_{k'\beta}$ and $Y\in \CatC_{\beta m'}$, we write the projection maps associated to the identification $F_{km}(X\otimes Y)\cong \oplus_{l\in I_\beta} \left(F_{kl}(X)\otimes F_{lm}(Y)\right)$ as \[\pi^{(klm)}_{X,Y}=(\iota^{(klm)}_{X,Y})^{*}:F_{km}(X\otimes Y) \rightarrow F_{kl}(X)\otimes F_{lm}(Y).\]

We choose a maximal family of mutually inequivalent irreducible objects $\{u_a\}_{a\in \mathcal{I}}$ in $\CatC$. We assume that the $u_a$ include the unit objects $\Unitb_{\alpha}$ for $\alpha\in \mathscr{I}$, and we may hence identify $\mathscr{I}\subseteq \mathcal{I}$. For $a\in \mathcal{I}$, we will write $u_a \in \CatC_{\lambda_a,\rho_a}$ with $\lambda_a,\rho_a\in \mathscr{I}$. For $\alpha,\beta\in \mathscr{I}$ fixed, we write $\mathcal{I}_{\alpha\beta}$ for the set of all $a\in \mathcal{I}$ with $\lambda_a=\alpha$ and $\rho_a=\beta$. When $a,b,c\in \mathcal{I}$ with $a\in \mathcal{I}_{\alpha\beta},b\in \mathcal{I}_{\beta\gamma}$ and $c\in \mathcal{I}_{\gamma\delta}$, we write $c\leq a\cdot b$ if $\Mor(u_c,u_a\otimes u_b)\neq \{0\}$. Note that with $a,b$ fixed, there is only a finite set of $c$ with $c\leq a\cdot b$. We also use this notation for multiple products.

\begin{Def} For $a\in \mathcal{I}$ and $k,l,m,n\in I$, define vector spaces \[\Gr{A}{k}{l}{m}{n}(a) =  \delta_{k',m',\lambda_a}\delta_{l',n',\rho_a} \Hom_{\C}(F_{mn}(u_a),F_{kl}(u_a))^*.\] Write \[\Gr{A}{k}{l}{m}{n} =\underset{a\in \mathcal{I}}{\bigoplus}\, \Gr{A}{k}{l}{m}{n}(a),\quad A(a) = \underset{k,l,m,n}{\bigoplus} \Gr{A}{k}{l}{m}{n}(a),\quad A = \underset{k,l,m,n}{\bigoplus} \Gr{A}{k}{l}{m}{n}.\] 
\end{Def} 

We first turn the $\Gr{A}{k}{l}{m}{n}$ into a partial coalgebra $\mathscr{A}$ over $I^2$.

\begin{Def} For $r,s\in I$, we define \[\Delta_{rs}: \Gr{A}{k}{l}{m}{n}\rightarrow \Gr{A}{k}{l}{r}{s}\otimes \Gr{A}{r}{s}{m}{n}\] as the direct sums of the duals of the composition maps \[\Hom_{\C}(F_{rs}(u_a),F_{kl}(u_a)) \otimes \Hom_{\C}(F_{mn}(u_a),F_{rs}(u_a))\rightarrow \Hom_{\C}(F_{mn}(u_a),F_{kl}(u_a)),\]\[x\otimes y \mapsto x\circ y.\]
\end{Def} 

\begin{Lem} The couple $(\mathscr{A},\Delta)$ is a partial coalgebra with counit map \[\epsilon:\Gr{A}{k}{l}{k}{l}(a)\rightarrow \C,\quad f\mapsto f(\id_{F_{kl}(u_a)}).\] Moreover, for each fixed $f\in \Gr{A}{k}{l}{m}{n}(a)$, the matrix $\left(\Delta_{rs}(f)\right)_{rs}$ is rcf.
\end{Lem} 
\begin{proof} Coassociativity and counitality are immediate by
  duality, as for each $a$ fixed the spaces $\Hom_{\C}(F_{mn}(u_a),F_{kl}(u_a))$ form a partial algebra with units $\id_{F_{kl}(u_a)}$. The rcf condition follows immediately from the rcf condition for the morphism $F$.
\end{proof}

In the next step, we define a partial algebra structure on $\mathscr{A} = \{\Gr{A}{k}{l}{m}{n}\mid k,l,m,n\}$. First note that we can identify \[\Nat(F_{mn},F_{kl}) \cong \underset{\rho_a=l'=n'}{\underset{\lambda_a=k'=m'}{\prod_a}} \Hom_{\C}(F_{mn}(u_a),F_{kl}(u_a)),\] where $\Nat(F_{mn},F_{kl})$ denotes the space of natural transformations from $F_{mn}$ to $F_{kl}$ when $k'=m'$ and $l'=n'$. Similarly, we can identify \[\Nat(F_{mn}\otimes F_{pq},F_{kl}\otimes F_{rs}) \cong  \prod_{b,c} \Hom_{\C}(F_{mn}(u_b)\otimes F_{pq}(u_c) ,F_{kl}(u_b)\otimes F_{rs}(u_c)),\] with the product over the appropriate index set and where \[F_{kl}\otimes F_{rs}:\CatC_{k'l'}\times \CatC_{r's'}\rightarrow \Vect_{\fd},\quad (X,Y) \mapsto F_{kl}(X)\otimes F_{rs}(Y).\] As such, there is a natural pairing of these spaces with resp.~ $\Gr{A}{k}{l}{m}{n}$ and $\Gr{A}{k}{l}{m}{n}\otimes \Gr{A}{r}{s}{p}{q}$. 

\begin{Def} For $k'=r', l'=s'$ and $m'=t'$, we define a product
  map \[M:\Gr{A}{k}{l}{r}{s} \otimes \Gr{A}{l}{m}{s}{t}\rightarrow
  \Gr{A}{k}{m}{r}{t},\quad f\otimes g \mapsto f\cdot g\] by the
  formula \[(f\cdot g)(x) = (f\otimes g)( \hat{\Delta}^{l}_{s}(x)),
  \qquad  x \in \Nat(F_{rt},F_{km}),\] where $\hat{\Delta}^l_s(x)$ is
  the natural transformation\[\hat{\Delta}^l_s(x):  F_{rs}\otimes
  F_{st}\rightarrow F_{kl}\otimes F_{lm},\quad
  \hat{\Delta}^l_s(x)_{X,Y} = \pi^{(klm)}_{X,Y} \circ x_{X\otimes Y}
  \circ \iota^{(rst)}_{X,Y},\quad X\in \CatC_{k'l'},Y\in \CatC_{l'm'}.\]
\end{Def}

\begin{Rem} It has to be argued that $f\cdot g$ has finite support (over $\mathcal{I})$ as a functional on $\Nat(F_{rt},F_{km})$. In fact, if $f$ is supported at $b\in \mathcal{I}_{r's'}$ and $g$ at $c\in \mathcal{I}_{s't'}$, then $f\cdot g$ has support in the finite set of $a\in \mathcal{I}_{r't'}$ with $a\leq b\cdot c$, since if $x$ is a natural transformation with support outside this set, one has $x_{u_b\otimes u_c}=0$, and hence any of the $\left(\hat{\Delta}^l_s(x)\right)_{u_b,u_c} =0$.
\end{Rem}

\begin{Lem} The above product maps turn $(\mathscr{A},M)$ into an $I^2$-partial algebra.
\end{Lem}
\begin{proof} We can extend the map $(\hat{\Delta}^l_s\otimes \id)$ on
  $\Nat(F_{rt},F_{km})\otimes \Nat(F_{tu},F_{mn})$ to a
  map \[(\hat{\Delta}^l_s\otimes \id): \Nat(F_{rt}\otimes
  F_{tu},F_{km}\otimes F_{mn}) \rightarrow  \Nat(F_{rs}\otimes
  F_{st}\otimes F_{tu},F_{kl}\otimes F_{lm}\otimes
  F_{mn}),\] \[(\hat{\Delta}^l_s\otimes \id)(x)_{X,Y,Z} =
  \left(\pi^{(klm)}_{X,Y}\otimes \id_{F_{mn}(Z)}\right) \circ
  x_{X\otimes Y, Z} \circ \left(\iota^{(rst)}_{X,Y} \otimes \id_{F_{tu}(Z)}\right).\]
By finite support, we then have that \[((f\cdot g)\cdot h)(x) = (f\otimes g\otimes h)((\hat{\Delta}^l_s\otimes \id)\hat{\Delta}^m_t(x))\] for all $f\in \Gr{A}{k}{l}{r}{s},g\in \Gr{A}{l}{m}{s}{t},h\in \Gr{A}{m}{n}{t}{u}$ and $x\in  \Nat(F_{ru},F_{kn})$. Similarly, \[((f\cdot g)\cdot h)(x) = (f\otimes g\otimes h)((\id\otimes \hat{\Delta}^m_t)\hat{\Delta}^l_s(x)).\] The associativity then follows from the 2-cocycle condition for the $\iota$- and $\pi$-maps. 

By a similar argument, one sees that the (non-zero) units are given by
$\UnitC{k}{l}\in \Gr{A}{k}{k}{l}{l}(\Unitb_{\alpha})$  (for
$\alpha=k'=l'$) corresponding to $1$ in the canonical
identifications  \[\Gr{A}{k}{k}{l}{l}(\alpha) =
\Hom_{\C}(F_{ll}(\Unitb_{\alpha}),F_{kk}(\Unitb_{\alpha}))^*\cong
\Hom_{\C}(\C,\C)^*  \cong \C^* \cong \C.\] Indeed, to prove for
example the right unit property, we use that (essentially)
$\pi_{u_a,\Unitb_{\alpha}}^{(kll)} =(\id\otimes \mu_l)$ and
$\iota_{u_a,\Unitb_{\alpha}}^{(kll)} = (\id\otimes \mu_l^{-1})$,
while \[\UnitC{k}{l}(\mu_k \circ x_{\Unitb_{\alpha}} \circ\mu_l^{-1}) = x_{\Unitb_{\alpha}} \in \C,\quad x\in \Nat(F_{ll},F_{kk}).\qedhere\] 
\end{proof} 

\begin{Prop} The partial algebra and coalgebra structures on $\mathscr{A}$ define a partial bialgebra structure on $\mathscr{A}$. 
\end{Prop}
\begin{proof} Let us check the properties in Definition \ref{DefPartBiAlg}. Properties \ref{Propa} and \ref{Propc} are left to the reader. Property \ref{Propd} was proven above. Property \ref{Propb} follows from the fact that for $k'=l'=s'=m'$, \[\hat{\Delta}^{l}_s(\id_{F_{km}}) = \delta_{ls} \id_{F_{kl}}\otimes \id_{F_{lm}}.\] 
It remains to show the multiplicativity property \ref{Prope}. This is equivalent with proving that, for each $x\in \Nat(F_{uw},F_{km})$ and $y\in \Nat(F_{rt},F_{uw})$ (with all first or second indices in the same class of $\mathscr{I}$), one has (pointwise) that (for $l'=s'$) \[ \hat{\Delta}^l_s(x\circ y) = \sum_{v,v'=l'} \hat{\Delta}^v_s(x)\circ \hat{\Delta}^l_v(y).\] This follows from the fact that $\sum_v \pi^{(uvw)}_{X,Y}\iota^{(uvw)}_{X,Y} \cong \id_{F_{uw}(X\otimes Y)}$ (where we again note that the left hand side sum is in fact finite).
\end{proof} 

Let us show now that the resulting partial bialgebra $\mathscr{A}$ has an invariant integral.

\begin{Def} Define $\phi: \Gr{A}{k}{k}{m}{m} \rightarrow \C$ as the functional which is zero on $\Gr{A}{k}{k}{m}{m}(a)$ with $a\neq \Unitb_{k'}$, and the canonical identification $\Gr{A}{k}{k}{m}{m}(k')\cong \C$ on the unit component (for $k'=m'$).
\end{Def}

\begin{Lem} The functional $\phi$ is an invariant integral.
\end{Lem}

\begin{proof} The normalisation condition $\phi(\UnitC{k}{k})=1$ is immediate by construction. Let us check left invariance, as right invariance will follow similarly.

Let $\hat{\phi}^k_l$ be the natural transformation from $F_{ll}$ to $F_{kk}$ which has support on multiples of $\Unitb_{k'}$, and with $(\hat{\phi}^k_l)_{\Unitb_{k'}} = 1$.  Then for $f\in \Gr{A}{k}{k}{l}{l}$, we have $\phi(f) = f(\hat{\phi}^k_l)$. The left invariance of $\phi$ then follows from the easy verification that for $x\in \Nat(F_{ll},F_{kn})$, \[x\circ \hat{\phi}^l_m =\delta_{k,n} \UnitC{k}{l}(x)\hat{\phi}^k_m.\qedhere\] 
\end{proof}

So far, we have constructed from $\CatCC$ and $F$ a partial bialgebra
$\mathscr{A}$ with invariant integral $\phi$. Let us further impose
for the rest of this section that $\CatCC$ admits duality.  We shall
use the following straightforward observation.
\begin{Lem}
  For all $k,l$ and $X\in \mathcal{C}_{k',l'}$,  the maps
  \begin{align*}
    \coev^{kl}_{X}  &:=  \pi^{(klk)}_{X,\hat X} \circ F_{kk}(\coev_{X})\colon \C \to F_{kl}(X)
    \otimes F_{lk}(\hat X), \\
    \ev^{kl}_{X} &:=  F_{ll}(\ev_{X}) \circ \iota^{(lkl)}_{\hat X,X} \colon
    F_{lk}(\hat X) \otimes F_{kl}(X) \to \C
  \end{align*}
  define a duality between $F_{kl}(X)$ and $F_{lk}(\hat X)$.
\end{Lem}

\begin{Prop}\label{PropAnti} The partial bialgebra $\mathscr{A}$ is a regular partial Hopf algebra.
\end{Prop} 

\begin{proof} 
  For any $x\in \Nat(F_{mn},F_{kl})$, let us define $\hat{S}(x) \in
  \Nat(F_{lk},F_{nm})$ by
\begin{align*}
  \hat{S}(x)_X &= 
(\id \otimes \ev^{lk}_{X}) \circ (\id \otimes x_{\hat X}
  \otimes \id) \circ (\coev^{nm}_{X} \otimes \id).
\end{align*}
Then the assigment $\hat{S}$ dualizes to maps $S:\Gr{A}{k}{l}{m}{n} \rightarrow \Gr{A}{n}{m}{l}{k}$ by $S(f)(x) = f(\hat{S}(x))$. We claim that $S$ is an antipode for $\mathscr{A}$. 

Let us check for example the formula \[\sum_r f_{(1){\tiny \begin{pmatrix}k&l\\n & r\end{pmatrix}}} S(f_{(2){\tiny \begin{pmatrix} n & r \\ m & l\end{pmatrix}}}) = \delta_{k,m}\epsilon(f)\UnitC{k}{n}\] for $f\in \Gr{A}{k}{l}{m}{l}$. The other antipode identity follows similarly.

By duality, this is equivalent to the pointwise identity of natural
transformations \[\sum_r\hat{M}^n_r(\id\otimes
\hat{S})\hat{\Delta}^l_r(x) = \delta_{k,m}\UnitC{k}{n}(x)
\id_{F_{kl}},\quad x\in \Nat(F_{nn},F_{km})\] where $\hat{M}^n_r$ and
$(\id\otimes \hat{S})$ are dual to $\Delta_{nr}$ and $\id\otimes S$,  respectively. 

Let us fix $X\in \mathcal{C}_{k'l'}$. Then for any $x\in
\Nat(F_{nr},F_{kl})$, $y\in \Nat(F_{rn},F_{lm})$, we have
\begin{align*}
  \left(\hat{M}^n_r(\id\otimes \hat{S})(x\otimes y)\right)_X =
\big(\id \otimes \ev_{X}^{ml}\big)  \big(x_{X} \otimes y_{\hat X} \otimes \id\big) 
  \big(\coev_{X}^{nr} \otimes \id\big).
\end{align*}
For any $x\in \Nat(F_{nn},F_{km})$, we therefore have
\begin{align*}
  \left(\hat{M}^n_r(\id\otimes \hat{S})\hat\Delta^{l}_{r}(x)\right)_X &=
\big(\id \otimes \ev^{ml}_{X}\big)  \big(\pi^{(klm)}_{X,\hat X}x_{X\otimes \hat
    X}\iota^{(nrm)}_{X,\hat X} \otimes \id\big) 
  \big(\coev^{nr}_{X} \otimes \id\big).
\end{align*}
We sum over $r$, use naturality of $x$, and obtain
\begin{align*}
\sum_{r}    \left(\hat{M}^n_r(\id\otimes
  \hat{S})\hat\Delta^{l}_{r}(x)\right)_X &=
\big(\id \otimes \ev_{X}^{ml}\big) \big(\pi^{(klm)}_{X,\hat X}x_{X\otimes \hat
    X}F_{nn}(\coev_{X}) \otimes \id\big) \\
  &=\delta_{k,m} \UnitC{k}{n}(x)
\big(\id \otimes \ev_{X}^{ml}\big) 
\big(\pi^{(mlm)}_{X,\hat X}F_{mm}(\coev_{X})
  \otimes \id\big) \\
  &=\delta_{k,m} \UnitC{k}{n}(x)
\big(\id \otimes \ev_{X}^{ml}\big) 
\big(\coev^{ml}_{X}
  \otimes \id\big) \\
  &= \delta_{k,m} \UnitC{k}{n}(x) \id.
\end{align*}
Similarly, one shows that $\mathscr{A}$ with the opposite multiplication has an antipode, using right duality. It follows that $\mathscr{A}$ is a regular partial Hopf algebra.  
\end{proof} 

Assume now that $\CatCC$ is a partial fusion C$^*$-category, and $F$ a $\phi$-morphism from $\CatCC$ to $\{\Hilb_{\fd}\}_{I\times I}$. Let us show that $\mathscr{A}$, as constructed above, becomes a partial Hopf $^*$-algebra with positive invariant integral.

\begin{Def} We define $^*: \Gr{A}{k}{l}{m}{n}\rightarrow \Gr{A}{l}{k}{n}{m}$ by the formula \[f^*(x) = \overline{f(\hat{S}(x)^*)},\qquad x\in \Nat(F_{nm},F_{lk}).\]
\end{Def}

\begin{Lem} The operation $^*$ is an anti-linear, anti-multiplicative, comultiplicative involution.
\end{Lem}

\begin{proof} Anti-linearity is clear. Comultiplicativity follows from the fact that $(xy)^* = y^*x^*$ and $\hat{S}(xy) = \hat{S}(y)\hat{S}(x)$ for natural transformations. To see anti-multiplicativity of $^*$, note first that, since $S$ is anti-multiplicative for $\mathscr{A}$, we have $\hat{S}$ anti-comultiplicative on natural transformations. Now as $(\iota_{X,Y}^{(klm)})^* = \pi_{X,Y}^{(klm)}$ by assumption, we also have $\hat{\Delta}^l_s(x)^* = \hat{\Delta}^s_l(x^*)$, which proves anti-multiplicativity of $^*$ on $\mathscr{A}$.  Finally, involutivity follows from the involutivity of $x\mapsto \hat{S}(x)^*$, which is a consequence of the fact that one can choose $\ev_{\bar{X}}^{kl} = (\coev_{X}^{lk})^*$ and $\coev_{\bar{X}}^{kl} = (\ev_X^{lk})^*$.
\end{proof}

\begin{Prop} The couple $(\mathscr{A},\Delta)$ with the above $^*$-structure defines a partial compact quantum group.
\end{Prop}
\begin{proof} The only thing which is left to prove is that our
  invariant integral $\phi$ is a positive functional. Now it is easily
  seen from the definition of $\phi$ that the $\Gr{A}{k}{l}{m}{n}(a)$
  are all mutually orthogonal. Hence it suffices to prove that the
  sesquilinear inner product \[\langle f| g\rangle = \phi(f^*g)\] on
  $\Gr{A}{k}{l}{m}{n}(a)$ is positive-definite.

  Let us write $\bar{f}(x) = \overline{f(x^*)}$. Let again
  $\hat{\phi}^k_m$ be the natural transformation from $F_{mm}$ to
  $F_{kk}$ which is the identity on $\Unitb_{k'}$ and zero on other
  irreducible objects. Then by definition, \[\phi(f^*g) =
  (\bar{f}\otimes g)((\hat{S}\otimes
  \id)\hat{\Delta}^k_m(\hat{\phi}^l_n)).\] 

Assume  that $f(x) = \langle v'| x_a v\rangle$ and
  $g(x) = \langle w' | x_aw\rangle$ for $v,w\in F_{mn}(u_a)$ and
  $v',w'\in F_{kl}(u_a)$. Then
  $\overline{f}(x) = \langle v|x_{a} v'\rangle$ and
 using the expression for $\hat{S}$ as
  in Proposition \ref{PropAnti}, we find that
  \begin{align*}
    \phi(f^*g) &= \langle v \otimes w'|
    (\ev_{a}^{kl})_{23} 
    (\hat\Delta^{k}_{m}(\hat \phi^{l}_{n})_{\bar a, a})_{24} 
    (\coev^{mn}_{a})_{12} (v'\otimes
    w)\rangle.
  \end{align*}
  However, up to a positive non-zero scalar, which we may assume to be
  1 by proper rescaling, we
  have 
  \[\hat{\Delta}^k_m(\hat{\phi}^l_n)_{\bar{a}, a} =
  (\ev^{kl}_{a})^{*}(\ev^{kl}_{a}).\] Hence
  \begin{align*}
    \phi(f^*g) &=
\langle v \otimes w'|     (\ev^{kl}_{a})_{23}  (
  (\ev^{kl}_{a})^{*}(\ev^{kl}_{a}))_{24}
 (\coev^{mn}_{a})_{12} (v'\otimes w)\rangle \\
&= \langle v \otimes w'|        (\ev^{kl}_{a})_{23} 
  (\ev^{kl}_{a})^{*}_{24}
 (w\otimes v')\rangle \\
    &= \langle v|w\rangle (\ev^{kl}_{a}|v'\rangle_{2})
    (\ev_{a}^{kl}|w'\rangle_{2})^{*},
  \end{align*}
where $\ev_{a}^{kl}|z\rangle_{2}$ denotes the map $y \mapsto
\ev_{a}^{kl}(y\otimes z)$.
If $v=w$ and $v'=w'$, the expression above clearly becomes positive.
\end{proof} 
Let us say that an $I$-partial compact quantum group with hyperobject set
$\mathscr{I}$ and corresponding partition function $\phi \colon \mathscr{I} \to
\mathscr{P}(I)$  is \emph{based over $\phi$}.

\begin{Theorem} \label{TheoTKPCQG}

The assigment $\mathscr{A}\rightarrow (\Corep_u(\mathscr{A}),F)$ is (up to isomorphism/equivalence) a one-to-one correspondence between partial compact quantum groups based over $\varphi:I\twoheadrightarrow \mathscr{I}$ and $\mathscr{I}$-partial fusion C$^*$-categories $\CatCC$ with unital morphism $F$ to $\{\Hilb_{\fd}\}_{I\times I}$ based over $\varphi$. 
\end{Theorem} 

\begin{proof} Fix first $\mathscr{A}$, and let $\mathscr{B}$ be the
  partial Hopf $^*$-algebra constructed from $\Corep_u(\mathscr{A})$
  with its natural forgetful functor. Then we have a map $\mathscr{B}
  \rightarrow \mathscr{A}$ by \[ \Gr{B}{k}{l}{m}{n}(a) =
  \Hom(\Gr{V}{}{(a)}{m}{n},\Gr{V}{}{(a)}{k}{l})^* \rightarrow
  \Gr{A}{k}{l}{m}{n}(a): f \mapsto (\id\otimes f)(X_a),\] where the
  $(V^{(a)},\mathscr{X}_a)$ run over all irreducible unitary corepresentations
  of $\mathscr{A}$. By Corollary \ref{cor:rep-pw}, this map is
  bijective. From the definition of $\mathscr{B}$, it is easy to check
  that this map is a morphism of partial Hopf $^*$-algebras.

  Conversely, let $\CatCC$ be an $\mathscr{I}$-partial fusion
  C$^*$-category with unital morphism $F$ to
  $\{\Hilb_{\fd}\}_{I\times I}$ based over $\varphi$. Let
  $\mathscr{A}$ be the associated partial Hopf $^*$-algebra. For each
  irreducible $u_a \in \CatCC$, let $V^{(a)} = F(u_a)$,
  and \[\Gr{(X_a)}{k}{l}{m}{n} = \sum_i e_i^*\otimes e_i,\] where
  $e_i$ is a basis of $\Hom_{\C}(F_{mn}(u_{a}), F_{kl}(u_{a}))$ and
  $e_i^*$ a dual basis. Then from the definition of $\mathscr{A}$ it
  easily follows that $X_a$ is a unitary corepresentation for
  $\mathscr{A}$. Clearly, $\mathscr{X}_a$ is
  irreducible. As the matrix coefficients of the $\mathscr{X}_a$ span
  $\mathscr{A}$, it follows that the $\mathscr{X}_a$ form a maximal class of
  non-isomorphic unitary corepresentations of $\mathscr{A}$. Hence we
  can make a unique equivalence \[\CatCC\rightarrow
  \Corep_u(\mathscr{A}), \quad u \mapsto (F(u),\mathscr{X}_u)\] such that
  $u_a\rightarrow \mathscr{X}_a$. From the definitions of the coproduct and
  product in $\mathscr{A}$, it is readily verified that the natural
  morphisms $\iota^{(klm)}_{u,v}:F_{kl}(u)\otimes F_{lm}(v)\rightarrow
  F_{km}(u\otimes v)$ turn it into a monoidal equivalence.
\end{proof}

\section{Examples}

\subsection{Hayashi's canonical partial compact quantum groups} \label{SubSecCan}

The following generalizes Hayashi's original construction.

\begin{Exa} 
Let $\CatCC$ be an $\mathscr{I}$-partial fusion C$^*$-category. Let
$\mathcal{I}$ label a distinguished maximal set $\{u_k\}$ of mutually
non-isomorphic irreducible objects of $\CatC$, with associated
bigrading $\Gru{\mathcal{I}}{\alpha}{\beta}$ over
$\mathscr{I}$. Define \[F_{kl}(X)  = \Hom(u_k,  X\otimes u_l),\qquad
X\in \CatC_{\alpha\beta}, k\in \Gru{\mathcal{I}}{\alpha}{\gamma},l\in \Gru{\mathcal{I}}{\beta}{\gamma}.\] Then each $F_{kl}(X)$ is a Hilbert space by the inner product $\langle f,g\rangle = f^*g$. Put $F_{kl}(X) = 0$ for $k,l$ outside their proper domains. Then clearly the application $(k,l)\mapsto F_{kl}(X)$ is rcf. Moreover, we have isometric compatibility morphisms \[F_{kl}(X)\otimes F_{lm}(Y)\rightarrow F_{km}(X\otimes Y),\quad f\otimes g \mapsto (\id\otimes g)f,\] while $F_{kl}(\Unitb_{\alpha}) \cong \delta_{kl} \C$ for $k,l\in \Gru{\mathcal{I}}{\alpha}{\alpha}$. 

It is readily verified that $F$ defines a unital morphism from $\CatCC$ to $\{\Hilb_{\fd}\}_{\mathcal{I}\times \mathcal{I}}$ based over the partition \[\mathcal{I}_{\alpha} = \bigcup_{\beta} \Gru{\mathcal{I}}{\alpha}{\beta},\quad\alpha\in \mathscr{I}.\] From the Tannaka-Kre$\breve{\textrm{\i}}$n-Woronowicz reconstruction result, we obtain a partial compact quantum group $\mathscr{A}_{\CatCC}$ with object set $\mathcal{I}$, which we call the \emph{canonical partial compact quantum group} associated with $\CatCC$. 
\end{Exa} 

\begin{Exa} More generally, let $\CatCC$ be an $\mathscr{I}$-partial fusion C$^*$-category, and let $\CatDD$ be a \emph{semi-simple partial $\CatCC$-module C$^*$-category} based over a set $\mathscr{J}$ and function $\phi:\mathscr{J}\rightarrow \mathscr{I},k\mapsto k'$. That is, $\CatDD$ consists of a collection of semi-simple C$^*$-categories $\CatD_{k}$ with $k\in \mathscr{J}$, together with tensor products $\otimes: \CatC_{k'l'}\times \CatD_{l}\rightarrow \CatD_{k}$ satisfying the appropriate associativity and unit constraints. Then if $\mathcal{I}$ labels a distinguished maximal set $\{u_a\}$ of mutually non-isomorphic irreducible objects of $\CatD$, with associated grading $\mathcal{I}_{k}$ over $\mathscr{J}$, we can again define \[F_{ab}(X)  = \Hom(u_a,  X\otimes u_b),\qquad X\in \CatC_{k'l'}, a\in \mathcal{I}_{k},b\in\mathcal{I}_{l},\] and we obtain a unital morphism from $\CatCC$ to $\{\Hilb_{\fd}\}_{\mathcal{I}\times \mathcal{I}}$. The associated partial compact quantum group $\mathscr{A}_{\CatCC}$ will be called the \emph{canonical partial compact quantum group} associated with $(\CatCC,\CatDD)$. The previous construction coincides with the special case $\CatCC= \CatDD$ with $\mathscr{J} = \mathscr{I}\times \mathscr{I}$ and $\phi$ projection to the first factor.
\end{Exa}

\begin{Exa}\label{ExaErgo} As a particular instance, let $\G$ be a compact quantum group, and consider an ergodic action of $\G$ on a unital C$^*$-algebra $C(\mathbb{X})$. Then the collection of finitely generated $\G$-equivariant $C(\mathbb{X})$-Hilbert modules forms a module C$^*$-category over $\Rep_u(\G)$, cf.~ \cite{DCY1}. 
\end{Exa}

\subsection{Morita equivalence}


\begin{Def} Two partial compact quantum groups $\mathscr{G}$ and $\mathscr{H}$ are said to be \emph{Morita equivalent} if there exists an equivalence $\Rep_u(\mathscr{G}) \rightarrow \Rep_u(\mathscr{H})$ of partial fusion C$^*$-categories. 
\end{Def} 

In particular, if $\mathscr{G}$ and $\mathscr{H}$ are Morita equivalent they have the same hyperobject set, but they need not share the same object set.

Our goal is to give a concrete implementation of Morita equivalence, as has been done for compact quantum groups \cite{BDV1}. Note that we slightly changed their terminology of monoidal equivalence into Morita equivalence, as we feel the monoidality is intrinsic to the context. We introduce the following definition, in which indices are considered modulo 2. 

\begin{Def} A \emph{linking partial compact quantum group} consists of a partial compact quantum group $\mathscr{G}$ defined by a partial Hopf $^*$-algebra $\mathscr{A}$ over a set $I$ with a distinguished partition $I = I_1\sqcup I_2$ such that the units $\UnitC{i}{j} = \sum_{k\in I_i,l\in I_j} \UnitC{k}{l} \in M(A)$ are central, and such that for each $r\in I_i$, there exists $s\in I_{i+1}$ such that $\UnitC{r}{s}\neq 0$.
\end{Def}

If $\mathscr{A}$ defines a linking partial compact quantum group, we can split $A$ into four components $A^i_j = A\UnitC{i}{j}$. It is readily verified that the $A^i_i$ together with all $\Delta_{rs}$ with $r,s \in I_i$ define themselves partial compact quantum groups, which we call the \emph{corner} partial compact quantum groups of $\mathscr{A}$. 

\begin{Prop} Two partial compact quantum groups are Morita equivalent iff they arise as the corners of a linking partial compact quantum group.
\end{Prop}

\begin{proof} Suppose first that $\mathscr{G}_1$ and $\mathscr{G}_2$ are Morita equivalent partial compact quantum groups with associated partial Hopf $^*$-algebras $\mathscr{A}_1$ and $\mathscr{A}_2$ over respective sets $I_1$ and $I_2$. Then we may identify their corepresentation categories with the same abstract partial tensor C$^*$-category $\CatCC$ based over their common hyperobject set $\mathscr{I}$. Then $\CatCC$ comes endowed with two forgetful functors $F^{(i)}$ to $\{\Hilb_{\fd}\}_{I_i\times I_i}$ corresponding to the respective $\mathscr{A}_i$.

With $I = I_1\sqcup I_2$, we may then as well combine the $F^{(i)}$ into a global unital morphism $F:\CatCC \rightarrow \{\Hilb_{\fd}\}_{I\times I}$, with $F_{kl}(X)=F_{kl}^{(i)}(X)$ if $k,l\in I_i$ and $F_{kl}(X)=0$ otherwise. Let $\mathscr{A}$ be the associated partial Hopf $^*$-algebra constructed from the Tannaka-Kre$\breve{\textrm{\i}}$n-Woronowicz reconstruction procedure. 

From the precise form of this reconstruction, it follows immediately that $\Gr{A}{k}{l}{m}{n} =0$ if either $k,l$ or $m,n$ do not lie in the same $I_i$. Hence the $\UnitC{i}{j} = \sum_{k\in I_i,l\in I_j} \UnitC{k}{l}$ are central. 

Moreover, fix $k\in I_i$ and any $l\in I_{i+1}$ with $k'=l'$. Then $\Nat(F_{ll},F_{kk})\neq \{0\}$. It follows that $\UnitC{k}{l}\neq 0$. Hence $\mathscr{A}$ is a linking compact quantum group. It is clear that $\mathscr{A}_1$ and $\mathscr{A}_2$ are the corners of $\mathscr{A}$. 

Conversely, suppose that $\mathscr{A}_1$ and $\mathscr{A}_2$ arise from the corners of a linking partial compact quantum group defined by $\mathscr{A}$ with invariant integral $\phi$. We will show that the associated partial compact quantum groups $\mathscr{G}$ and $\mathscr{G}_1$ are Morita equivalent. Then by symmetry $\mathscr{G}$ and $\mathscr{G}_2$ are Morita equivalent, and hence also $\mathscr{G}_1$ and $\mathscr{G}_2$.

For $(V,\mathscr{X}) \in \Corep_u(\mathscr{A})$, let $F(V,\mathscr{X})
= (W,\mathscr{Y})$ be the pair obtained from $(V,\mathscr{X})$ by
restricting all indices to those contained in $I_1$. It is immediate that $(W,\mathscr{Y})$ is a unitary corepresentation of $\mathscr{A}_1$, and that the functor $F$  becomes a unital morphism in a trivial way. What remains to show is that $F$ is an equivalence of categories, i.e.~ that $F$ is faithful and essentially surjective. 

Let us first show that $F$ is faithful.  Lemma
\ref{lemma:rep-invertible} implies that for every $(V,\mathscr{X}) \in
\Corep_u(\mathscr{A})$, we have $\Gru{V}{k}{l}=0$ whenever $k\in
I_{i}$ and $l\in I_{i+1}$.  If $T$ is a morphism in
$\Corep_u(\mathscr{A})_{\alpha\beta}$ and $\Gru{T}{k}{l}=0$ for all
$k,l \in I_{1}$, we therefore get $\Gru{T}{k}{l}=0$ for all $k\in I$
and $l\in I_{1}$. Since $I_{\beta}\cap I_{1}$ is non-empty by
assumption, we can apply Lemma \ref{LemInjMor} and conclude that
$T=0$.

To complete the proof, we only need to show that $F$ induces a
bijection between isomorphism classes of irreducible unitary 
corepresentations of $\mathcal{A}$ and of $\mathcal{A}_{1}$. Note that
by Proposition \ref{prop:rep-cosemisimple} and Lemma
\ref{lemma:rep-regular-embedding}, each such class can be represented
by a restriction of the regular corepresentation of $\mathcal{A}$ or
$\mathcal{A}_{1}$, respectively.

So, let $(W,\mathscr{Y})$ be an irreducible restriction of the regular
corepresentation of $\mathcal{A}_{1}$. Pick a non-zero $a \in
\Gru{W}{m}{n}$, define $\Gru{V}{p}{q} \subseteq \bigoplus_{k,l}
\Gr{A}{k}{l}{p}{q}$ as in \eqref{eq:element-reg-corep} and form the
regular corepresentation $(V,\mathscr{X})$ of $\mathscr{A}$. Then
$\Gru{V}{p}{q} = \Gru{W}{p}{q}$ for all $p,q\in I_{1}$ by Lemma
\ref{lemma:regular-corep} (2) and hence $F(V,\mathscr{X}) =
(W,\mathscr{Y})$.  Since $F$ is faithful, $(V,\mathscr{X})$ must be
irreducible.

Conversely, let $(V,\mathscr{X})$ be an irreducible restriction of the
regular corepresentation of $\mathcal{A}$. Since $F$ is faithful,
there exist $k,l\in I_{1}$ such that $\Gru{V}{k}{l}\neq 0$. Applying
Corollary \ref{cor:rep-pw-morphisms}, we may assume that
$\Gru{V}{p}{q} \subseteq \Gr{A}{k}{l}{p}{q}$ for some $k,l\in I_{1}$
and all $p,q\in I$. But then $F(V,\mathscr{X})$ is a restriction of
the regular corepresentation of $\mathcal{A}_{1}$.  If
$F(V,\mathscr{X})$ would decompose into a direct sum of several
irreducible corepresentations, then the same would be true for
$(V,\mathscr{X})$ by the argument above. Thus, $F(V,\mathscr{X})$ is irreducible.

Finally, assume that
$(V,\mathscr{X})$ and $(W,\mathscr{Y})$ are 
inequivalent irreducible unitary corepresentations of $\mathcal{A}$. Then $\mathcal{C}(V,\mathscr{X}) \cap
\mathcal{C}(W,\mathscr{Y})=0$ by Corollary \ref{cor:rep-unitary-orthogonality-1} and hence $\mathcal{C}(F(V,\mathscr{X}))
\cap\mathcal{C}(F(W,\mathscr{Y})) =0$, whence $F(V,\mathscr{X})$ and
$F(W,\mathscr{Y})$ are inequivalent.  
\end{proof}

\begin{Exa} If $\mathscr{G}_1$ and $\mathscr{G}_2$ are Morita equivalent compact quantum groups, the total partial compact quantum group is the co-groupoid constructed in \cite{Bic1}. 
\end{Exa}

\begin{Exa}  Let $\G$ be a compact quantum group with ergodic action on a unital C$^*$-algebra $C(\mathbb{X})$. Consider the module C$^*$-category $\CatD$ of finitely generated $\G$-equivariant Hilbert $C(\mathbb{X})$-modules as in Example \ref{ExaErgo}. Then $\G$ is Morita equivalent with the canonical partial compact quantum group constructed from $(\Rep_u(\G),\CatD)$. The off-diagonal part of the associated linking partial compact quantum group was studied in \cite{DCY1}. We will make a detailed study of the case $\G = SU_q(2)$ in \cite{DCT2}, in particular for $\X$ a Podle\'{s} sphere. This will lead us to partial compact quantum group versions of the dynamical quantum $SU(2)$-group.
\end{Exa}

\subsection{Weak Morita equivalence}

\begin{Def} A \emph{linking} partial fusion C$^*$-category consists of a partial fusion C$^*$-category with a distinguished partition $\mathscr{I} =\mathscr{I}_1 \cup \mathscr{I}_2$ such that for each $\alpha\in \mathscr{I}_1$, there exists $\beta \in \mathscr{I}_{2}$ with $\CatC_{\alpha\beta}\neq \{0\}$.

The \emph{corners} of $\CatCC$ are the restrictions of $\CatCC$ to $\mathscr{I}_1$ and $\mathscr{I}_2$.
\end{Def}

The following notion is essentially the same as the one by M. M\"{u}ger \cite{Mug1}. 

\begin{Def} Two partial semi-simple tensor C$^*$-categories $\CatCC_1$ and $\CatCC_2$ with duality over respective sets $\mathscr{I}_1$ and $\mathscr{I}_2$ are called \emph{Morita equivalent} if there exists a linking partial fusion C$^*$-category $\CatCC$ over the set $\mathscr{I}=\mathscr{I}_1\sqcup \mathscr{I}_2$ whose corners are isomorphic to $\CatCC_1$ and $\CatCC_2$.

We say two partial compact quantum groups $\mathscr{G}_1$ and $\mathscr{G}_2$ are \emph{weakly Morita equivalent} if their representation categories $\Rep_u(\mathscr{G}_i)$ are Morita equivalent. 
\end{Def} 

One can prove that this is indeed an equivalence relation. 

\begin{Def}\label{DefCoLink} A \emph{co-linking partial compact quantum group} consists of a partial compact quantum group $\mathscr{G}$ defined by a Hopf $^*$-algebra $\mathscr{A}$ over an index set $I$, together with a distinguished partition $I = I_1\cup I_2$ such that  $\UnitC{k}{l}=0$ whenever $k\in I_i$ and $l\in I_{i+1}$, and such that for each $k\in I_i$, there exists $l\in I_{i+1}$ with $\Gr{A}{k}{l}{k}{l}\neq 0$.  
\end{Def} 

It is again easy to see that if we restrict all indices of a co-linking partial compact quantum group to one of the distinguished sets, we obtain a partial compact quantum group which we will call a corner. In fact, write $e_i = \sum_{k,l\in I_i} \UnitC{k}{l}$. Then we can decompose the total algebra $A$ into components $A_{ij} = e_{i}Ae_{j}$, and correspondingly write $A$ in matrix notation \[ A = \begin{pmatrix} A_{11} & A_{12}  \\ A_{21} & A_{22}\end{pmatrix},\] where multiplication is matrixwise and where comultiplication is entrywise. Note that we have $A_{12}A_{21} = A_{11}$, and similarly $A_{21}A_{12} = A_{22}$. Indeed, take $k\in I_1$, and pick $l\in I_2$ with $\Gr{A}{k}{l}{k}{l}\neq \{0\}$. Then in particular, we can find an $a\in \Gr{A}{k}{l}{k}{l}$ with $\epsilon(a)\neq 0$. Hence for any $m\in I_1$, we have $\UnitC{k}{m} = \UnitC{k}{m} a_{(1)}S(a_{(2)}) \in A_{12}A_{21}$. As this latter space contains all local units of $A_{11}$ and is a right $A_{11}$-module, it follows that it is in fact equal to $A_{11}$. We hence deduce that in fact $A_{11}$ and $A_{22}$ are Morita equivalent algebras, with the Morita equivalence implemented by $A$. 

\begin{Rem} For finite partial compact quantum groups, one can then
  easily show that the notion of a co-linking partial compact quantum
  group is dual to the notion of a linking partial compact quantum group.\end{Rem}

\begin{Def} We call two partial compact quantum groups \emph{co-Morita equivalent} if there exists a \emph{co-linking partial compact quantum group} having these partial compact quantum groups as its corners.
\end{Def}

\begin{Lem} Co-Morita equivalence is an equivalence relation. 
\end{Lem} 

\begin{proof} Symmetry is clear. Co-Morita equivalence of
  $\mathscr{A}$ with itself follows by considering as co-linking
  quantum groupoid the product of $\mathscr{A}$ with the partial
  compact quantum group $M_2(\C)$, where $\Delta(e_{ij}) =
  e_{ij}\otimes e_{ij}$, arising from a groupoid as in  Example \ref{ExaGrpd}.

Let us show the main elements to prove transitivity. Let us assume $\mathscr{G}_1$ and $\mathscr{G}_2$ as well as $\mathscr{G}_2$ and $\mathscr{G}_3$ are co-Morita equivalent. Let us write the global $^*$-algebras of the associated co-linking quantum groupoids as \[A_{\{1,2\}} = \begin{pmatrix} A_{11} & A_{12} \\ A_{21} & A_{22} \end{pmatrix}, \quad A_{\{2,3\}} = \begin{pmatrix} A_{22} & A_{23} \\ A_{32} & A_{33}\end{pmatrix}.\] Then we can define a new $^*$-algebra $A_{\{1,2,3\}}$ as \[ A_{\{1,2,3\}} = \begin{pmatrix} A_{11} & A_{12} &   A_{13} \\ A_{21} & A_{22} & A_{23} \\ A_{31} & A_{32} & A_{33} \end{pmatrix},\] where $A_{13} = A_{12}\underset{A_{22}}{\otimes } A_{23}$ and $A_{31} = A_{32}\underset{A_{22}}{\otimes} A_{21}$, and with multiplication and $^*$-structure defined in the obvous way. 

It is straightforward to verify that there exists a unique $^*$-homomorphism $\Delta: A_{\{1,2,3\}} \rightarrow M(A_{\{1,2,3\}}\otimes A_{\{1,2,3\}})$ whose restrictions to the $A_{ij}$ with $|i-j|\leq 1$ coincide with the already defined coproducts. We leave it to the reader to verify that $(A,\Delta)$ defines a regular weak multiplier Hopf $^*$-algebra satisfying the conditions of Proposition \ref{PropCharPBA}, and hence arises from a regular partial weak Hopf $^*$-algebra. 

Let now $\phi$ be the functional which is zero on the off-diagonal entries $A_{ij}$ and coincides with the invariant positive integrals on the $A_{ii}$. Then it is also easily checked that $\phi$ is invariant. To show that $\phi$ is positive, we invoke Remark \ref{RemPos}. Indeed, any irreducible corepresentation of $A_{\{1,2,3\}}$ has coefficients in a single $A_{ij}$. For those $i,j$ with $|i-j|\leq 1$, we know that the corepresentation is unitarizable by restricting to a corner $2\times 2$-block. If however the corepresentation $\mathscr{X}$ has coefficients living in (say) $A_{13}$, it follows from the identity $A_{12}A_{23}=A_{13}$ that the corepresentation is a direct summand of a product $\mathscr{Y}\Circt \mathscr{Z}$ of corepresentations with coefficients in respectively $A_{12}$ and $A_{23}$. This proves unitarizability of $\mathscr{X}$. It follows from Remark \ref{RemPos} that $\phi$ is positive, and hence $\mathscr{A}_{\{1,2,3\}}$ defines a partial compact quantum group.  

We claim that the subspace $\mathscr{A}_{\{1,3\}}$ (in the obvious notation) defines a co-linking compact quantum group between $\mathscr{G}_1$ and $\mathscr{G}_3$. In fact, it is clear that the $\mathscr{A}_{11}$ and $\mathscr{A}_{33}$ are corners of $\mathscr{A}_{\{1,3\}}$, and that $\UnitC{k}{l}=0$ for $k,l$ not both in $I_1$ and $I_{3}$. To finish the proof, it is sufficient to show now that for each $k\in I_1$, there exists $l\in I_{3}$ with $\Gr{A}{k}{l}{k}{l}\neq 0$, as the other case follows by symmetry using the antipode. But there exists $m\in I_2$ with $\Gr{A}{k}{m}{k}{m} \neq \{0\}$, and $l\in I_3$ with $\Gr{A}{m}{l}{m}{l}\neq\{0\}$. As in the discussion following Definition \ref{DefCoLink}, this implies that there exists $a\in \Gr{A}{k}{m}{k}{m}$ and $b\in \Gr{A}{m}{l}{m}{l}$ with $\epsilon(a)=\epsilon(b)=1$. Hence $\epsilon(ab)=1$, showing $\Gr{A}{k}{l}{k}{l}\neq \{0\}$.
\end{proof} 

\begin{Prop}\label{PropCoWeak} Assume that two partial compact quantum groups $\mathscr{G}_1$ and $\mathscr{G}_2$ are co-Morita equivalent. Then they are weakly Morita equivalent.
\end{Prop} 
\begin{proof} 
Consider the corepresentation category $\CatCC$ of a co-linking partial compact quantum group $\mathscr{A}$ over $I = I_1\cup I_2$. Let $\varphi:I\rightarrow \mathscr{I}$ define the corresponding partition along the hyperobject set. Then by the defining property of a co-linking partial compact quantum group, also $\mathscr{I} = \mathscr{I}_1\cup \mathscr{I}_2$ with $\mathscr{I}_i=\varphi(I_i)$ is a partition. Hence $\CatCC$ decomposes into parts $\CatCC_{ij}$ with $i,j\in \{1,2\}$ and $\CatC_{ii}\cong \Rep_u(\mathscr{G}_i)$. 

To show that $\mathscr{G}_1$ and $\mathscr{G}_2$ are weakly Morita equivalent, it thus suffices to show that $\{\CatCC_{ij}\}$ forms a linking partial fusion C$^*$-category. But fix $\alpha\in I_1$ and $k\in I_{\alpha}$. Then as $\mathscr{A}$ is co-linking, there exists $l \in I_2$ with $\Gr{A}{k}{l}{k}{l}\neq \{0\}$. Hence there exists a non-zero regular unitary corepresentation inside $\oplus_{m,n}\Gr{A}{k}{l}{m}{n}$. If then $l\in I_{\beta}$ with $\beta\in \mathscr{I}_2$, it follows that $\CatC_{\alpha\beta}\neq 0$. By symmetry, we also have that for each $\alpha \in \mathscr{I}_2$ there exists $\beta \in \mathscr{I}_1$ with $\CatC_{\alpha\beta}\neq \{0\}$. This proves that the $\{\CatCC_{ij}\}$ forms a linking partial fusion C$^*$-category.
\end{proof}

\begin{Prop}\label{PropCoLink} Let $\CatCC$ be a linking $\mathscr{I}$-partial fusion C$^*$-category. Then the associated canonical partial compact quantum group is a co-linking partial compact quantum group. 
\end{Prop} 

\begin{proof} Let $\mathscr{I}= \mathscr{I}_1\cup \mathscr{I}_2$ be the associated partition of $\mathscr{I}$. Let $\mathscr{A} = \mathscr{A}_{\CatCC}$ define the canonical partial compact quantum group with object set $I$ and hyperobject partition $\varphi:I\rightarrow \mathscr{I}$. Let $I=I_1\cup I_2$ with $I_i = \varphi^{-1}(\mathscr{I}_i)$ be the corresponding decomposition of $I$. By construction, $\UnitC{k}{l}=0$ if $k$ and $l$ are not both in $I_1$ or $I_2$. 

Fix now $k\in I_{\alpha}$ for some $\alpha \in I_i$. Pick $\beta\in I_{i+1}$ with $\CatC_{\alpha\beta}\neq\{0\}$, and let $(V,\mathscr{X})$ be a non-zero irreducible corepresentation inside $\CatC_{\alpha\beta}$. Then by irreducibility, we know that $\oplus_l \Gru{V}{k}{l} \neq \{0\}$, hence there exists $l\in I_{\beta}$ with $\Gru{V}{k}{l}\neq \{0\}$. As $(\epsilon\otimes \id)\Gr{X}{k}{l}{k}{l} = \id_{\Gru{V}{k}{l}}$, it follows that $\Gr{A}{k}{l}{k}{l} \neq 0$. This proves that $\mathscr{A}$ defines a co-linking partial compact quantum group.
\end{proof} 

\begin{Rem} Note however that the corners of the canonical partial compact quantum group associated to linking $\mathscr{I}$-partial fusion C$^*$-category \emph{are not} the canonical partial compact quantum groups associated to the corners of the linking $\mathscr{I}$-partial fusion C$^*$-category. Rather, they are Morita equivalent copies of these.
\end{Rem} 

\begin{Theorem} Two partial compact quantum groups $\mathscr{G}_1$ and $\mathscr{G}_2$ are weakly Morita equivalent if and only if they are connected by a string of Morita and co-Morita equivalences. 
\end{Theorem}

\begin{proof} Clearly if two partial compact quantum groups are Morita
  equivalent, they are weakly Morita equivalent. By Proposition
  \ref{PropCoWeak}, the same is true for co-Morita equivalence. This proves one direction of the theorem. 

Conversely, assume $\mathscr{G}_1$ and $\mathscr{G}_2$ are weakly Morita equivalent. Let $\CatCC$ be a linking fusion C$^*$-category between $\Rep_u(\mathscr{G}_1)$ and $\Rep_u(\mathscr{G}_2)$. Then $\mathscr{G}_i$ are Morita equivalent with the corners of the canonical partial compact quantum group associated to $\CatCC$. But Proposition \ref{PropCoLink} shows that these corners are co-Morita equivalent. 
\end{proof} 

\begin{Rem} 
\begin{enumerate}
\item Note that it is essential that we allow the string of equivalences to pass through partial compact quantum groups, even if we start out with (genuine) compact quantum groups.
\item One can show that if $\mathscr{G}$ is a finite partial compact quantum group, then $\mathscr{G}$ is weakly Morita equivalent with its dual $\widehat{\mathscr{G}}$ (defined by the dual weak Hopf $^*$-algebra). In fact, if $\mathscr{G}$ is the canonical partial compact quantum group associated to a finite partial fusion C$^*$-category, then $\mathscr{G}$ is isomorphic to the co-opposite of its dual, e.g.~ the case of dynamical quantum $SU(2)$ at roots of unity. In any case, it follows that two finite quantum groups $H$ and $G$ are weakly Morita equivalent if and only if they can be connected by a string of 2-cocycle-elements and 2-cocycle functionals. 
\end{enumerate}
\end{Rem}

\section{Partial compact quantum groups from reciprocal random walks}

In this section, we study in more detail the construction from Section \ref{SubSecCan} in case the category $\CatC$ is the Temperley-Lieb C$^*$-category.

\subsection{Reciprocal random walks}

We recall some notions introduced in \cite{DCY1}. We slightly change the terminology for the sake of convenience.

\begin{Def} Let $t\in \R_0$. A \emph{$t$-reciprocal random walk} consists of a quadruple $(\Gamma,w,\sgn,i)$ with \begin{itemize}
\item[$\bullet$] $\Gamma=(\Gamma^{(0)},\Gamma^{(1)},s,t)$ a graph with \emph{source} and \emph{target} maps \[s,t:\Gamma^{(1)}\rightarrow \Gamma^{(0)},\]
\item[$\bullet$] $w$ a function (the \emph{weight} function) $w:\Gamma^{(1)}\rightarrow \R_0^+$,
\item[$\bullet$] $\sgn$ a function (the \emph{sign} function) $\sgn:\Gamma^{(1)}\rightarrow \{\pm 1\}$,
\item[$\bullet$] $i$ an involution \[i:\Gamma^{(1)} \rightarrow \Gamma^{(1)},\quad e\mapsto \overline{e}\] with $s(\bar{e}) = t(e)$ for all edges $e$,
\end{itemize}
such that the following conditions are satisfied:
\begin{enumerate}[label=(\arabic*)]
\item (weight reciprocality) $w(e)w(\bar{e}) = 1$ for all edges $e$,
\item (sign reciprocality) $\sgn(e)\sgn(\bar{e}) = \sgn(t)$ for all edges $e$,
\item (random walk property) $p(e) = \frac{1}{|t|}w(e)$ satisfies $\sum_{s(e)=v} p(e) = 1$ for all $v\in \Gamma^{(0)}$.
\end{enumerate}
\end{Def}


Note that, by \cite[Proposition 3.1]{DCY1}, there is a uniform bound on the number of edges leaving from any given vertex $v$, i.e.~ $\Gamma$ has a finite degree.

For examples of $t$-reciprocal random walks, we refer to \cite{DCY1}. One particular example (which will be needed for our construction of dynamical quantum $SU(2)$) is the following.

\begin{Exa}\label{ExaGraphPod} Take $0<|q|<1$ and $x\in \R$. Write $2_q = q+q^{-1}$. Then we have the reciprocal $-2_q$-random walk \[\Gamma_x =(\Gamma_x,w,\sgn,i)\] with \[ \Gamma^{(0)} = \Z,\quad \Gamma^{(1)} = \{(k,l)\mid |k-l|= 1\}\subseteq \Z\times \Z\] with projection on the first (resp. second) leg as source (resp. target) map, with weight function \[w(k,k\pm 1) = \frac{|q|^{x+k\pm 1}+|q|^{-(x+k\pm 1)}}{|q|^{x+k}+|q|^{-(x+k)}},\] sign function \[\sgn(k,k+1) = 1,\quad \sgn(k,k-1) = -\sgn(q),\] and involution $\overline{(k,k+1)} = (k+1,k)$. 

By translation we can shift the value of $x$ by an integer. By a point reflection and changing the direction of the arrows, we can change $x$ into $-x$. It follows that by some (unoriented) graph isomorphism, we can always arrange to have $x\in \lbrack 0,\frac{1}{2}\rbrack$.
\end{Exa} 

\subsection{Temperley-Lieb categories}

Let now $0<|q|\leq 1$, and let $SU_q(2)$ be Woronowicz's twisted $SU(2)$ group \cite{Wor1}. Then $SU_q(2)$ is a compact quantum group whose category of finite-dimensional unitary representations $\Rep(SU_q(2))$ is generated by the spin $1/2$-representation $\pi_{1/2}$ on $\C^2$. It has the same fusion rules as $SU(2)$, and conversely any compact quantum group with the fusion rules of $SU(2)$ has its representation category equivalent to $\Rep(SU_q(2))$ as a tensor C$^*$-category. Abstractly, these tensor C$^*$-categories are referred to as the \emph{Temperley-Lieb C$^*$-categories}.

Let now $\Gamma = (\Gamma,w,\sgn,i)$ be a $-2_q$-reciprocal random walk. Define $\Hsp^{\Gamma}$ as the $\Gamma^{(0)}$-bigraded Hilbert space $l^2(\Gamma^{(1)})$, where the $\Gamma^{(0)}$-bigrading is given by \[\delta_e \in \Gru{\Hsp^{\Gamma}}{s(e)}{t(e)}\] for the obvious Dirac functions. Note that, because $\Gamma$ has finite degree, $\Hsp^{\Gamma}$ is \emph{row- and column finite-dimensional} (rcfd), i.e.~ $\oplus_{v\in \Gamma^{(0)}} \Gru{\Hsp^{\Gamma}}{v}{w}$ (resp.~ $\oplus_{w\in \Gamma^{(0)}} \Gru{\Hsp^{\Gamma}}{v}{w}$) is finite-dimensional for all $w$ (resp.~ all $v$). 

Consider now $R_{\Gamma}$ as the (bounded) map \[R_{\Gamma}:l^2(\Gamma^{(0)})\rightarrow \Hsp^{\Gamma}\underset{\Gamma^{(0)}}{\otimes} \Hsp^{\Gamma}\] given by \begin{eqnarray*} R_{\Gamma} \delta_v &=& \sum_{e,s(e) = v} \sgn(e)\sqrt{w(e)}\delta_e \otimes \delta_{\bar{e}}.\end{eqnarray*} Then $R_{\Gamma}^*R_{\Gamma} = |q|+|q|^{-1}$ and \[(R_{\Gamma}^*\underset{\Gamma^{(0)}}{\otimes} \id_{\Hsp^{\Gamma}})(\id_{\Hsp^{\Gamma}}\underset{\Gamma^{(0)}}{\otimes} R_{\Gamma}) = -\sgn(q)\id.\]



Hence, by the universal property of $\Rep(SU_q(2))$ (\cite[Theorem 1.4]{DCY1}, based on \cite{Tur1,EtO1,Yam1,Pin2,Pin3}), we have a strongly monoidal $^*$-functor
\begin{equation}\label{EqForget} F_{\Gamma}: \Rep(SU_q(2)) \rightarrow {}^{\Gamma^{(0)}}\Hilb_{\rcf}^{\Gamma^{(0)}}\end{equation} into the tensor C$^*$-category of rcfd $\Gamma^{(0)}$-bigraded Hilbert spaces such that $F_{\Gamma}(\pi_{1/2}) = \Hsp_{\Gamma}$ and $F_{\Gamma}(\mathscr{R}) = R_{\Gamma}$, with \[(\pi_{1/2},\mathscr{R},-\sgn(q)\mathscr{R})\] a solution for the conjugate equations for $\pi_{1/2}$. Up to equivalence, $F_{\Gamma}$ only depends upon the isomorphism class of $(\Gamma,w)$, and is independent of the chosen involution or sign structure. Conversely, any strong monoidal $^*$-functor from $\Rep(SU_q(2))$ into $\Gr{\Hilb}{I}{I}{}{\rcf}$ for some set $I$ arises in this way \cite{DCY2}.

\subsection{Universal orthogonal partial compact quantum groups}


It follows from the previous subsection and the
Tannaka-Kre$\breve{\textrm{\i}}$n-Woronowicz in Theorem \ref{TheoTKPCQG} that for each reciprocal random walk on a graph $\Gamma$, one obtains a $\Gamma^{(0)}$-partial compact quantum group $\mathscr{G}$, and conversely every partial compact quantum group $\mathscr{G}$ with the fusion rules of $SU(2)$ arises in this way. Our first aim is to give a direct representation of the associated algebras $A(\Gamma) = P(\mathscr{G})$ by generators and relations. We will write $\Gamma_{vw}\subseteq \Gamma^{(1)}$ for the set of edges with source $v$ and target $w$.

\begin{Theorem}\label{TheoGenRel} Let $0<|q|\leq 1$, and let $\Gamma = (\Gamma,w,\sgn,i)$ be a $-2_q$-reciprocal random walk. Let $A(\Gamma)$ be the total $^*$-algebra associated to the $\Gamma^{(0)}$-partial compact quantum group constructed from the fiber functor $F_{\Gamma}$ as in \eqref{EqForget}. Then $A(\Gamma)$ is the universal $^*$-algebra generated by a copy of the $^*$-algebra of finitely supported functions on $\Gamma^{(0)}\times \Gamma^{(0)}$ (with the Dirac functions written as $\UnitC{v}{w}$) and elements $(u_{e,f})_{e,f\in \Gamma^{(1)}}$ where $u_{e,f}\in \Gr{A(\Gamma)}{s(e)}{t(e)}{s(f)}{t(f)}$ and 
\begin{eqnarray} 
\label{EqUni1}\sum_{v\in \Gamma^{(0)}}\sum_{g\in \Gamma_{vw}} u_{g,e}^*u_{g,f} = \delta_{e,f}\mathbf{1}\Grru{w}{t(e)}, \qquad \forall w\in \Gamma^{(0)}, e,f\in \Gamma^{(1)},\\ 
\label{EqUni2}\sum_{w\in \Gamma^{(0)}} \sum_{g\in \Gamma_{vw}} u_{e,g}u_{f,g}^* = \delta_{e,f} \mathbf{1}\Grru{s(e)}{v}\qquad \forall v\in \Gamma{(0)}, e,f\in \Gamma^{(1)},\\ 
\label{EqInt}u_{e,f}^* \;=\; \sgn(e)\sgn(f)\sqrt{\frac{w(f)}{w(e)}} u_{\bar{e},\bar{f}},\qquad \forall e,f\in \Gamma^{(1)}.
\end{eqnarray}

If moreover $v,w\in \Gamma^{(0)}$ and $e,f\in \Gamma^{(1)}$, we have \[\Delta_{vw}(u_{e,f}) = \underset{t(g) = w}{\sum_{s(g) = v}} u_{e,g}\otimes u_{g,f},\]
\[\varepsilon(u_{e,f}) = \delta_{e,f}\] and \[S(u_{e,f}) = u_{f,e}^*.\] 
\end{Theorem} 

Note that the sums in \eqref{EqUni1} and \eqref{EqUni2} are in fact finite, as $\Gamma$ has finite degree. 

\begin{proof} Let $(\Hsp,V)$ be the generating unitary corepresentation of $A(\Gamma)$ on $\Hsp = l^2(\Gamma^{(1)})$. Then $V$ decomposes into parts \[ \Gr{V}{k}{l}{m}{n} = \sum_{e,f} v_{e,f} \otimes e_{e,f} \in \Gr{A}{k}{l}{m}{n}\otimes B(\Gru{\Hsp}{m}{n},\Gru{\Hsp}{k}{l}),\] where the $e_{e,f}$ are elementary matrix coefficients and with the sum over all $e$ with $s(e)=k,t(e)=l$ and all $f$ with $s(f) = m, t(f)=n$. By construction $V$ defines a unitary corepresentation of $A(\Gamma)$, hence the relations \eqref{EqUni1} and \eqref{EqUni2} are satisfied for the $v_{e,f}$. Now as $R_{\Gamma}$ is an intertwiner between the trivial representation on $\C^{(\Gamma^{(0)})} = \oplus_{v\in \Gamma^{(0)}} \C$ and $V\Circtv{\Gamma^{(0)}} V$, we have for all $v\in \Gamma^{(0)}$ that \begin{equation}\label{EqMorR}\underset{t(f)=s(h),t(e)=s(g)}{\sum_{e,f,g,h\in \Gamma^{(1)}}} v_{e,f}v_{g,h}\otimes \left((e_{e,f}\otimes e_{g,h})\circ R_{\Gamma} \delta_v\right) = \sum_w \UnitC{w}{v}\otimes R_{\Gamma}\delta_v,\end{equation} hence
\[\underset{t(e)=s(g),s(k)=v}{\sum_{e,g,k}} \sgn(k)\sqrt{w(k)}\left( v_{e,k}v_{g,\bar{k}} \otimes \delta_e\otimes \delta_{g}\right) =\underset{s(k)=w}{\sum_{w,k}}\sgn(k)\sqrt{w(k)} \left(\UnitC{w}{v} \otimes \delta_k\otimes \delta_{\bar{k}}\right).\] Hence if $t(e) = s(g)=z$, we have \[\sum_{k,s(k)=v} \sgn(k)\sqrt{w(k)} v_{e,k}v_{g,\bar{k}} =  \delta_{e,\bar{g}} \sgn(e)\sqrt{w(e)}\UnitC{s(e)}{v}.\] Multiplying to the left with $v_{e,l}^*$ and summing over all $e$ with $t(e) = z$, we see from \eqref{EqUni1} that also relation \eqref{EqInt} is satisfied. Hence the $v_{e,f}$ satisfy the universal relations in the statement of the theorem. The formulas for comultiplication, counit and antipode then follow immediately from the fact that $V$ is a unitary corepresentation.

Let us now a priori denote by $B(\Gamma)$ the $^*$-algebra determined by the relations \eqref{EqUni1},\eqref{EqUni2} and \eqref{EqInt} above, and write $\mathscr{B}(\Gamma)$ for the associated $\Gamma^{(0)}\times \Gamma^{(0)}$-partial $^*$-algebra induced by the local units $\UnitC{v}{w}$. Write $\Delta(\UnitC{v}{w}) = \sum_{z\in \Gamma^{(0)}} \UnitC{v}{z}\otimes \UnitC{z}{w}$ and \[\Delta(u_{e,f}) = \sum_{g\in \Gamma^{(1)}} u_{e,g}\otimes u_{g,f},\] which makes sense in $M(B(\Gamma)\otimes B(\Gamma))$ as the degree of $\Gamma$ is finite. Then we compute for $w\in \Gamma^{(0)}$ and $e,f\in \Gamma^{(1)}$ that \begin{eqnarray*} \sum_{v\in \Gamma^{(0)}}\sum_{g\in \Gamma_{vw}}\Delta(u_{g,e})^*\Delta(u_{g,f}) &=& \sum_{v\in \Gamma^{(0)}}\sum_{g\in \Gamma_{vw}} \sum_{h,k\in \Gamma^{(1)}} u_{g,h}^*u_{g,k}\otimes u_{h,e}^*u_{k,f}\\ &=& \sum_{h,k\in \Gamma^{(1)}} \delta_{h,k} \UnitC{w}{t(h)}\otimes u_{h,e}^*u_{k,f}\\ &=&  \sum_{z\in \Gamma^{(0)}}\underset{t(h)=z}{\sum_{h\in \Gamma^{(1)}}} \UnitC{w}{z}\otimes u_{h,e}^*u_{h,f} \\ &=& \delta_{e,f} \sum_{z\in \Gamma^{(0)}} \UnitC{w}{z}\otimes \UnitC{z}{t(e)}\\ &=& \delta_{e,f} \Delta(\UnitC{w}{t(e)}).\end{eqnarray*}  Similarly, the analogue of \eqref{EqUni2} holds for $\Delta(u_{e,f})$. As also \eqref{EqInt} holds trivially for $\Delta(u_{e,f})$, it follows that we can define a $^*$-algebra homomorphism \[\Delta:B(\Gamma)\rightarrow M(B(\Gamma)\otimes B(\Gamma))\] sending $u_{e,f}$ to $\Delta(u_{e,f})$ and $\UnitC{v}{w}$ to $\Delta(\UnitC{v}{w})$. Cutting down, we obtain maps \[\Delta_{vw}:\Gr{B(\Gamma)}{r}{s}{t}{z}\rightarrow \Gr{B(\Gamma)}{r}{s}{v}{w}\otimes \Gr{B(\Gamma)}{v}{w}{t}{z}\] which then satisfy the properties \ref{Propa}, \ref{Propd} and \ref{Prope} of Definition \ref{DefPartBiAlg}. Moreover, the $\Delta_{vw}$ are coassociative as they are coassociative on generators.

Let now $e_{v,w}$ be the matrix units for $l^2(\Gamma^{(0)})$. Then one verifies again directly from the defining relations of $B(\Gamma)$ that one can define a $^*$-homomorphism \[\widetilde{\varepsilon}: B(\Gamma)\rightarrow B(l^2(\Gamma^{(0)})),\quad \left\{\begin{array}{lll} \UnitC{v}{w}&\mapsto &\delta_{v,w}\, e_{v,v}\\ u_{e,f}&\mapsto& \delta_{e,f}\, e_{s(e),t(e)}\end{array}\right.\] We can hence define a map $\varepsilon: B(\Gamma)\rightarrow \C$ such that \[\widetilde{\varepsilon}(x) = \varepsilon(x) e_{v,w},\qquad  \forall x\in \Gr{B(\Gamma)}{v}{w}{v}{w},\] and which is zero elsewhere. Clearly it satisfies the conditions \ref{Propb} and \ref{Propc} of Definition \ref{DefPartBiAlg}. As $\varepsilon$ satisfies the counit condition on generators, it follows by partial multiplicativity that it satisfies the counit condition on the whole of $B(\Gamma)$, i.e.~ $B(\Gamma)$ is a partial $^*$-bialgebra. 

It is clear now that the $u_{e,f}$ define a unitary corepresentation $U$ of $B(\Gamma)$ on $\Hsp^{\Gamma}$. Moreover, from \eqref{EqUni1} and \eqref{EqInt} we can deduce that $R_{\Gamma}: \C_{\Gamma^{(0)}}\rightarrow \Hsp^{\Gamma}\underset{\Gamma^{(0)}}{\otimes}\Hsp^{\Gamma}$ is a morphism from $\C^{(\Gamma^{(0)})}$ to $U\Circtv{\Gamma^{(0)}} U$ in $\Corep_{\rcf,u}(\mathscr{B}(\Gamma))$, cf.~ \eqref{EqMorR}. From the universal property of $\Rep(SU_q(2))$, it then follows that we have a (unique and faithful) strongly monoidal $^*$-functor \[G^{\Gamma}: \Rep(SU_q(2)) \rightarrow \Corep_{\rcf,u}(\mathscr{B}(\Gamma))\] such that $G^{\Gamma}(\pi_{1/2}) = U$. On the other hand, as we have a $\Delta$-preserving $^*$-homomorphism $B(\Gamma)\rightarrow A(\Gamma)$ by the universal property of $\mathscr{B}(\Gamma)$, we have a strongly monoidal $^*$-functor $H^{\Gamma}:  \Corep_{\rcf,u}(\mathscr{B}(\Gamma))\rightarrow \Corep_u(\mathscr{A}(\Gamma)) = \Rep(SU_q(2))$ which is inverse to $G^{\Gamma}$. Then since the commutation relations of $\mathscr{A}(\Gamma)$ are completely determined by the morphism spaces of $\Rep(SU_q(2))$, it follows that we have a $^*$-homomorphism $\mathscr{A}(\Gamma)\rightarrow \mathscr{B}(\Gamma)$ sending $v_{e,f}$ to $u_{e,f}$. This proves the theorem. 
\end{proof}

\subsection{Dynamical quantum $SU(2)$ from the Podle\'{s} graph}


Let us now fix a $-2_q$-reciprocal random walk, and assume further
that there exists a finite set $T$ partitioning $\Gamma^{(1)} = \cup_a
\Gamma^{(1)}_a$ such that for each $a\in T$ and $v\in \Gamma^{(0)}$,
there exists a unique $e_a(v)\in \Gamma^{(1)}_a$ with source $v$. Write $av$ for the range of $e_a(v)$. Assume moreover that $T$ has an involution $a\mapsto \bar{a}$ such that $\overline{e_a(v)} = e_{\bar{a}}(av)$. Then for each $a$, the map $v\mapsto av$ is a bijection on $\Gamma^{(0)}$ with inverse $v\mapsto \bar{a}v$. In particular, also for each $w\in \Gamma^{(0)}$ there exists a unique $f_w(a) \in \Gamma^{(1)}_a$ with target $w$.

Let us further write $w_a(v) = w(e_a(v))$ and $\sgn_a(v) =
\sgn(e_a(v))$. Let $A(\Gamma)$ be the total $^*$-algebra of the
associated partial compact quantum group. Using Theorem
\ref{TheoGenRel}, we have the following presentation of
$A(\Gamma)$. Let $B$ be the $^*$-algebra of finitely supported
functions on $\Gamma^{(0)}\times \Gamma^{(0)}$, whose Dirac functions
we write as $\UnitC{v}{w}$. Then $A(\Gamma)$ is generated by a copy of
$B$ and elements \[(u_{a,b})_{v,w} := u_{e_a(v),e_b(v)} \in
\Gr{A(\Gamma)}{v}{av}{w}{bw}\] for $a,b\in T$ and $v,w\in
\Gamma^{(0)}$ with defining relations \begin{eqnarray*} \sum_{a\in T}
  (u_{a,b})_{\bar{a}v,w}^* (u_{a,c})_{\bar{a}v,z}&=& \delta_{w,z}
  \delta_{b,c} \UnitC{v}{bw},\\ \sum_{a\in T} (u_{b,a})_{w,v}
  (u_{c,a})_{z,v}^* &=& \delta_{b,c}\delta_{w,z} \UnitC{w}{v}\\
  (u_{a,b})_{v,w}^* &=&
  \frac{\sgn_b(w)\sqrt{w_b(w)}}{\sgn_a(v)\sqrt{w_a(v)}}(u_{\bar{a},\bar{b}})_{av,bw}.\end{eqnarray*}

Let us now consider $M(A(\Gamma))$, the multiplier algebra of $A(\Gamma)$. For a function $f$ on $\Gamma^{(0)}\times \Gamma^{(0)}$, write $f(\lambda,\rho) = \sum_{v,w} f(v,w)\UnitC{v}{w} \in M(A(\Gamma))$. Similarly, for a function $f$ on $\Gamma^{(0)}$ we write $f(\lambda) = \sum_{v,w} f(v)\UnitC{v}{w}$ and $f(\rho) = \sum_{v,w}f(w)\UnitC{v}{w}$. We then write for example $f(a\lambda,\rho)$ for the element corresponding to the function $(v,w)\mapsto f(av,w)$.

We can further form in $M(A(\Gamma))$ the elements $u_{a,b} =
\sum_{v,w} (u_{a,b})_{v,w}$. Then $u=(u_{a,b})$ is a unitary
m$\times$m matrix for
$m=\#T$. Moreover, \begin{equation}\label{EqAdju}u_{a,b}^* =
  u_{\bar{a},\bar{b}}\frac{\gamma_b(\rho)}{\gamma_a(\lambda)},\end{equation}
where $\gamma_a(v) = \sgn_a(v)\sqrt{w_a(v)}$.   We then have the
following commutation relations between functions on
$\Gamma^{(0)}\times \Gamma^{(0)}$ and the entries of
$u$: \begin{equation}\label{EqGradu} f(\lambda,\rho)u_{a,b} =
  u_{a,b}f(\bar{a}\lambda,\bar{b}\rho),\end{equation} where
$f(\bar{a}\lambda,\bar{b}\rho)$ is given by $(v,w) \mapsto f(\bar{a}v,\bar{b}w)$.
 The coproduct is given by
$\Delta(u_{a,b}) = \Delta(1) \sum_c(u_{a,c}\otimes u_{c,b})$. Note
that the $^*$-algebra generated by the $u_{a,b}$ is no longer a weak
Hopf $^*$-algebra when $\Gamma^{(0)}$ is infinite, but rather one can
turn it into a Hopf algebroid. 

\begin{Rem}
  The weak multiplier Hopf algebra $A(\Gamma)$ is related to the free
  orthogonal dynamical quantum groups introduced in
  \cite{timmermann:free} as follows.  Denote by $G$ the free group
  generated by the elements of $T$ subject to the relation
  $\bar{a}=a^{-1}$ for all $a\in T$. By assumption on $\Gamma$, the
  formula $(af)(v):=f(\bar{a}v)$ defines a left action of $G$ on
  $\Fun(\Gamma^{(0)})$. Denote by $C\subseteq \Fun(\Gamma^{(0)})$ the
  unital subalgebra generated by all $\gamma_{a}$ and their inverses
  and translates under $G$,  write the
  elements of $T \subseteq G$ as a tuple in the form
  $\nabla=(a_{1},\bar{a_{1}},\ldots,a_{n},\bar{a_{n}})$, and define a
  $\nabla\times\nabla$ matrix $F$ with values in $C$ by $F_{a,b} :=
  \delta_{b,\bar{a}} \gamma_{a}$.  Then the free orthogonal dynamical
  quantum group $A_{\mathrm{o}}^{C}(\nabla,F,F)$ introduced in
  \cite{timmermann:free} is the universal unital $*$-algebra generated
  by a copy of $C\otimes C$ and the entries of a unitary $\nabla\times\nabla$-matrix
  $v=(v_{a,b})$ satisfying
  \begin{align*}
    v_{a,b}(f \otimes g) &= (af\otimes bg) v_{a,b}, &
    (aF_{a,\bar{a}}\otimes 1)v_{\bar{a},\bar{b}}^{*} &=
    v_{a,b}(1\otimes F_{b,\bar{b}})
  \end{align*}
  for all $f,g\in C$ and $a,b\in \nabla$. The second equation
  can be rewritten in the form
  $v_{\bar{a},\bar{b}}^{*}=v_{a,b}(\gamma_{a}^{-1} \otimes
  \gamma_{b})$.   Comparing with
  \eqref{EqAdju} and \eqref{EqGradu}, we see that there exists a
  $*$-homomorphism
  \begin{align*}
  A^{C}_{\mathrm{o}}(\nabla,F,F)  \to
  M(A(\Gamma)), \quad
  \begin{cases}
    f\otimes g&
\mapsto f(\lambda)g(\rho), \\
    v_{a,b} &\mapsto u_{\bar{a},\bar{b}}.
  \end{cases}
\end{align*}
 The two quantum groupoids are related by an
analogue of the unital base changes considered for dynamical quantum
groups in \cite[Proposition 2.1.12]{timmermann:free}. Indeed, Theorem
\ref{TheoGenRel} shows that $A(\Gamma)$ is the image of
$A^{C}_{\mathrm{o}}(\nabla,F,F)$ under a non-unital base change from
$C$ to $\Fun_{f}(\Gamma^{(0)})$ along the natural map $C \to
M(\Fun_{f}(\Gamma^{(0)}))$.

\end{Rem}
\begin{Exa}
  As a particular example, consider the Podle\'{s} graph of Example
  \ref{ExaGraphPod} at parameter $x\in \lbrack
  0,\frac{1}{2}\rbrack$. Then one can take $T = \{+,-\}$ with the
  non-trivial involution, and label the edges $(k,k+1)$ with $+$ and
  the edges $(k+1,k)$ with $-$. Let us write \[F(k) = |q|^{-1}w_+(k) =
  |q|^{-1}\frac{|q|^{x+k+1}+|q|^{-x-k-1}}{|q|^{x+k}+|q|^{-x-k}},\] and
  further put\[\alpha =
  \frac{F^{1/2}(\rho-1)}{F^{1/2}(\lambda-1)}u_{--},\qquad \beta =
  \frac{1}{F^{1/2}(\lambda-1)}u_{-+}.\] Then the unitarity of
  $(u_{\epsilon,\nu})_{\epsilon,\nu}$ together with \eqref{EqAdju} and
  \eqref{EqGradu} are equivalent to the commutation
  relations \begin{equation}\label{EqqCom} \alpha \beta =
    qF(\rho-1)\beta\alpha \qquad \alpha\beta^* =
    qF(\lambda)\beta^*\alpha\end{equation} \begin{equation}\label{EqDet}
    \alpha\alpha^* +F(\lambda)\beta^*\beta = 1,\qquad
    \alpha^*\alpha+q^{-2}F(\rho-1)^{-1}\beta^*\beta =
    1,\end{equation}\begin{equation*} F(\rho-1)^{-1}\alpha\alpha^*
    +\beta\beta^* = F(\lambda-1)^{-1},\qquad F(\lambda)\alpha^*\alpha
    +q^{-2}\beta\beta^* =
    F(\rho),\end{equation*} \begin{equation}\label{EqGrad}
    f(\lambda)g(\rho)\alpha = \alpha f(\lambda+1)g(\rho+1),\qquad
    f(\lambda)g(\rho)\beta = \beta
    f(\lambda+1)g(\rho-1).\end{equation}

  These are precisely the commutation relations for the dynamical
  quantum $SU(2)$-group as in \cite[Definition 2.6]{KoR1}, except that
  the precise value of $F$ has been changed by a shift in the
  parameter domain by a complex constant. The (total) coproduct on
  $A_x$ also agrees with the one on the dynamical quantum
  $SU(2)$-group, namely \begin{eqnarray*} \Delta(\alpha) &=& \Delta(1)
    (\alpha\otimes \alpha - q^{-1}\beta\otimes \beta^*),\\
    \Delta(\beta) &=& \Delta(1)(\beta\otimes \alpha^* +\alpha\otimes
    \beta)\end{eqnarray*} where $\Delta(1) = \sum_{k\in \Z}
  \rho_k\otimes \lambda_k$.
\end{Exa}

\bibliographystyle{habbrv}

\end{document}